\documentclass[reqno]{amsart}
\usepackage{amsthm, amscd, amsfonts,graphicx}
\usepackage{color}
\usepackage{enumerate}
\usepackage{verbatim}
\numberwithin{equation}{section}
\theoremstyle{plain}
 \newtheorem{theorem}{Theorem}[section]
 \newtheorem{lemma}[theorem]{Lemma}

 \newtheorem{corollary}[theorem]{Corollary}
\theoremstyle{definition}
 \newtheorem{definition}[theorem]{Definition}
 \newtheorem{convention}[theorem]{Convention}
 \newtheorem{remark}[theorem]{Remark}
 \newtheorem{algorithm}[theorem]{Algorithm}
 
\theoremstyle{remark}
 
 \newtheorem{example}[theorem]{Example}

\newenvironment{enumeratei}{\begin{enumerate}[\quad\upshape (i)]} {\end{enumerate}}
\newcommand \dom [1] {\textup{Dom}(#1)}
\newcommand \many{\boldsymbol\sigma}
\newcommand \tmany {$\many$}
\newcommand \smany{$\many$-many}

\newcommand \azeset[1] {(C#1)}
\newcommand \eset[1] {\par \azeset{#1}}
\DeclareMathOperator \Sub {Sub}
\newcommand \poset[1]{(#1;\leq)}
\newcommand \mslat[1]{(#1;\wedge)}
\newcommand \slat[1]{(#1;\vee)}
\newcommand \nlat[1]{(#1;\leq,\vee,\wedge)}
\newcommand \ppnlat[1]{(#1;\leq,\vee',\wedge')}
\newcommand \snt[1] {\textcircled{#1}}
\newcommand \SUT[1] {\underline{\textcircled{{#1}}}}
\newcommand \ed {\widehat{\textup{e}}}
\newcommand \RR {{\mathbb R}}
\renewcommand \epsilon {\varepsilon}
\newcommand \nonparallel {\mathrel{\not\kern-1.5pt{\mathop\parallel}}}
\newcommand \qn {qn}
\newcommand \ce{{L_{\textup{ce}}}}
\newcommand \amin{\tilde a}
\newcommand \bmin{\tilde b}
\newcommand \cmin{\tilde c}
\newcommand \dmin{\tilde d}
\newcommand \xmin{\tilde x}
\newcommand \ymin{\tilde y}
\newcommand \zmin{\tilde z}
\newcommand \lmoon{C_{\textup{left}}}
\newcommand \rmoon{C_{\textup{right}}}
\newcommand \eeqref[1]{\overset{\eqref{#1}}{=}}
\newcommand \sgeeqref[1]{\overset{\eqref{#1}}{>}}
\newcommand \lshape {\textup{LSh}}
\newcommand \lat[1]{(#1;\vee,\wedge)}

\newcommand \nnul{\mathbb N_0}

\DeclareMathOperator \Con {Con}
\newcommand \ideal {\mathord{\downarrow}}
\newcommand \filter {\mathord{\uparrow}}
\newcommand\set [1]{\{#1\}}

\newcommand \red [1] {{\color{red}#1\color{black}}}

\newcommand \nothing [1] {}

\begin{document}
\title
[Planar semilattices and nearlattices]
{Planar semilattices and nearlattices with eighty-three subnearlattices}

\author[G.\ Cz\'edli]{G\'abor Cz\'edli}
\address{University of Szeged, Bolyai Institute, Szeged,
Aradi v\'ertan\'uk tere 1, Hungary 6720}
\email{czedli@math.u-szeged.hu}
\urladdr{http://www.math.u-szeged.hu/~czedli/}

\thanks{This research was supported by the Hungarian Research, Development and Innovation Office under grant number KH 126581.\hskip5cm\red{August 22, 2019}}

\subjclass{06A12, 06B75, 20M10}

\keywords{Planar nearlattice, planar semilattice, planar lattice, chopped lattice, 
number of subalgebras, computer-assisted proof, commutative idempotent semigroup}

\dedicatory{Dedicated to professor George A.\ Gr\"atzer on his eighty-third birthday}

\begin{abstract} Finite (upper) nearlattices are essentially the same mathematical entities as finite semilattices,  finite commutative idempotent semigroups,  finite  join-enriched meet semilattices, and chopped lattices.
We prove that if an $n$-element nearlattice has at least $83\cdot 2^{n-8}$ subnearlattices, then it has a planar Hasse diagram. For $n>8$, this result is sharp.
\end{abstract}

\maketitle

\section{Result  and introduction} 
At first, after few definitions, we formulate our main (and only) result; historical and other comments and  an outline of the paper will be given thereafter.

\begin{definition}\label{defnearlat}
Let $\slat L$ be a \emph{finite} $n$-element join-semilattice. 
Its natural ordering  is defined by $x\leq y\iff x\vee y=y$. For $x,y\in L$, let $x\wedge y$ be the infimum of $\set{x,y}$ provided it exists. If this infimum does not exist, then $x\wedge y$ is undefined. The structure $\nlat L$ is called a finite (\emph{upper}) \emph{nearlattice}; note that this nearlattice and the join-semilattice $\slat L$ mutually determine each other. Apart from a short historical survey, the adjective ``upper'' will always be dropped.
\end{definition}
The adjective ``finite'' will usually be dropped but understood. A nearlattice $\nlat L$ or, equivalently, the corresponding join-semilattice $\slat L$ is \emph{planar} if the poset (also known as partially ordered set) $\poset L$  is planar; that is, if $\poset L$ has a Hasse diagram that is also a planar representation of a graph. 
A nonempty subset of $L$  closed with respect to (the total operation) join and (the partial operation) meet is called a \emph{subnearlattice}. Our goal is to prove the following theorem.

\begin{theorem}[Main Theorem]\label{thmmain} Let $\nlat L$ be a finite nearlattice, and let  $n:=|L|$ denote the number of its elements. If $\nlat L$ has at least $83\cdot 2^{n-8}$ subnearlattices, then it is planar. For $n\geq 9$, this statement is sharp since there exists an $n$-element non-planar nearlattice with exactly  $83\cdot 2^{n-8}-1$ subnearlattices.
\end{theorem}

For another and equivalent variant of the Main Theorem, see Theorem~\ref{thmreform} later. 

\begin{remark}\label{remmsVnpL} Every nearlattice with  \emph{at most seven elements} is planar, regardless the number of its subnearlattices. 
While the eight-element non-planar boolean lattice, as 
a nearlattice, has $73$ subnearlattices, 
every \emph{eight-element} nearlattice with at least $74=74\cdot 2^{8-8}$ subnearlattices is planar. 
\end{remark}

\subsection*{Outline} The rest of this introductory section consists of four subsection, namely:  ``Outline'' (the present subsection),  ``Motivation and historical comments'',
``Notes on the proof'', and ``Notes on the dedication''. After the present section, apart from two-thirds of a page to prove Remark~\ref{remmsVnpL},  
the rest of the paper is devoted to the proof of Theorem~\ref{thmmain}. 
In particular, Section~\ref{sectotherform} contains Theorem~\ref{thmreform}, which is a useful reformulation of Theorem~\ref{thmmain}. Section~\ref{sectgeom} is a short section formulating some statements of 
geometrical nature on planar nearlattice diagrams.  Section~\ref{sectionJMc} defines \qn-lattices, which are
certain substructures of nearlattices. They are only technical tools, and the section contains some lemmas to make it clear that the lion's share of the proof of the main result relies on \qn-lattices. Section~\ref{sectexsomeqn} consists of a series of lemmas to exclude some small \qn-lattices as substructures of a minimal counterexample of the Main Theorem, while Sections~\ref{sectaminbmin} and  \ref{sectanchors} exclude further \qn-lattices as substructures with special stipulations.
Note that after Section~\ref{sectionJMc}, the reader may decide to jump immediately to Section~\ref{sectlemmasatwork} at first reading in order to see how
to benefit from the lemmas of  Sections~\ref{sectexsomeqn}, \ref{sectaminbmin}, and \ref{sectanchors}  rather than checking their proofs.
Most of these proofs rely on a humanly impossible amount of computation done by a computer in the background, but lots of theoretical arguments are also needed and they are presented in a readable form. Section~\ref{sectlemmasatwork} completes the proof of Theorem~\ref{thmmain}.
Section~\ref{sectappndf} is an appendix to describe how to use our freely availably computer program, which is outlined in Section~\ref{sectionJMc}. Also, this section contains a short sample input file, which was used at one of our lemmas. There is a second appendix in the 
extended version\footnote{\red{This is the extended version.}} 
of the paper that contains \emph{all} output files; that version is available from the author's website or from \texttt{www.arxiv.org} .

\subsection*{Motivation and historical comments}
Our result is motivated by similar or analogous results
about  lattices and semilattices with many congruences, sublattices, and  subsemilattices; see 
Ahmed and Horv\'ath~\cite{delbrineszter},
Cz\'edli~\cite{czgnotelatmanyC}, \cite{czgslatmanc}, \cite{czglatmancplanar}, \cite{czg83}, and \cite{czg127},    Cz\'edli and Horv\'ath~\cite{czgkhe}, and Mure\c san and Kulin~\cite{kulinmuresan}. Below, for later reference, two of the motivating results are mentioned, both are sharp.

\begin{theorem}[Cz\'edli~\cite{czg83}]\label{thmlat83} If a finite lattice $\lat L$ has at least $83\cdot 2^{|L|-8}$ sublattices, then it is a planar lattice.  
\end{theorem}

Clearly, Theorem~\ref{thmmain} generalizes the above result. 

\begin{theorem}[Cz\'edli~\cite{czg127}]\label{thmslat127} If an $n$-element  join-semilattice   has at least $127\cdot 2^{n-8}$ subsemilattices, then it is planar.  
\end{theorem}

This theorem gives a sufficient condition for a semilattice to be planar. Since a join-semilattice $\slat L$ and the corresponding nearlattice $\nlat L$ mutually determine each other, Theorem~\ref{thmmain} gives another sufficient condition.

Assuming finiteness, semilattices, nearlattices, join-enriched meet semilattices, join algebras, commutative idempotent semigroups, and chopped lattices are essentially the same mathematical entities modulo the duality principle. They have been studied from various aspects, and they have been discovered, studied, and baptized several times. These discoveries and re-discoveries seem not to be aware of each other; this is our excuse if the list of the earlier names of these structures is not complete.  

The concept of \emph{semilattices} is as old as that of lattices, so the above-mentioned entities occur frequently in mathematics. 

Our definition of (upper) \emph{nearlattices} is the same as the finite version of the concept of  nearlattices studied by 
Ara\'ujo and  Kinyon~\cite{araujokinyon2010},
Chajda and Hala\v s~\cite{chajdahalas2006}, 
Chajda and Kola\v{r}\'\i{}k~\cite{chajdakolarik2006a} and \cite{chajdakolarik2006b}, and Hala\v s~\cite{halas2006}.
Under a different name, as \cite{chajdahalas2006} points out, this concept appeared already in Sholander~\cite{sholander1952} and \cite{sholander1954}. 
The definition used in the above papers but Sholander's ones is the following, but the adjective ``upper'' is our suggestion: by a (not necessarily finite) \emph{upper nearlattice} we mean a  join-semilattice in which every principal filter is a lattice. Since our convention for the paper is that
\begin{equation}\left.
\parbox{7.8cm}{unless otherwise explicitly stated, every structure occurring in this paper is assumed to be finite even if this is not repeated all the times,}\right\}
\label{eqpbxmndvGs}
\end{equation} 
the condition on principal filters holds automatically in the scope of this convention. 
Hence, in the subsequent subsections and sections, 
 each of our join-semilattices and nearlattices is an (upper) nearlattice in Chajda at al's sense. It is a matter of taste and the actual situation whether one considers the meet as a partial operation in the definition of finite nearlattices; we do. Note that for a finite join-semilattice $\slat L$, each of  
\begin{equation}
\text{$\slat L$, $\nlat L$, $\poset L$, and the partial algebra $\mslat L$}
\label{eqtxtzkKvlsdFck}
\end{equation}
determines the other three.  Thus, no matter which one of the four structures listed in \eqref{eqtxtzkKvlsdFck} is given, we will also use the other three without further notice.

Finite nearlattices are in very close connection with lattices. First, finite lattices are exactly the nearlattices with smallest elements.  Second,
if we add a (possibly new) zero (that is, a least) element to a finite nearlattice $\nlat L$, then we obtain a finite \emph{lattice} $\nlat{L^{(+0)}}$. Conversely, if we start from a nonsingleton finite lattice $\nlat K$, then we obtain a nearlattice $\nlat{K^{(-0)}}$ by deleting its smallest element. Beginning with a finite join-semilattice $\slat L$, each of the lattice $\nlat{L^{(+0)}}$ and the structures in \eqref{eqtxtzkKvlsdFck} determines the other four.

Many authors, including 
C\={\i}rulis~\cite{cirulis} (who calls them \emph{join-enriched meet semilattices}),
Hickman~\cite{hickman1980} (who calls them \emph{join algebras}), 
Cornish and Noor~\cite{cornishnoor1982},
Nieminen~\cite{nieminen1986},
Noor and Rahman~\cite{noorrahman2002}, and 
Van Alten~\cite{vanalten1999} deal with meet-semilattices in which all principal ideals are lattices; that is,
they define \emph{lower nearlattices} as the duals of upper nearlattices; the adjective ``lower'' is our suggestion for the sake of distinction.  Clearly, our Theorem~\ref{thmmain} remains valid for lower nearlattices.

For the finite case, lower nearlattices appeared and were intensively studied in, say,
Gr\"atzer~\cite{ggproofpict}, 
Gr\"atzer, Lakser and Roddy~\cite{gglakserscht}, and
Gr\"atzer and Schmidt~\cite{ggscht95a}, \cite{ggscht95b} and \cite{ggscht99} under the name \emph{chopped lattices}. Gr\"atzer~\cite{ggeneral} notes that this concept goes back to G.\ Gr\"atzer and H.\ Lakser. With the help of chopped lattices, a lot of deep results have been proved for congruence lattices of lattices in the above-mentioned papers.

Note that (upper) nearlattices occur frequently, since the subalgebras of an  algebra $(A;F)$ form a nearlattice with respect to set inclusion; this nearlattice is not a lattice in general since the emptyset is not a subalgebra. In particular, the subnearlattices from Theorem~\ref{thmmain} also form a nearlattice, which is not a lattice.

Theorem~\ref{thmslat127} from \cite{czg127}, which is
closely related to Theorem~\ref{thmmain}, has been elaborated to \emph{join}-semilattices.
This explains that we will use upper nearlattices rather than lower ones. 
Now, at the end of this short historical survey, let us emphasize again that nearlattices in the rest of the paper are always \emph{finite} and they are understood according to Definition~\ref{defnearlat}.

Finally, we mention three additional ingredients of our motivation; hopefully, they are applicable for many algebraic structures, not only for lattices and their generalizations.  

First, it is quite natural to study general algebraic structures $(A;F)$ for which the size of the \emph{congruence lattice}, $|\Con(A;F)|$, or that of the subalgebra lattice, $|\Sub(A;F)|$, are small, because they are the building stones of other structures in some sense. For example, the description of non-singleton finite groups $(G;\cdot)$ with 
$|\Con(G;\cdot)|$ being as small as possible is probably the deepest mathematical result that has ever been proved; it is the classification of finite simple groups. Fields are typically constructed from prime fields, that is, from fields whose subfield lattices are singletons.  Once the smallest values of 
$|\Con(A;F)|$ and $|\Sub(A;F)|$ have been paid a lot of attention to, it seems reasonable to study also the largest values. 

Second, the papers mentioned right before Theorem~\ref{thmlat83} indicate that the study of large or the largest values of $|\Con(A;F)|$ and $|\Sub(A;F)|$ often leads to interesting results with nontrivial proofs and, sometimes, to structural descriptions. For example, while it seems to be hopeless to give a structural description of non-singleton finite lattices $\lat L$ with $|\Con\lat L|$ being the smallest or the second smallest possible number, even the $n$-element finite lattices $\lat L$ with $|\Con\lat L|$ being the third and fourth largest possible numbers have been structurally described in Ahmed and Horv\'ath~\cite{delbrineszter}. Roughly saying, while algebras $(A;F)$ with \emph{small} $|\Con(A;F)|$ or $|\Sub(A;F)|$ are the building stones, some of those with \emph{large} $|\Con(A;F)|$ or $|\Sub(A;F)|$ are nice buildings.

Third, it is generally a good idea to associate integer numbers with algebraic structures, like the numbers of their elements, congruences, and subalgebras, because these numbers might help in discovering relations between distinct fields of mathematics by the help of Sloan~\cite{sloan}.

\subsection*{Notes on the proofs} 
Although (the earlier) Theorem~\ref{thmlat83} is a particular case of  (our main) Theorem~\ref{thmmain}, these two theorems require different approaches. The proof of Theorem~\ref{thmlat83} in  Cz\'edli~\cite{czg83} was based on the powerful characterization of planar lattices given by Kelly and Rival~\cite{kellyrival}. Since no similar characterization of planar \emph{semilattices} is known at the time of this writing, the present paper is quite different from and more involved than \cite{czg83}. Note that
while the proof of the Kelly--Rival characterization relies heavily on the fact that every finite planar lattice contains an element with exactly one upper cover and one lower cover (a so-called \emph{doubly irreducible} element), the join-semilattice $\slat{T_2}$ in Figure~\ref{figd4} witnesses that a planar semilattice need not contain such an element. So, even if the future brings some characterization of finite planar join-semilattices, it will not be obtained from Kelly and Rival~\cite{kellyrival} by easy modifications.

Since  Theorem~\ref{thmlat83} takes care of the case when $\nlat L$ from Theorem~\ref{thmmain} happens to be a lattice, this paper deals only with the case when it is \emph{not} a lattice.

\subsection*{Notes on the dedication}
The number 83 plays a key role in Theorem~\ref{thmmain}, and 
at the time of submitting the first version\footnote{\red{This is the first version.}} of the paper,  professor George Gr\"atzer celebrated his 83-rd birthday. Furthermore, the topic of the present paper is close to his research interest; this is witnessed by, say, his papers
Cz\'edli and Gr\"atzer~ \cite{czgggchapter} and \cite{czgggresection},
Cz\'edli, Gr\"atzer, and Lakser~\cite{czggglakser},
Gr\"atzer \cite{ggnotes10}, \cite{ggsection14}, \cite{ggONczg}, \cite{ggswing15}, \cite{ggforkcon16}, and \cite{ggtrajcon18}, 
Gr\"atzer and Knapp~\cite{ggknapp1}, \cite{ggknapp2}, \cite{ggknappAU}, \cite{ggknapp3}, and \cite{ggknapp4}, Gr\"atzer and Lakser~\cite{gglakser92},  Gr\"atzer, Lakser, and Schmidt~\cite{gglakserscht}, Gr\"atzer and Quackenbush~\cite{ggqbush}, 
Gr\"atzer and Schmidt~\cite{ggscht14}, and  Gr\"atzer and Wares~\cite{ggwares} on \emph{planar lattices} and his already mentioned papers on \emph{chopped lattices}.

\section{Another form of our result}\label{sectotherform}
\subsection*{Relative number of subuniverses} For a nearlattice $\nlat L$, 
\begin{align}
\text{the \emph{domain} of $\wedge$ is }\dom\wedge:=\set{(x,y)\in L^2: x\wedge y \text{ is defined}},\text{ and}\cr
\Sub\nlat L:=\{X: X\subseteq L,\text{ } x\vee_L y\in X\text{ for every }(x,y)\in X^2,\text{ and }\cr
x\wedge_L y\in X\text{ for every }(x,y)\in X^2\cap\dom{\wedge_L}\}.
\label{eqSubdF}
\end{align}
Of course, the subscript $L$ above indicates that the join and meet have to be taken in $\nlat L$ rather than in $\poset X$ with the inherited ordering. So it may happen that a subposet $\poset X$ of $\poset L$ is a nearlattice (or even a lattice) on its own right (with respect to the ordering inherited from $\poset L$) but $X\notin\Sub\nlat L$. The members of $\Sub\nlat L$ are  the \emph{subuniverses} of $\nlat L$, while the nonempty members of $\Sub\nlat L$ are called the \emph{subnearlattices} of $\nlat L$. So, 
\begin{equation}
\parbox{6.5 cm}{$|\Sub\nlat L|$ is bigger than the number of subnearlattices of $\nlat L$ by 1.}
\label{pbxzhBrjplgnK}
\end{equation} 
The following concept and notation are taken from Cz\'edli~\cite{czg83} and \cite{czg127}; it will be more useful in our arguments than the number of subnearlattices.

\begin{definition}\label{defsigMa}
The \emph{relative number of subuniverses} of an $n$-element nearlattice $\nlat L$ is defined to be and denoted by
\[
\many\nlat L:=|\Sub\nlat L|\cdot 2^{8-n}.
\] 
Furthermore, we say that a finite nearlattice $\nlat L$ has  \emph{\smany{}  subuniverses} if $\many(L)>83$.
\end{definition}

\subsection*{An equivalent form of our result} By \eqref{pbxzhBrjplgnK},  Theorem~\ref{thmmain} is clearly equivalent to  the following equivalent theorem; it will be sufficient to prove the latter.

\begin{theorem}\label{thmreform}
If $\nlat L$ is a finite nearlattice with $\many\nlat L>83$, then $\nlat L$ is planar. In other words, finite nearlattices with \smany{} subuniverses are planar. Furthermore, for every natural number $n\geq 9$, there exists an $n$-element non-planar nearlattice $\nlat L$ such that $\many\nlat L=83$.
\end{theorem}

\begin{figure}[htb] 
\centerline
{\includegraphics[scale=1.0]{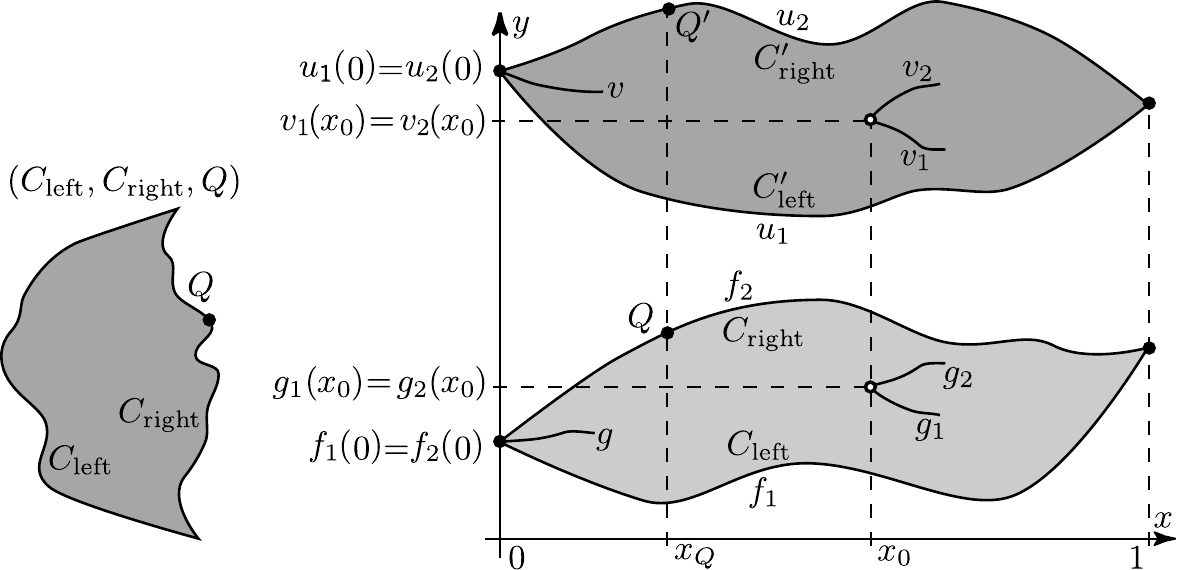}}
\caption{A pointed contour and two rotated pointed contours 
\label{figd1}}
\end{figure}%

\section{On the geometry of planar lattices}\label{sectgeom}
In this section, we make a distinction between (\emph{straight}) \emph{diagrams}, which are the usual Hasse diagrams of posets with straight edges, and \emph{curved diagrams}, which are poset diagrams in which curved edges are also allowed. 
Following Kelly and Rival~\cite{kellyrival}, by a curved edge  we mean a
set $\set{(f(y),y): a\leq y \leq b}$, where $a<b\in\RR$ and $f\colon [a,b]\to \RR$ is a differentiable function; this curved edge \emph{goes from the (initial) point $(f(a),a)$  to the (terminal) point $(f(b),b)$};  these two points are the \emph{endpoints of the curved edge}.  
Of course, at the endpoints $a$ and $b$ of the closed interval $[a,b]\subset \RR$, the differentiability is required only from the right and from the left, respectively.
Their differentiability ensures that curved edges keep going strictly upwards. Since the curved edges are the graphs of differentiable functions (but the role of the $x$-axis is interchanged with that of the $y$-axis), they have \emph{directional vectors} at each of their points; note that a directional vector is of length 1 by definition. In case of a curved edge, the directional vector is horizontal at none of its points.

\begin{definition}\label{defCrvdslwDg} By a \emph{curved diagram} of a finite poset $\poset P$ we mean a collection of curved edges and $|P|$ many vertices (that is, points) in the plane such that 
\begin{enumeratei}
\item \label{defCrvdslwDga} each element of $P$ is represented by exactly one vertex;
\item \label{defCrvdslwDgb} 
whenever two distinct curved edges intersect at a point (possibly at an endpoint), then they have distinct directional vectors at that point;
\item \label{defCrvdslwDgc} there exists a unique curved edge going from a vertex to another vertex if and only if the second vertex represents an element of $P$ that covers the element represented by the first vertex.
\end{enumeratei}
If, in addition,
\begin{enumeratei}\setcounter{enumi}{3}
\item \label{defCrvdslwDgd} no two distinct curved edges intersect except possibly at a common endpoint,
\end{enumeratei}
then the curved diagram is \emph{planar}.
\end{definition}
In a curved diagram, Definition~\ref{defCrvdslwDg}\eqref{defCrvdslwDgc} allows us to speak of the curved edge $a\prec b$ if $b$ covers $a$ in the poset $\poset P$. 
We know from Kelly~\cite{kelly} that 
\begin{equation}\left.
\parbox{7cm}{a poset has a planar curved diagram if and only if it has a (straight) planar diagram.}\right\}
\label{eqpbxDKlls}
\end{equation}
So, if a nearlattice has a curved planar diagram, then it is planar. In order to formulate a useful lemma, we need some additional concepts.
\begin{definition}\label{defmoonshape}
By a \emph{pointed contour} we mean a system $(\lmoon,\rmoon,Q)$ such that $\lmoon$ and $\rmoon$ are curved edges with common initial points and common terminal points, theses two endpoints are their only common points, $\rmoon$ is to the left of $\lmoon$, $\lmoon$ and $\rmoon$ have distinct directional vectors at the common initial point and also at the common terminal point, and $Q$ is an internal point of the curve $\rmoon$; see Figure~\ref{figd1}. The union $\lmoon\cup\rmoon$ is a closed Jordan curve; the union of this curve and its inside region will be called the \emph{L-shape} determined by the pointed contour; it is denoted by $\lshape(\lmoon,\rmoon,Q)$. So $\lmoon\cup\rmoon\subseteq \lshape(\lmoon,\rmoon,Q)$; in fact, $\lmoon\cup\rmoon$ is the boundary of this L-shape.
\end{definition}

Note that, by Definition~\ref{defCrvdslwDg}, the directional vector  of $\lmoon$ at a point is never vertical, and the same holds for $\rmoon$. If $e$ and $e'$ are distinct edges with a common initial point $X$ in a curved diagram, then  they have distinct directional vectors at $X$ and it makes sense to say that $e$ is to the left of $e'$ or conversely, depending on the directional vectors. The situation is analogous at a common terminal point.

\begin{definition}\label{defThszBjknMdNps}
Two pointed contours,   $(\lmoon,\rmoon,Q)$ and $(\lmoon',\rmoon',Q')$ are \emph{equivalent} if there exists a bijective transformation (that is, a map)
\begin{equation}
T\colon \lshape(\lmoon,\rmoon,Q)\to \lshape(\lmoon',\rmoon',Q')
\label{eqqHsTsrthsrf}
\end{equation} 
such that the following conditions hold: $T(Q)=Q'$, $T(\lmoon)=\lmoon'$, $T(\rmoon)=\rmoon'$, the $T$-image of every curved diagram $D$ in $\lshape(\lmoon,\rmoon,Q)$ is a curved diagram in $\lshape(\lmoon',\rmoon',Q')$, this $D$ is planar if and only if so is its $T$-image,
and whenever $e$ and $e'$ are distinct curved edges with a common endpoint in $D$ such that $e$ is to the left of $e'$, then $T(e)$ is to the left of $T(e')$.
\end{definition}

\begin{lemma}\label{lemmamnshps}
Any two pointed contours are equivalent.
\end{lemma}

\begin{proof} It is easy to see that the relation ``equivalent'' in the sense of Definition~\ref{defThszBjknMdNps} is an equivalence relation on the set of all pointed contours.
If, for all $(x,y)\in\RR^2$,   $T\colon (x,y)\mapsto (cx,cy)$ with a positive constant $c\in\RR$ or 
$T\colon (x,y)\mapsto (x+c,x+d)$ with constants $c,d\in \RR$,  that is, if $T$ is a positive homothety or a translation, then the $T$-image of  $(\lmoon,\rmoon,Q)$  is clearly equivalent to  $(\lmoon,\rmoon,Q)$. This allows us to assume that the $y$ coordinate of the bottom of  $(\lmoon,\rmoon,Q)$ and that of the bottom of $(\lmoon',\rmoon',Q')$ are  both 0,
and they are both 1 for the tops. Next, we are going to use some rudiments of Analysis; less than what is generally taught for undergraduates in the first semester. For convenience, we rotate our  pointed contours  counterclockwise by $90$ degrees, and in the rest of the proof, we will work with the rotated versions. 
Let the coordinates of $Q$ and $Q'$ be denoted by $(x_Q,y_Q)$ and $(x_Q',y_Q')$, respectively.

First, we deal with the case $x_Q=x'_Q$. As Figure~\ref{figd1} shows, we use the notation
$\lmoon=\set{(x,f_1(x)): x\in[0,1]}$, $\rmoon=\set{(x,f_2(x)): x\in[0,1]}$, $\lmoon'=\set{(x,u_1(x)): x\in[0,1]}$, 
$\rmoon'=\set{(x,u_2(x)): x\in[0,1]}$; here the $f_1$, $f_2$, $u_1$, and $u_2$  are differentiable functions on the closed interval $[0,1]\subset \RR$ and
\begin{align}
&f_1(x)<f_2(x)\text{ and }u_1(x)< u_2(x)\text{ for all }0<x<1,\label{alignxhmgTszmrxi}
\\
&f_1(0)=f_2(0),\,\,f_1(1)=f_2(1),\,\,u_1(0)=u_2(0),\,\,u_1(1)=u_2(1)\label{alignxhmgTszmra}
\\
&f_1'(0)<f_2'(0),\,\,u_1'(0)<u_2'(0),\,\, f_1'(1)>f_2'(1),\,\,u_1'(1)>u_2'(1).\label{alignxhmgTszmrb}
\end{align}
Differentiability is required only from the right at 0 and from the left at 1. We let 
\begin{align}
p(x)&:=\frac{u_2(x)-u_1(x)}{f_2(x)-f_1(x)}, \text{ which is positive for } x\in(0,1) \text{ by \eqref{alignxhmgTszmrxi},} 
\label{eqpdFsnhtMz}\\
h(x,y)&:=p(x)\cdot y + u_1(x) - p(x)\cdot f_1(x)\text{, and }T(x,y):=(x, h(x,y)).
\label{eqhTdhbL}
\end{align} 
At present, $p(x)$ and  $h(x,y)$ are defined only for $0<x<1$. However, indicating the application of \eqref{alignxhmgTszmra} over the equality sign, we let
\begin{align}
p(0)&:=\lim_{x\to 0+0} p(x) 
\eeqref {alignxhmgTszmra} \lim_{x\to 0+0}
\frac{(u_2(x)-u_2(0))-(u_1(x)-u_1(0))}{(f_2(x)-f_2(0))-(f_1(x)-f_1(0))}\cr
&=\lim_{x\to 0+0}
\frac{\frac{u_2(x)-u_2(0)}{x-0}-\frac{u_1(x)-u_1(0)}{x-0}}
{\frac{f_2(x)-f_2(0)}{x-0}-\frac{f_1(x)-f_1(0)}{x-0}} 
=  \frac{u_2'(0)-u_1'(0)}{f_2'(0)-f_1'(0)} \sgeeqref {alignxhmgTszmrb} 0.
\label{alignppZsnnl}
\end{align}
We obtain similarly that $p(1):=\lim_{x\to 1-0} p(x)>0$. 
Hence, $p(x)$ and $h(x,y)$ are defined for all $x\in[0,1]$, and $T(x,y)$ is defined on $[0,1]\times \RR$.
For $x\in[0,1]$, the equality $h(x,f_1(x))=u_1(x)$ is obvious from \eqref{eqhTdhbL}. This yields that $T(\lmoon)=\lmoon'$. For $\rmoon$, we have to work a bit more. If $0<x<1$, then
\begin{align*}
h(x,f_2(x))= p(x)\cdot (f_2(x)-f_1(x)) + u_1(x)= u_2(x)-u_1(x)+u_1(x)=u_2(x),
\end{align*} 
while 
\begin{align*}
h(0,f_2(0))&=p(0)\cdot f_2(0) + u_1(0) - p(0)\cdot f_1(0) \cr
& \eeqref {alignxhmgTszmra}  
p(0)\cdot f_2(0) + u_2(0) - p(0)\cdot f_2(0)=u_2(0).
\end{align*}
We obtain $h(1,f_2(1))= u_2(1)$ similarly. Hence, $T(\rmoon)=\rmoon'$.  

Next, let $g_1$ and $g_2$ be differentiable real functions defined
in some interval $[x_0,x_0+\epsilon)\subseteq [0,1]$ such that 
$g_1(x_0)=g_2(x_0)$, $f_1(x)\leq g_1(x)<g_2(x)\leq f_2(x)$ for all $x\in (x_0,x_0+\epsilon)$, and  $g_1'(x)<g_2'(x)$. This describes the situation where two curved edges within $\lshape(\lmoon,\rmoon,Q)$ have a common terminal point and the first one is to the left of the second one. 
In order to show that their $T$-images have the same property, it suffices to show that $v_1'(x_0)<v_2'(x_0)$, where $v_i(x):= h(x,g_i(x))$. Computing by \eqref{eqhTdhbL} for $i\in\set{0,1}$, we obtain that, for $x_0\in (0,1)$, 
\[v_i'(x_0)=p(x_0)g_i'(x_0)+\underbrace{p'(x_0)g_i(x_0)+u_1'(x_0)-p'(x_0)f_1(x_0)-p(x_0)f_1'(x_0)}.
\]
This implies the required $v_1'(x_0)<v_2'(x_0)$, since $g_1(x_0)=g_2(x_0)$ shows that the under-braced term does not depend on $i$ and $p(x_0)>0$ by \eqref{alignxhmgTszmrxi} and \eqref{eqpdFsnhtMz}. By determining $v_i'(x_0)$ above, we have also obtained that $T$ maps curved edges to curved edges and the ``left to'' relation is preserved, 
except possibly if the curved edge departs from the leftmost point or arrives at the rightmost point of $(\lmoon,\rmoon, Q)$. (At $0$ and $1$, our functions are defined as limits, whereby the standard derivation rules do not apply automatically.) By symmetry, it suffices to deal only with the leftmost point, $(0,f_1(0))$. 
Assume that $g\colon[0,\epsilon)\to \RR$ is a function such that $g(0)=f_1(0)$ and $g$ is differentiable at $0$ from the right and in $(0,\epsilon)$ for some small $0<\epsilon\in\RR$. Let $v(x):=h(x,g(x))$. We need to show that $v(x)$ is differentiable at $0$ from the right; we have already shown  that it is differentiable in $(0,\epsilon)$. We need to show also that the larger the $g_0'(0)$, the larger the $v_0'(0)$. 
Using that $g(0)=f_1(0)$  leads to $v(0)=p(0)g(0)+u_1(0)-p(0)f_1(0)=u_1(0)$, 
\begin{align*}
(v(x)&-v(0))/(x-0)=x^{-1}\cdot\bigl(
p(x)(g(x)-f_1(x))+  u_1(x)-u_1(0) \bigr)
\cr
&=x^{-1}\cdot\bigl(
p(x)( g(x)-g(0)-(f_1(x)-f_1(0)) )+ u_1(x)-u_1(0) \bigr)
\cr
&=p(x) \frac{g(x)-g(0)}{x-0}-p(x) \frac{f_1(x)-f_1(0)}{x-0}
+ \frac{u_1(x)-u_1(0)}{x-0}.
\end{align*}
Hence, letting $x$ tend to $0+0$ and using the continuity of $p$ at $0$ from the right, we obtain that $v'(0)$ (from the  right) exists and $v'(0)= p(0)g'(0)-p(0)f_1'(0) +u_1'(0)$. Since $p'(0)>0$ by \eqref{alignppZsnnl}, we have also shown that if $g_0(0)$ gets larger, then so does  $v'(0)$. 

Finally, we need to show that $T$, defined in \eqref{eqhTdhbL}, is bijective. Clearly, if $(x_1,y_1), (x_2,y_2)\in\lshape(\lmoon,\rmoon,Q)$ such that $x_1\neq x_2$, then $T(x_1,y_1)$ and $T(x_2,y_2)$ differ in their first coordinates. By \eqref{eqpdFsnhtMz} and \eqref{alignppZsnnl}, $p(x)>0$ for all $x\in[0,1]$. Hence if $y_1<y_2$, then $T(x,y_1)\neq T(x,y_2)$ by \eqref{eqhTdhbL}, and it follows that $T$ is injective. 
For a fixed $x_0\in[0,1]$, the positivity of $p(x_0)$ and \eqref{eqhTdhbL} show that $h(x_0,y)$ is a strictly increasing function of $y$. 
This fact, $T(\lmoon)=\lmoon'$ and $T(\rmoon)=\rmoon'$ imply that $T(x_0,y)\in \lshape(\lmoon',\rmoon',Q')$ whenever $(x,y)\in \lshape(\lmoon,\rmoon,Q)$. Thus, $T$ is indeed a map from $\lshape(\lmoon,\rmoon,Q)$ to $\lshape(\lmoon',\rmoon',Q')$, as required. Next, let $(x_0,z)$ be an arbitrary point of $\lshape(\lmoon',\rmoon',Q')$. Then, using that $T$ maps $\lmoon$ and $\rmoon$ onto $\lmoon'$ and $\rmoon'$, respectively, we have that  
\[h(x_0, f_1(x_0))= u_1(x_0)\leq z\leq u_2(x_0)= h(x_0, f_2(x_0)).
\]
Therefore, since we have seen that $h(x_0,y)$ is a strictly increasing continuous function of $y$, there exists a $y_0\in[f_1(x_0), f_2(x_0)]$ such that $h(x_0,y_0)=z$. Thus, $(x_0,y_0)\in\lshape(\lmoon,\rmoon,Q)$ and $T(x_0,y_0)=(x_0,z)$. This shows that $T$ is surjective, completing the first part of the proof. 

Second, we drop the assumption that $x_Q=x_Q'$. By definition, 
both $x_Q$ and $x_Q'$ are in the open interval $(0,1)$. Take a strictly increasing differentiable function $g\colon [0,1]\to[0,1]$ such that $g(0)=0$, $g(1)=1$,  $g'(0)=g'(1)=1$, and 
$g(x_Q)=x_Q'$. With this $g$, consider the transformation $T_2\colon (x,y)\mapsto (g(x),y)$.  Let $(\lmoon'',\rmoon'',Q'')$ be the $T_2$-image of $(\lmoon,\rmoon,Q)$. ``Locally'', $T$ acts approximately like an affine transformation. Furthermore, since  $g'(0)=g'(1)=1$, $T_2$ approximates the identity map at the leftmost and rightmost points of the pointed contour. Hence, it is straightforward to see even without computation that the requirements formulated in Definition~\ref{defThszBjknMdNps} are fulfilled for $T_2$,  $(\lmoon,\rmoon,Q)$, and $(\lmoon'',\rmoon'',Q'')$. Furthermore, $Q''=Q'$ by the choice of $g$. 

Clearly, the composite  $T\circ T_2$ of the translations considered above satisfies the requirements of  Definition~\ref{defThszBjknMdNps}, and it is a bijection  from $\lshape(\lmoon,\rmoon,Q)$ to $\lshape(\lmoon',\rmoon',Q')$. This proves that $(\lmoon,\rmoon,Q)$ and $(\lmoon',\rmoon',Q')$ are equivalent. Furthermore, the bijectivity of $T$ implies that a curved diagram in  $\lshape(\lmoon,\rmoon,Q)$ is planar if and only if so is its $T$-image.
\end{proof}

Armed with Lemma~\ref{lemmamnshps}, the following statement follows trivially from the fact that, for each element $u$ in a diagram $D$, there exists an appropriately small pointed contour $(\lmoon,\rmoon,Q)$ such that $u$ is the only point of the plane that is a common point of $\lshape(\lmoon,\rmoon,Q)$ and the union of edges of $D$. 

\begin{corollary}\label{corolnarrow}
Let  $\nlat L$ and $\nlat K$ be  nearlattices, and let $u\in  L$. By taking isomorphic copies if necessary, we assume that $L\cap K=\emptyset$. Then, on the set $L\cup (K\setminus\set{1_K})$, the ordering defined by 
\[
x\leq y \iff
\begin{cases}
x,y\in L \text{ and } x\leq y\text{ in }\nlat L,\text{ or}\cr
x,y\in K\setminus\set{1_K} \text{ and } x\leq y\text{ in }\nlat K,\text{ or}\cr
x\in K\setminus\set{1_K},\,\, y\in L,\text{ and }u\leq y\text{ in }\nlat L
\end{cases}
\]
yields a new nearlattice $\nlat{K+_u L}$. If $\nlat K$ and $\nlat L$ are planar, then so is $\nlat{K+_u L}$.  
\end{corollary}

The edges of a planar nearlattice divide the plane into regions; see Kelly and Rival~\cite{kellyrival}. In the sense of the Euclidean metric, some of the regions are  infinite as they contain ``the rest of planar points outside'', and there can be \emph{finite} regions. Note that there exists at least one finite region if and only if $\dom\wedge$ contains a pair of incomparable elements.  Following Gr\"atzer and Knapp~\cite{ggknapp1}, a minimal finite region is called a \emph{cell}. In a nearlattice, the lower covers of the top element 1 are called \emph{coatoms}. The following corollary is illustrated by Figure~\ref{figd2}.

\begin{figure}[htb] 
\centerline
{\includegraphics[scale=1.0]{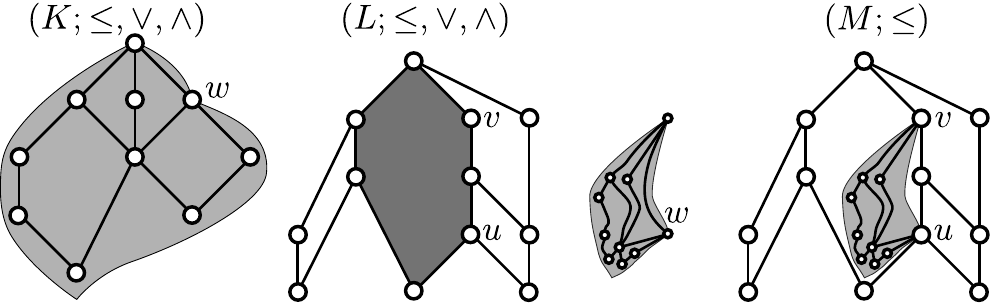}}
\caption{Illustration for Corollary~\ref{corolmoon} 
\label{figd2}}
\end{figure}

\begin{corollary}\label{corolmoon} 
Let  $\nlat L$ and $\nlat K$ be planar nearlattices and fix a planar diagram for each of them. By taking isomorphic copies if necessary, we can assume that $L\cap K=\emptyset$.
Let $u,v\in L$ be distinct elements on the same (left or right) boundary chain of the same cell of $\nlat L$, with respect to its fixed diagram, such that $u<v$ and $u$ is not the smallest element of the cell. Let $w\in K$ be a coatom on the boundary of $\nlat K$, with respect to its fixed diagram, again. Let $\poset M$ be the poset with $M:=L\cup (K\setminus\set{w,1_K})$ and the ordering defined by 
\[
x\leq y \iff
\begin{cases}
x,y\in L \text{ and } x\leq_L y\text{ in }\nlat L,\text{ or}\cr
x,y\in K\setminus\set{w, 1_K} \text{ and } x\leq_K y\text{ in }\nlat K,\text{ or}\cr
x\in K\setminus\set{w, 1_K},\,\, y\in L,\,\, x\leq_K w,
\text{ and } u\leq_L y,\text{ or} 
\cr
x\in K\setminus\set{w, 1_K},\,\, y\in L,
\text{ and } v\leq_L y. 
\end{cases}
\]
Then $\poset M$ is a \emph{planar} poset.
\end{corollary}

\begin{proof} 
After reflecting one or two of our nearlattices across a vertical axis, we can assume that $u$ and $v$ are on the right boundary of the cell mentioned in the theorem and $w$ is on the right boundary of $\nlat K$. The cell in question is dark grey in the second part of Figure~\ref{figd2}. 
Choose a pointed contour $(\lmoon,\rmoon,Q)$ inside the cell such that $v$ is the top of this pointed contour and $Q=u$; see on the right of Figure~\ref{figd2}. Also, choose another pointed contour $(\lmoon',\rmoon',Q')$ such that the diagram of $\lat K$ is inside it and $w=Q'$; see on the left of the figure. It follows from Lemma~\ref{lemmamnshps} that $\nlat K$ has a planar diagram inside $(\lmoon,\rmoon,Q)$ such that $w$ and $1_K$ are at $Q=u$ and $v$. The union of this diagram and that of $\nlat L$ is a curved planar diagram of $\poset M$. Hence, $\poset M$ is planar by D. Kelly's theorem, \eqref{eqpbxDKlls}.
\end{proof}

\section{Substructures, \qn-lattices, and jm-constraints}\label{sectionJMc}

This section begins with some definitions that will be used later in the paper. Note in advance that
even for lattices rather than nearlattices,  the concepts we are going to introduce are distinct from those of partial lattices and weak partial sublattices discussed in Gr\"atzer's monograph \cite{ggglt}. 

For a nearlattice $\nlat L$, the domain  of the meet operation is 
\begin{equation}
\left.
\begin{aligned}
\dom\wedge&=\{(x,y)\in L^2: \inf\set{x,y}\text{ exists}\}\cr
&=\{(x,y)\in L^2: \set{x,y}\text{ has a lower bound}\}
\end{aligned}
\right\}
\label{eqsPiGhWmV}
\end{equation}
since $x\wedge y=\bigvee\set{z: z\leq x\text{ and }z\leq y}$, provided the set mentioned here is nonempty. The domain $\dom\vee$ is, of course, $L\times L$. By reducing the domains, we obtain the following concept.
By a \emph{partial jm-algebra} we shall mean a partial algebra of type $(\vee,\wedge)$ where $\vee$ and $\wedge$ stand for binary partial operations; the letter j and m come from the names of operation symbols. We will adhere to the convention that 
\begin{equation}
\parbox{7.5cm}{
each equality $x\vee y=z$ for a partial jm-algebra will mean that $(x,y)\in\dom\vee$ and $x\vee y=z$, and similarly for the other partial operation.}
\label{pbxdHGjf}
\end{equation} 
A \emph{partial jm-algebra with ordering} is a structure $\nlat K$ such that $\poset K$  is a poset and $\lat K$ is a partial jm-algebra. (That is, we do not require that the partial operations are isotone.) 
The following concept is more subtle. Our guiding example is a nonempty subset $K$ of a nearlattice $\nlat L$ and
the restrictions of the operations of $\nlat L$ to $K$ such that $\dom{\wedge_K}= \dom{\wedge_L}\cap\set{(x,y)\in K^2: x\wedge_L y\in K}$ and analogously for 
$\dom{\vee_K}$.

\begin{definition}\label{defjpm} A \emph{\qn-lattice} is a finite partial jm-algebra $\nlat K$  with ordering such that the following five axioms hold for every $x,y,z,u,v\in K$; convention \eqref{pbxdHGjf} will be in effect.
\begin{itemize}
\item[(A1)] if $(x,y)\in \dom\vee$, then $x\vee y=\sup\set{x,y}$ and, dually,  $x\wedge y=\inf\set{x,y}$ whenever $(x,y)\in\dom\wedge$.
\item[(A2)] if $x$ and $y$ are comparable, then $(x,y)\in\dom\vee$ and  $(x,y)\in \dom\wedge$. 
\item[(A3)] $(x,y)\in\dom\vee \iff (y,x)\in \dom\vee$ and, dually,  $(x,y)\in\dom\wedge \iff (y,x)\in \dom\wedge$.
\item[(A4)] $x\vee y=z$, $z\vee u=v$ and $y\leq u$ imply that $x\vee u=v$. Dually, $x\wedge y=z$, $z\wedge u=v$ and $u\leq y$ imply that $x\wedge u=v$.

\item[(A5)] $x\vee y=z$ and $x\leq u\leq z$ imply that $u\vee y=z$. Dually, $x\wedge y=z$ and $z\leq u\leq x$ imply that $u\wedge y=z$.
\end{itemize}
\end{definition}

The letters q and n in the name of \qn-lattices comes from ``quasi'' and ``near''.  Clearly, every nearlattice is also a \qn-lattice, and every \qn-lattice is a partial jm-algebra with  ordering. 
Since the notations $\nlat K$ and $\lat K$ can mean various things like a nearlattice, a lattice, or a partial jm-algebra, it will be important to frequently specify the meanings of our notations. 

Let $\nlat L$ and $\nlat K$ be partial jm-algebras  with orderings. (In particular, they can be nearlattices or a \qn-lattices.) We say that  $\nlat K$ is a \emph{weak partial subalgebra} of $\nlat L$ if its ordering is the restriction of that of $L$ to $K$,   $K\subseteq L$, 
\begin{equation}
\text{$\dom{\vee_K}\subseteq \dom{\vee_L}$, \  $\dom{\wedge_K}\subseteq \dom{\wedge_L}$,}
\label{eqtxtHnprPVlDtnC} 
\end{equation}
and, in addition, $\vee_K$ and $\wedge_K$ are  the restrictions of $\vee_L$ to $\dom{\vee_K}$ and $\wedge_L$ to $\dom{\wedge_K}$, respectively. In this case, $\lat K$ is also said to be a \emph{weak partial subalgebra} of $\nlat L$. 
The adjective ``weak'' reminds us that the inclusions in \eqref{eqtxtHnprPVlDtnC} 
 can be proper.  If both $\nlat L$ and $\nlat K$ are \qn-lattices, then we prefer to say that $\nlat K$ is a \emph{sub-\qn-lattice} of $\nlat L$ instead of saying that it is a weak partial subalgebra of $\nlat L$. Let us emphasize that, by definition, 
\begin{equation}
\text{a sub-\qn-lattice is automatically a \qn-lattice.}
\label{eqtxtbhTsmRtghjH}
\end{equation}
For a \qn-lattice $\nlat K$, 
\begin{equation}\left.
\parbox{6.1cm}{the set $\Sub\nlat K$ of \emph{subuniverses} is defined analogously to \eqref{eqSubdF};}\right\}
\label{pbxZtvnkSlG}
\end{equation}
the only difference (apart from replacing $L$ by $K$) is that now we need $(x,y)\in X^2\cap \dom{\vee_K}$ rather than $(x,y)\in X^2$. As it is clear from \eqref{pbxZtvnkSlG}, now
a subuniverse is just a subset of $K$ without any structure on it. Thus,  
a \qn-lattice $\nlat K$  can have much more sub-\qn-lattices than  $|\Sub\nlat K|$, because a subuniverse can be the support set of several sub-\qn-lattices. So the counterpart of \eqref{pbxzhBrjplgnK}  does not hold for \qn-lattices. 
Having its subuniverses just defined, $\many\nlat K$ is also meaningful for a \qn-lattice $\nlat K$; see Definition~\ref{defsigMa}.

The following easy lemma, which is a particular case of Lemma 2.3 of Cz\'edli~\cite{czg83}, indicates the importance of our new concepts.

\begin{lemma}\label{lemmaGvngnnT}
If  $\nlat K$ is a sub-\qn-lattice of a \qn-lattice $\nlat M$, then $\many \nlat K\geq \many\nlat M$.
\end{lemma}

By a \emph{jm-constraint} over a set $K$ we mean a formal equality $x\vee y=z$ or $x\wedge y=z$ such that $\set{x,y,z}$ is a three-element subset of $K$. 
If $W$ is a set of jm-constraints such that 
whenever $x\vee y=z_1$ and $x\vee y=z_2$ belong to $W$ then $z_1=z_2$ and dually, then $W$ is \emph{coherent}. 
A coherent $W$ together with $K$ \emph{determine a partial algebra} $\lat K$ in the natural way: $\dom{\vee_K}=\{(x,y)\in K^2:\text{ there is a }z\in K$ such that $x\vee y=z$ belongs to $W\}$, similarly for $\dom{\wedge_K}$, and the action of $\vee_K$ and $\wedge_K$ in their domains are given by the jm-constraints in $W$. If $K$ is understood, we speak about the partial algebra determined by $W$. 
We are interested only in the following particular case.

\begin{definition}\label{defcjmcRsrtS}
\begin{enumeratei} 
\item\label{defcjmcRsrtSa} Let $K$ be nonempty subset of a \qn-lattice $\nlat L$. Let $W$ be a set of jm-constraints over $K$ such that each jm-constraint in $W$ is a valid equality in $\nlat L$; such a $W$ is necessarily coherent. Then $W$ is said to be a \emph{set of jm-constraints over $K$ compatible with $\nlat L$}. Note that the partial algebra determined by $W$ and $K$ is clearly a weak partial subalgebra of $\nlat L$. 
\item\label{defcjmcRsrtSb} 
By an  \emph{$\nlat L$-compatible set of jm-constraints} we mean a set $W$ of jm-constraints over some $K\subseteq L$ such that $W$ is compatible with $\nlat L$. In other words, $W$ is a collection of true equalities in $\nlat L$.
\item\label{defcjmcRsrtSc} Over a subset $K$ of $L$,  let $W$ be a set of jm-constraints compatible with $\nlat L$, and keep  \eqref{eqtxtbhTsmRtghjH} in mind. The least sub-\qn-lattice $\nlat K$ of $\nlat L$ such that 
all the jm-constraints in $W$ are valid equalities in this sub-\qn-lattice is called the \emph{\qn-lattice determined by $W$ over $K$ in $\nlat L$}. Here ``least'' means that whenever all the jm-constraints of $W$ are valid equalities in a sub-\qn-lattice $\ppnlat K$ of $\nlat L$, then $\nlat K$ is a sub-\qn-lattice of $\ppnlat K$.  
\item\label{defcjmcRsrtSd} If $K$ is the collection of all elements occurring in the jm-constraints belonging to $W$, then the reference to $K$ in the form  ``over $K$'' is usually dropped and we speak of the \emph{\qn-lattice determined by $W$ in $L$}.  If every element of a \qn-lattice $\nlat L$ occurs in a jm-constraint belonging to $W$, then even the reference to $\nlat L$ is often dropped and we simply speak of the  \emph{\qn-lattice determined by $W$}; however, then it should be clear from the context what $\poset L$ is. This convention of not mentioning $\nlat L$ is typical when $\nlat L$ is given by its diagram.
\item\label{defcjmcRsrtSe} If a \qn-lattice $\nlat L$ is determined by $W$ and a diagram as in \eqref{defcjmcRsrtSd} and the diagram contains \emph{dashed edges}, then $\nlat L$ means any of the several \qn-lattices determined so that we remove some of the dashed edges, possibly none of them, and make solid the rest of dashed edges, possibly none of them. In this case, a statement ``$\nlat L$ is not a sub-\qn-lattice of a given \qn-lattice $\nlat M$'' means that ``no matter which dashed edges are erased and which are made solid, the \qn-lattice determined in this way is not a sub-\qn-lattice of $\nlat M$''. Note that the dashed edges  should not be confused with the dotted ones occurring later in the paper.
\end{enumeratei}
\end{definition}

\begin{figure}[htb] 
\centerline
{\includegraphics[scale=1.0]{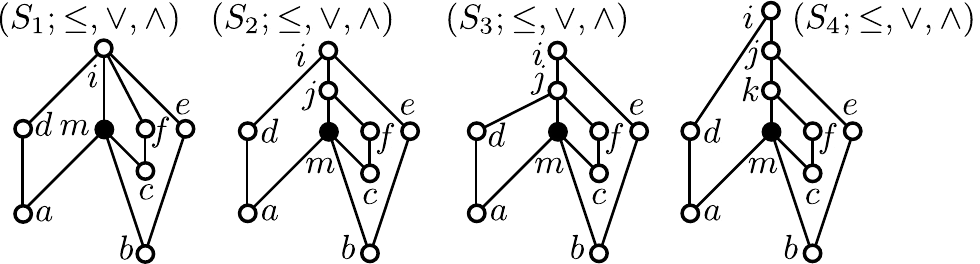}}
\caption{The \qn-lattices $\nlat{S_\ell}$ for $\ell\in\set{1,\dots,4}$
\label{figd3}}
\end{figure}

\begin{example}\label{examplSlCs} To exemplify Definition~\ref{defcjmcRsrtS} \eqref{defcjmcRsrtSd}, we define four \qn-lattices as follows; note that they will be needed later. Let 
\begin{align*}
W&:=\{a\vee b=m,\,\, a\vee c=m,\,\, b\vee c=m, \cr
 &\phantom{:=\{}d\wedge m=a,\,\, e\wedge m=b,\,\, f\wedge m=c\},\cr
W_1&:=\set{d\vee m=i,\,\,e\vee m=i,\,\,f\vee m=i},\cr
W_2&:=\set{d\vee m=i,\,\,e\vee m=i,\,\,f\vee m=j},\cr
W_3&:=\set{d\vee m=j,\,\,e\vee m=i,\,\,f\vee m=j},\cr
W_4&:=\set{d\vee m=i,\,\,e\vee m=j,\,\,f\vee m=k};
\end{align*}
see Figure~\ref{figd3}. Then, for $\ell \in\set{1,2,3,4}$, let $\nlat{S_\ell }$ be the \qn-lattice determined by the set $W\cup W_\ell$ of jm-constraints and its diagram. Note that $\nlat{S_1}$ is a nearlattice but this is not so for $\nlat{S_\ell }$, $\ell \in\set{2,3,4}$. For example, $(e,j)\notin\dom{\wedge}$ for $\ell \in\set{2,3}$ and $(e,k)\notin\dom{\wedge}$ for $\ell =4$.
\end{example}

The following lemma is quite easy but it will be important in most of our arguments later.

\begin{lemma}\label{lemmacjmcRsrtSd}  
If $W$ is a set of jm-constraints compatible with a \qn-lattice $\nlat L$ over $K$, $\emptyset\neq K\subseteq L$, then the \qn-lattice determined by $W$ over $K$ can be described as follows. We begin with the partial algebra $(K;\vee_0,\wedge_0)$ determined by $W$, and take  $(K;\leq,\vee_0,\wedge_0)$, where $\leq_K$ is inherited from $\nlat L$. Assume that, for some $i\in\nnul:=\set{0,1,2,\dots}$,   $(K;\leq,\vee_i,\wedge_i)$ is already given.  If one of the axioms (A1)--(A5) is violated, then pick a pair $(x,y)\in K^2$ and one of the axioms violated by this pair, and extend the domain of $\vee_i$ or $\wedge_i$ by $(x,y)$ to get rid of this violation; let $(K;\leq,\vee_{i+1},\wedge_{i+1})$ denote what we obtain in this way. Note that $(K;\vee_i,\wedge_i)$ is a weak partial subalgebra of  $(K;\vee_{i+1},\wedge_{i+1})$. If none of the  axioms (A1)--(A5) is violated, which happens sooner or later by finiteness, then  $\nlat K:=(K;\leq,\vee_i,\wedge_i)$ is the \emph{\qn-lattice determined by $W$ over $K$ in $\nlat L$}.
\end{lemma}

\begin{proof} Since none of the axioms (A1)--(A5) is violated after the inductive procedure, $\nlat K$ is a \qn-lattice. We need to show only that $(K;\leq,\vee_i,\wedge_i)$ is a weak subalgebra of $\nlat L$ for all $i\in\nnul$. We prove this by induction on $i$. The case $i=0$ is evident. Assume that $(K;\leq,\vee_i,\wedge_i)$ a weak subalgebra of $\nlat L$ for some $i$. We can assume that  $(K;\leq,\vee_{i+1},\wedge_{i+1})$ is obtained from $(K;\vee_i,\wedge_i)$ so that we got rid of a violation of (A4) or (A5), because the axioms (A1)--(A3) create no problem. If the first half of (A4) was violated then, with the notation taken from Definition~\ref{defjpm}, the induction hypothesis gives that 
$x\vee_L y=x\vee_i y=z$ and $z\vee_L u=z\vee_i u=v$. Since $y\leq u$, we can compute in $\nlat L$ as follows:
$x\vee u=x\vee (y\vee u)=(x\vee y)\vee u=z\vee u=v$. That is, $x\vee_L u=v$. Hence,  enriching the domain of $\vee_i$ with $(x,u)$ and letting $x\vee_{i+1}u=v$ results in a weak subalgebra of $\nlat L$, and this subalgebra is  $(K;\leq,\vee_{i+1},\wedge_{i+1})$. Duality takes care of the second part of (A4); however, then the following fact, which is a trivial property of infima, has also to be used:
\begin{equation}
\left. 
\parbox{9cm}{if $y\wedge z$ and $x\wedge (y\wedge z)$ are defined in a nearlattice, then so are $x\wedge y$ and $(x\wedge y)\wedge z$, and the latter equals $x\wedge (y\wedge z)$.}\right\}
\end{equation}
A similar argument applies for (A5) since nearlattice operations are isotone. This completes the induction step and the lemma is concluded. 
\end{proof}

\begin{convention}\label{convSsttn}
Given a nearlattice $\nlat L$, let $W$ be an $\nlat L$-compat\-ible set of jm-constraints. By the  \emph{\tmany-value $\many(W)$ of $W$} and, if $W$ and $\nlat L$ are understood from the context, the \emph{\tmany-value of the situation} we mean $\many\nlat K$ where $K$ is a subset of $L$ such that the jm-constraints in $W$ are over $K$ and $\nlat K$ is the
\qn-lattice determined by $W$ over $K$ in $\nlat L$. 
The least appropriate $K$ will be denoted by $K_W$; that is, $K_W$ is the collection of all elements that occur in jm-constraints belonging to $W$. Hence, $\many(W)=\many\nlat{K_W}$. If $W$ or the situation is clear from the context, then its \tmany-value will often be given by an equality like $\many=83$.
\end{convention}

We formulate the following easy lemma, which will be used implicitly.

\begin{lemma}\label{lemmacnVmsKn} 
Convention~\ref{convSsttn} makes sense, that is, $\many\nlat K$ above does not depend on the choice of the subset $K$ of $L$.
\end{lemma}

\begin{proof}
Let $K$ be an arbitrary subset of $L$ such that $W$ is over $K$. Clearly, $K_W\subseteq K$. Since there is no stipulation on the elements of $K\setminus K_W$, we have that 
\begin{align*}
\many\nlat K&=
|\Sub\nlat K|\cdot 2^{8-|K|} \cr
&=|\set{X\cup Y: X\in\Sub\nlat{K_W},\,\, Y\subseteq K\setminus K_W}|\cdot 2^{8-|K|}\cr
&=|\Sub\nlat {K_W}|\cdot |\set{Y: Y\subseteq K\setminus K_W}|\cdot 2^{8-|K|}\cr
&=|\Sub\nlat {K_W}|\cdot 2^{8-|K_W|} \cdot 2^{|K\setminus K_W|}\cdot 2^{|K_W|-|K|}\cr
&=\many\nlat {K_W},
\end{align*}
completing the proof.
\end{proof}

Lemmas~\ref{lemmaGvngnnT} and \ref{lemmacnVmsKn} together with Convention~\ref{convSsttn} imply the following lemma.

\begin{lemma}\label{lemmasmzgsTtn} 
If $\nlat L$ is a nearlattice and $W$ is an $\nlat L$-compat\-ible set of jm-constraints, then 
$\many(W)\geq \many\nlat L$. In other words, then 
$\many \nlat L$ is at most the \tmany-value of the situation.
\end{lemma}

This lemma will be our main tool to show that $\many\nlat L$ is sufficiently small.

\subsection*{A computer program}
Lemma~\ref{lemmasmzgsTtn} will be useful for our purposes only if we can determine the \tmany-values of many situations. Since that much work would be impossible manually, we have developed a computer program, using Bloodshed Dev-Pascal v1.9.2 (Freepascal) under Windows 10, to do it. 
This program, called \emph{sublatts}, is available from the author's website;
to find it, look for the present paper in the list of publications.  
The input of the program is an unformatted text file describing an $\nlat L$-compatible set $W$ of jm-constraints and the corresponding poset $\poset{K_W}$; see Convention~\ref{convSsttn}. As its output, the program displays $\many(W)$ on the screen and saves it into a text file.
Together with the result, $\many(W)$, the set $W$ is also displayed and saved. 
Upon request (using the \verb!\verbose=true! command), even the \qn-lattice $\nlat{K_W}$ determined by $W$ 
is displayed and saved. The algorithm implemented by the program  is trivial. 
Namely, by computing the $(K_W;\leq, \vee_i, \wedge_i)$ for $i=0,1,2,\dots$ successively according to Lemma~\ref{lemmacjmcRsrtSd}, the program determines the \qn-lattice $\nlat{K_W}$ in the first step. In the second step, the program takes all the $2^{|K_W|}$ subsets of $K_W$ and counts those that are closed with respect to the partial operations of $\nlat{K_W}$. It is clear from the second step that the running time depends exponentially on $|K_W|$. Fortunately, the biggest $|K_W|$ we need for this paper is only 12, and the program computes $\many(W)$ for 101 many times 
in half a second on a desktop computer with an Intel Core i5-4400 Quad-Core 3.10 GHz processor.

Note that an earlier program, which was crucial for the papers Cz\'edli~\cite{czg83} and \cite{czg127}, could also be used here but that would require much more human effort. This is so because the above-mentioned first step is not built in the earlier program and the user has to make this step manually while preparing the input files. Note also that the concept of \qn-lattices 
has been developed for the sake of this first step.

As a consequence of Lemma~\ref{lemmacnVmsKn}, note that 
if the input files gives a poset $\poset K$ such that 
$\poset {K_W}$ is a proper subposet of $\poset K$, then the program still computes $\many(W)$ but in a longer time. Note also the following. Even if the program can detect many types of errors in the input file, it is the user's responsibility that the ordering should harmonize with the \qn-lattice operations. However, in most of the cases, it will not cause an error if some edges are missing from the input; see \eqref{eqpbxdelEDg} and Remark~\ref{remdelEdge} later.

\begin{figure}[htb] 
\centerline
{\includegraphics[scale=1.0]{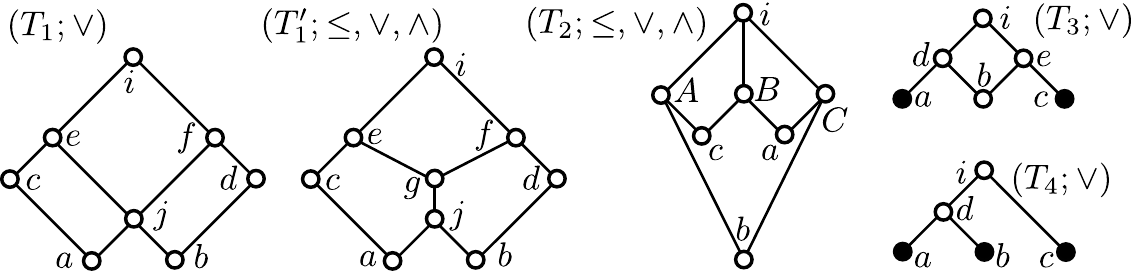}}
\caption{Three join-semilattices and two \qn-lattices
\label{figd4}}
\end{figure}%

\section{Excluding some \qn-sublattices}\label{sectexsomeqn}
In order to outline the purpose of this section, we need the following convention.

\begin{convention}\label{convinDrkt}
For the rest of the paper, we assume that Theorem~\ref{thmreform} fails and $\nlat\ce$ will denote a counterexample of minimal size $|\ce|$. In particular, $\many\nlat\ce>83$ but $\nlat\ce$ is not planar. The notation $\nlat\ce$ will always be understood as a nearlattice (rather than, say, a \qn-lattice).
\end{convention}

We are going to prove several properties of $\nlat\ce$ until it appears that $\nlat\ce$ cannot exist; this will imply Theorem~\ref{thmreform}. We begin with the following easy lemma. 
As usual, a subnearlattice is \emph{proper} if it is not the original nearlattice.

\begin{lemma}\label{lemmaproperplanar} 
Every proper subnearlattice of  $\nlat\ce$ is planar.
\end{lemma}

\begin{proof} Let $\nlat K$ be a proper subnearlattice of $\nlat\ce$. Since $|K|<|\ce|$, Theorem~\ref{thmreform} holds for $\nlat K$.  By Lemma~\ref{lemmaGvngnnT} and Convention~\ref{convinDrkt}, we obtain that $83<\many\nlat\ce\leq\many\nlat K$. Hence, Theorem~\ref{thmreform}, which is now applicable even if it has not been proved, yields that $\nlat K$ is planar.
\end{proof}

For later reference, we prove the following easy lemma. 
Not only the statement of the lemma but also its straightforward proof will often be referenced, explicitly or (later) implicitly. The \qn-lattice $\nlat{T'_1}$ in Figure~\ref{figd4} is defined as follows. 
Let $W$ be the collection of the following jm-constraints: $c\vee j=e$, $j\vee d=f$, $y\vee z=x$ for  every  $x$ with two distinct lower covers $y$ and $z$, and $y\wedge  z=x$ for  every  $x$ with two distinct covers $y$ and $z$. Then  $\nlat{T_1'}$ is the \qn-lattice determined by $W$.

\begin{lemma}\label{lemmatgsPlts}
If the join-semilattice $\slat{T_1}$ given in Figure~\ref{figd4} is a \emph{subposet} of a finite join-semilattice $\slat K$, then  the nearlattice $\nlat{T_1}$ is a subnearlattice of the nearlattice $\nlat K$, or the \qn-lattice $\nlat{T'_1}$  a sub-\qn-lattice of the nearlattice $\nlat K$.
\end{lemma}

\begin{proof}Unless otherwise stated by subscripts, the operations will be understood in $\nlat K$. 
If $j\neq a\vee b$, then let $j':=a\vee b$. Clearly, 
$j'<j$ but $j'\not\leq c$ and $j'\not\leq e$, because otherwise we would get $b\leq c$ or $a\leq d$, which fail.
Hence, we can replace $j$ by $j'$ so that $\poset{(T_1\setminus\set{j})\cup\set{j'}}$ is (order-) isomorphic to 
$\poset{T_1}$. This allows us to assume that $j=a\vee b$. 
Since $c\not\leq f$, we have that $c\vee j\not\leq f$.
Hence, if $c\vee j\neq e$, then we can replace $e$ by 
$c\vee j$ and we still obtain a poset isomorphic to $\poset{T_1}$. Thus, we can assume that $c\vee j=e$. 
In the next step, based on symmetry (reflection across a vertical line), we can similarly assume that $j\vee d=f$.

Next, if $c\wedge j\neq a$, then we replace $a$ by 
$a':=c\wedge j>a$. Clearly, say, $a'\not\geq b$ since $c\not\geq b$, and we get an isomorphic poset. But now we have to show that an earlier achievement remains valid, that is, $a'\vee b=j$. This is clear again since $a'\leq j$ and so $j=a\vee b\leq a'\vee b\leq j$. Hence, we can assume that $c\wedge j=a$. In the next step, we can also assume 
$j\wedge d=b$ by symmetry. Note at this point that the order of our steps made so far was not arbitrary. 
Now, there are two cases.

First, if $e\wedge f=j$, then the equalities assumed so far are sufficient to say that the nearlattice $\nlat{T_1}$ is a subnearlattice of the nearlattice $\nlat K$; for example, $c\wedge f = (c\wedge e)\wedge f=c\wedge(e\wedge f)=c\wedge j=a$. In other words, denoting by $W$ the collection of the equalities assumed, the \qn-lattice $\nlat{T_1}$ determined by $W$ is the same as the nearlattice $\nlat{T_1}$.

Second, if $e\wedge f=:g>j$, then $\set{a,b,\dots,g,i,j}$ is a subposet isomorphic to $\poset{T'_1}$. For example, 
$c\not\leq g$ because otherwise we would obtain that $c\leq f$. The earlier equalities and $e\wedge f=g$ form exactly the set of jm-constraints determining $\nlat{T'_1}$, and the lemma follows. 
\end{proof}

Note in advance that the following lemma as well as many of the subsequent lemmas come with associated input and output files, which are available together with our computer program from the author's website. Also, the extended  version\footnote{\red{This is the extended version.}} of the paper, available from arXiv or preferably from the author's website, contains all the output files as appendices, and the input files can easily be obtained from the output files. Note also that, as a rule, the input and output files associated 
with a lemma on $X_i$ are called \verb!LmXi.txt! and \verb!LmXi-out.txt!, respectively. Here $i$ is either a concrete natural number, or it is the letter ``i'' to denote a range of natural numbers. Once, to differentiate between two versions, we insert ``\verb!a!'' or `\verb!b!''  right before `\verb!-out.text!''.  For example, Lemma~\ref{LmT4} is a statement on $T_4$ (to be defined in due course) and the associated files are \verb!LmT4.txt! and \verb!LmT4-out.txt! while the files corresponding to the following lemma are \verb!LmSi.txt! and \verb!LmSi-out.txt!
The \qn-lattices occurring in the lemma below have been defined in Example~\ref{examplSlCs}.

\begin{lemma}\label{LmSi}
For $i=1,\dots,4$, the \qn-lattice $\nlat{S_i}$ is not a sub-\qn-lattice of $\nlat\ce$.
\end{lemma}

\begin{proof} Since $\many\nlat{S_1}=77$, $\many\nlat{S_2}=69.5$, $\many\nlat{S_3}=69.5$, and $\many\nlat{S_4}=64.75$, computed by the program, are all less than or equal to 83, the lemma we are proving follows from Lemma~\ref{lemmaGvngnnT} and Convention~\ref{convinDrkt}.
\end{proof}

The next property of $\nlat\ce$ that we are going to prove is the following.

\begin{lemma}\label{LmT1a}
The nearlattice $\nlat{T_1}$ is not a subnearlattice of $\nlat\ce$.
\end{lemma}

The method of the proof below will be referenced as our \emph{parsing technique (with the help of our computer program)}. If the reader intends to check the output file \texttt{LmT1a-out.txt}, then the appendix given in Section~\ref{sectappndf} is worth reading.

\begin{proof}[Proof of Lemma~\ref{LmT1a}] Suppose the contrary. The notation given in Figure~\ref{figd4} will be in effect. The initial situation, to be denoted by \azeset{}, consists of  $\poset{T_1}$ and the \emph{operation table}, that is, all equalities of the form $x\vee y=z$ and $u\wedge v=w$ that hold in $\nlat{T_1}$. Note that the input file need not and does not contain the whole operation table; it suffices to give a smaller set $W$ of jm-constraints such that the \qn-lattice determined $W$ and $\poset{T_1}$ is the nearlattice $\nlat{T_1}$. In fact, it follows from Lemma~\ref{lemmaGvngnnT} that even a set smaller than $W$ suffices if its \tmany-value is at most 83. Unfortunately, since now the \tmany-value of \azeset{} is 84, which is too large for us, further work is necessary. 
Note the following principle in advance; it will frequently be used later, mostly without referencing it explicitly.
\begin{equation}\left.
\parbox{8.9cm}{By Lemma~\ref{lemmaGvngnnT} and Convention~\ref{convinDrkt}, if the \tmany-value of a case (or situation) is at most $83$, then this case (or situation) is excluded. Furthermore, if $\many\leq 83$ for \emph{all possible cases}, then $\many\nlat\ce\leq 83$ by Lemma~\ref{lemmaGvngnnT} and Convention~\ref{convinDrkt}, which is a contradiction completing the proof.}\right\}
\label{eqpbxqed}
\end{equation}

Since $\nlat{T_1}$ is planar but $\nlat\ce$ is not, there is an element $h\in \ce\setminus T_1$. If all elements of $\ce\setminus T_1$ belonged to the principal filter $\filter i$, then $\filter i$ would be planar by Lemma~\ref{lemmaproperplanar} and thus $\ce$ would also be planar, which is not the case. Hence, we can assume that $h\not\geq i$, and there are two cases, \azeset1 and \azeset2, to consider. 
As always in the rest of the paper unless otherwise explicitly stated, the operations will be understood in $\nlat\ce$. Before going into details, let us agree that, in the whole paper, our convention for subcases is the following.
\begin{equation}
\left.
\parbox{9.2cm}{Case \azeset{$\vec y$} is a \emph{subcase} of  
Case \azeset{$\vec x$} or, in other words, \azeset{$\vec x$} is a \emph{parent case} of \azeset{$\vec y$} if and only if $\vec x$ is a proper prefix, possibly the empty prefix, of the string $\vec y$. When dealing with \azeset{$\vec y$}, then the assumptions and jm-constraints  of all parent cases are automatically assumed and they belong to the situation even if this is not emphasized all the time.}
\right\}
\label{pbx8mjKrhcs}
\end{equation}
Note that, in order to increase readability, we insert a dot after every second character of $\vec y$ when referencing \azeset{$\vec y$}.

\eset{1}: $h$ and $i$ are incomparable, in notation, $h\parallel i$. Let $i\vee h:=k$. According to \eqref{pbx8mjKrhcs}, now the situation is that of \azeset{} together with $i\vee h=k$, $i<k$, and $h<k$. Using our computer program with the input file \verb!LmT1a.txt!, we obtain that the \tmany-value of the situation \azeset{1} is $71.75$; see Convention~\ref{convSsttn}. This does not exceed $83$, so we are on the way to applying \eqref{eqpbxqed}. (Alternatively, \eqref{eqpbxqed} excludes this case.)

\eset{2}: $h<i$. This case branches into two subcases. 

\eset{2a}: $h<a$ and $h<b$. Then we can assume that $a\wedge b=h$ since otherwise we can replace $h$ by $a\wedge b$. The \tmany-value of the situation is $73.5$.

\eset{2b}: $h>a$ and $h>b$. Then $h\geq j=a\vee b$. By \eqref{pbx8mjKrhcs}, $h<i$. But $h\notin T_1$, whence $j<h<i$. Using that $e\wedge f=j$, $e\vee f=i$, and $e\parallel f$, it follows that $h\parallel e$ or $h\parallel f$. By symmetry, we can assume that $h\parallel e$. Our argument splits according to the value of $e\vee h$. First, if \azeset{2b.1}: $e\vee h=i$, then $\many=77$. Second, if \azeset{2b.2} $e\vee h=:k$ is a new element, then $k<i$ and $\many=66.5$. By \eqref{eqpbxqed},  \azeset{2a} and  \azeset{2b} are excluded, whence  there remains only one subcase at this level of parsing: $h\parallel a$ or $h\parallel b$. By symmetry, it suffices to deal only with the former.

\eset{2c}: $h\parallel a$. There are six subcases depending on $h\vee a$. Note that $h\vee a\leq i$ and $h\vee a$ cannot be $b$ or $d$, because they are not in the filter $\filter a$. 
If \azeset{2c.1}: $h\vee a=:k$ is a new element, then $h<k$, $a<k$, $k<i$, and $\many=74.5$. If \azeset{2c.2}: $h\vee a=c$, then $\many=64.5$. For \azeset{2c.3}: $h\vee a=j$, we have that  $\many=75.5$. \azeset{2c.4}: $h\vee a=e$ gives that $\many=65.5$. The \tmany-value of \azeset{2c.5}: $h\vee a=f$ is $67.5$. Finally, \azeset{2c.6}: $h\vee a=i$ yields that $ \many=68.5$. By \eqref{eqpbxqed}, \azeset{2c} is excluded. 

We have parsed all subcases, in other words, the \emph{parsing tree is complete}. By \eqref{eqpbxqed}, the proof of Lemma~\ref{LmT1a} is complete.
\end{proof}

The lemma below is much easier; the \qn-lattice $\poset{T'_1}$ has been defined right before Lemma~\ref{lemmatgsPlts}.
The file associated with this lemma is \verb!LmT1b-out.txt! .
\nothing{It is a separate statement only because each of the input files are named by the corresponding lemmas.}

\begin{lemma}\label{LmT1b}
The \qn-lattice $\poset{T'_1}$ is not a sub-\qn-lattice of $\nlat\ce$.
\end{lemma}

\begin{proof} Since $\many\nlat{T'_1}=81$, Lemma~\ref{lemmaGvngnnT} and Convention~\ref{convinDrkt} apply.
\end{proof}

Now, we are in the position to prove the following statement.

\begin{lemma}\label{Lt1}
The poset $\poset{T_1}$, given by Figure~\ref{figd4}, is not a subposet of the nearlattice $\nlat\ce$.
\end{lemma}

\begin{proof} Combine Lemmas~\ref{lemmatgsPlts}, \ref{LmT1a}, and \ref{LmT1b}.
\end{proof}

Next, let $\nlat{T_2}$ be the seven-element \qn-lattice defined by Figure~\ref{figd4} and determined by the set $W:=\{a\vee b=C$, $b\vee c=a$, $c\vee a=B$, $A\vee B=i$, $B\vee C=i$, $C\vee A=i\}$  of jm-constraints. Note that $W$ is redundant. Note also that $\slat{T_2}$ is the join-semilattice freely generated by $\set{a,b,c}$ and that $\dom\wedge=\emptyset$.

\begin{lemma}\label{LmT2} The \qn-lattice $\nlat{T_2}$ is not a sub-\qn-lattice of the nearlattice $\nlat\ce$.
\end{lemma}

\begin{proof} Suppose the contrary. A routine argument, similar to the proof of Lemma~\ref{lemmatgsPlts}, shows that $B\wedge C=a$, $C\wedge A=b$, and $A\wedge B=c$ can be assumed, because otherwise we can easily enlarge $a$, $b$, and $c$.
So, in the rest of the proof, we add these three equalities to the set of jm-constraints defining the \qn-lattice $\nlat{T_2}$. So, from now on in the proof, $\nlat{T_2}$ is a nearlattice. Since $\many\nlat{T_2}=90$ is rather large, a whole hierarchy of cases have to be considered. Since $\nlat{T_2}$ is planar and it is a sub-\qn-lattice of the nonplanar nearlattice $\nlat\ce$, it follows that $\ce\setminus T_2$ is nonempty. By Lemma~\ref{lemmaproperplanar}, its filter $\filter i$ is planar. Hence,  $\ce\setminus T_2$ is not included in this filter, because otherwise, as the glued sum of two planar nearlattices, $\nlat\ce$ would be planar.  
So, we can pick an element $d\in \ce$ such that $d\not\geq i$; so either $d\parallel i$, or $d<i$. 
If \azeset1:  
$d\parallel i$, then with $d\vee i:=j$, we obtain that $\sigma=75.5$. 
\eset{2}: $d<i$; this case splits into two subcases. First, assume that \azeset{2a}: $d<a$, $d<b$, and $d<c$. Then the principal filter $\filter d$ in $\nlat\ce$ is a lattice by ~\eqref{eqsPiGhWmV}. In this lattice, $\set{A=b\vee c,B=a\vee c,C=a\vee b}$ generate an 8-element boolean sublattice by, say, Gr\"atzer~\cite[Lemma 73]{ggglt}. By the program or Cz\'edli~\cite[Lemma 2.7]{czg83}, the \tmany-value of this boolean sublattice is 74, whereby \eqref{eqpbxqed} excludes  \azeset{2a}.   The conjunction of $d>a$, $d>b$ and $d>c$ is also excluded, since it would lead to
$d\geq a\vee b\vee c=i$, contradicting the choice of $d$. Hence, using that \azeset{2a} has already been excluded, $d$ is incomparable with at least one of $a$, $b$, and $c$. By symmetry, we can assume that $d\parallel a$. So, the second subcase is  \azeset{2b}: $d\parallel a$. Since $a\vee d\in\filter a\setminus\set a$, we have that 
$a\vee d\in\set{B,C,i,e}$, where $e$ denotes a new element such that $e<i$ since $a,d<i$. So there are exactly four subcases; namely, \azeset{2b.1}: $a\vee d=e$, \azeset{2b.2}: $a\vee d=B$, \azeset{2b.3}: $a\vee d=C$, and \azeset{2b.4}: $a\vee d=i$, but all of them is excluded by \eqref{eqpbxqed} since the corresponding \tmany-values are $77.5$, $80$, $80$, and $81$. (Note that \azeset{2b.3} does not need a separate computation since it follows from \azeset{2b.2} by $B$--$C$ symmetry.)  By \eqref{convinDrkt}, \azeset{2b} is excluded and  Lemma~\ref{LmT2} is concluded. 
\end{proof}

An element in a poset is \emph{meet irreducible} if it has exactly on cover.
If an element $x$ has at least two covers, then it is \emph{meet-reducible}, no matter whether any meet resulting in $x$ is defined in the \qn-lattice we deal with.
Except for the singleton nearlattice (or join-semilattice), which is surely distinct from $\nlat\ce$, a minimal element is either meet irreducible, or meet reducible.

\begin{lemma}\label{lemmatwzBrqv} 
Every minimal element of $\nlat\ce$ is meet reducible.
\end{lemma}

\begin{proof}
Suppose the contrary, and let $u\in\ce$ be a meet irreducible minimal element.  Clearly, $\nlat{\ce\setminus\set{u}}$ is a proper subnearlattice of $\ce$. This subnearlattice is planar by Lemma~\ref{lemmaproperplanar}. With reference to Corollary~\ref{corolnarrow},  $\nlat\ce$ is obtained from this subnearlattice, playing the role of $\nlat L$, and the two-element lattice, acting as $\nlat K$, by the construction described there. Hence, $\nlat\ce$ is planar by Corollary~\ref{corolnarrow}, which is a contradiction as required.
\end{proof}

\section{Sub-nearlattices containing some minimal elements of $\nlat\ce$}\label{sectaminbmin}
Although $\nlat\ce$ is a nearlattice, sometimes we need its join-semilattice reduct, $\slat\ce$; for example, in the following lemma. The join-semilattice $\slat{T_3}$ is defined by Figure~\ref{figd4}, where the notations of its elements are also given.

\begin{lemma}\label{LmT3} The join-semilattice $\slat{T_3}$ cannot be a subsemilattice of the join-semilattice $\slat\ce$ so that $a$ and $b$ (the black-filled elements in the figure) are minimal elements in $\slat\ce$.
\end{lemma}

\begin{proof} Suppose contrary. Replacing $b$ by $d\wedge e$ if necessary, we can assume that $d\wedge e=b$. 
By Lemma~\ref{lemmatwzBrqv}; $a$ has two distinct covers, $p$ and $q$; clearly, $p\wedge q=a$. Similarly, $u\wedge v= c$ with distinct covers $u$ and $v$ of $c$.
Observe that 
\begin{equation}
\set{p,q}\cap\set{u,v}=\emptyset,
\label{eqndshbzWvLzW}
\end{equation} 
because a common cover $x$ of $a$ and $c$ would satisfy $x\geq a\vee c=i$ but then $a<d<i\leq x$ would contradict $a\prec x$.  By \eqref{convinDrkt}, we exclude case \azeset1{}: $\set{p,q}\subseteq\ideal d$ and   $\set{u,v}\subseteq\ideal e$ since its \tmany-value is $63.25$. Note, in advance, that \eqref{eqndshbzWvLzW} will also be valid for the rest of cases.
Next, we deal with the case when exactly one of the previous two inequalities hold; by $a$--$c$-symmetry, this is case \azeset2: 
$\set{p,q}\subseteq\ideal d$ but   $\set{u,v}\not\subseteq\ideal e$. Let, say, $u\notin\ideal e$. 
Observe that $u\not>e$ since otherwise $c\prec u$ would not hold. Hence, $u\parallel c$.
We will not use $v$ because it may equal $e$.  Since $c\leq e\wedge u<u$ and $c\prec u$, we have that $e\wedge u=c$. 
If \azeset{2a}: $u\vee e$ is an ``old element'', then $u\vee e=i$ (the only old element larger than $e$), $u<i$, and $\many=73$. Otherwise, if  \azeset{2b}: $u\vee e=:w$ is a new element, then $\many=76.25$. Hence,  \eqref{convinDrkt} excludes \azeset{2}. The next case is \azeset3: $\set{p,q}\not\subseteq\ideal d$ and $\set{u,v}\not\subseteq\ideal e$. Let, say, $p\not\leq d$ and $u\not\leq e$; we will not work with $q$ and $v$. We still have that $e\wedge u=c$ and $u\parallel e$. Similarly, $d\wedge p=a$ and $p\parallel d$. At present, $\set{a,b,c,d,e,p,u,i}$ is the set of old elements; only $i$ from the old elements can be but need not be an upper bound of $\set{d,p}$, and the same holds for $\set{e,u}$. If  both $d\vee p$ and $e\vee u$ are old elements, then \azeset{3a}: $d\vee p=e\vee u=i$ and $\many=81$. 
If only one of the two above-mentioned joins is a new element, then symmetry allows us to assume that \azeset{3b}: $d\vee p=:x$ is new and $e\vee u=i$ is old, and then $\many=81.5$.   Next, if \azeset{3c}: $p\vee d=:x$ and $u\vee e=:y$ are distinct new elements, then either \azeset{3c.1}: $x\vee y=:z$ is a new element and $\many=64.875$, or \azeset{3c.2}: $x\vee y=i$ and $\many=62$. If \azeset{3d}: $p\vee d=u\vee e$, then $d\vee p=e\vee u=:x$ is a new element since  \eqref{convinDrkt} excludes  \azeset{3a},  $x>i$ since $x\geq a\vee c=i$, and $\many=71$. Thus, \eqref{convinDrkt} completes the proof.
\end{proof}

The join-semilattice $\slat{T_4}$ is defined in Figure~\ref{figd4}.

\begin{lemma}\label{LmT4} The join-semilattice $\slat{T_4}$ cannot be a subsemilattice of the join-semilattice $\slat\ce$ so that $a$, $b$ and $c$ (the black-filled elements in the figure) are minimal elements in $\slat\ce$.
\end{lemma}

\begin{proof} Suppose the contrary. By Lemma~\ref{lemmatwzBrqv}, there are elements $p,q, e,f, u,v\in\ce$ such that $p$ and $q$ are distinct covers of $a$, the elements $e$ and $f$ are those of $b$, and $u$ and $v$ are those of $c$. Let us agree that $\set{p,e,u}\cap \set{d,i}=\emptyset$, whereby 
$\set{p,e,u}\cap \set{a,b,c,d,i}=\emptyset$ and we can always use the elements $p$, $e$, and $u$. However, $q$, $f$, and $v$ will be used only if they are distinct from $d$, $d$, and $i$, respectively.  This means that mostly when, say, $v$ occurs in the argument, then $v\neq i$ is assumed even if this is not mentioned again, and similarly for $q$ and $f$, which are mentioned typically when they are distinct from $d$. Note that if, say, both $u$ and $v$ are used ``at $c$'', then the jm-constraint $u\wedge v=c$ is included in the situation, and similarly at $a$ and $b$. The inequalities $a<p$, $b<e$, and $c<u$ are permanent in the situations but if, say, $a<p$ is the only occurrence of $p$, then the element $p$ could be omitted without changing the \tmany-value. 
Now, the parsing tree will be larger than in the previous proofs. Note at this point that whenever a case is excluded, then this always happens by \eqref{convinDrkt} even if  \eqref{convinDrkt} is not referenced. Furthermore, when a new case or subcase begins, all the previous ones have already been excluded even if this is not mentioned.

\eset{1}: $u\leq i $ and $v\leq i$. More precisely, $c$ has two distinct covers in $\ideal i$ and they are denoted by $u$ and $v$. 
Then, since $c\prec u<u\vee v\leq i$, we have that $u<i$ and $v<i$. As we have already mentioned, $u\wedge v=c$.

\eset{1a}: $u\vee v=i$. Now, we have to look at $a$, and then $b$.

\eset{1a.1} $\set{p,q}\subseteq\ideal d$. Then $p\wedge q=a$ and $p\vee q\leq d$. Observe that none of $p$ and $q$ covers $b$, because otherwise this cover of $a$ would be greater than or equal to $d=a\vee b$, which would contradict $a\prec p<d$.
Now if \azeset{1a.1a}: $e\parallel d$, then $b\prec e$ and $b\leq d\wedge e<e$ imply that $d\wedge e=b$ and we have that $\many=76.5$; otherwise \azeset{1a.1b}: $f$ exists, $e\wedge f=b$, and $\many=75.875$. Hence, by \eqref{convinDrkt}, \azeset{1a.1} is excluded.
\eset{1a.2}:  $\set{p,q}\not\subseteq\ideal d$.  That is, $\set{p,q}\notin\ideal d$, and we can chose $p$ outside $\ideal d$ by $p$--$q$ symmetry; $q$ will not be used because we do not know if it is equal to or distinct from $d$. Since $a\prec p$, we obtain easily that $p\wedge d=a$. Observe that $p\vee d$ is either $i$, the only old element larger than $d$, or a new element $x$.
\eset{1a.2a}: $p\vee d=i$.
First, \azeset{1a.2a.1}: $\set{e,f}\subseteq\ideal d$ yields that $e\wedge f=b$ and  $\many=72$. Second, assume that one of $e$ and $f$, say $e$, is not in $\ideal d$, so \azeset{1a.2a.2}: $e\not\leq d$, that is, $d\wedge e=b$. Now \azeset{1a.2a.2a}: $d\vee e=i$ yields that  $\many=81.5$ while \azeset{1a.2a.2b}: $d\vee e=:x$, a new element, leads to $\many=76.75$. Hence, by \eqref{convinDrkt}, \azeset{1a.2a} is excluded.
\eset{1a.2b}: $p\vee d=:x$ is a new element. First, let 
\azeset{1a.2b.1}: $x<i$. Then \azeset{1a.2b.1a}: $\set{e,f}\subseteq \ideal d$ gives that $\many=61.625$ while \azeset{1a.2b.1b}: $\set{e,f}\not\subseteq \ideal d$ and, say, $e\not\leq d$ yields to $d\wedge e=b$ and $\many=76.5$. Hence, by \eqref{convinDrkt}, \azeset{1a.2b.1} is excluded.  Second, let \azeset{1a.2b.2}: $i<x$. Then either \azeset{1a.2b.2a}: $\set{e,f}\subseteq \ideal d$ and $\many=61.125$, or say \azeset{1a.2b.2b}: $e\not\leq d$ and we obtain that $e\wedge d=b$ and $\many=74.5$. By \eqref{convinDrkt}, \azeset{1a.2b.2} is excluded. Third, if \azeset{1a.2b.3}: $x\parallel i$, then $x\vee i=:y$ is a new element and $\many=65.75$. Hence, \azeset{1a.2b}, \azeset{1a.2}, and \azeset{1a} are excluded.

\eset{1b}: $u\vee v=:g<i$ is a new element. 
\eset{1b.1}: $\set{p,q}\leq d$ and $q\leq d$. Then $p\wedge q=a$ and $\many=80$, whence this case is excluded. 

\eset{1b.2}: $\set{p,q}\not\leq d$ and, say, $p\not\leq d$. Then $p\wedge d=a$, and we  continue similarly to \azeset{1a.2}.

\eset{1b.2a}: $d\vee p=i$ (the only possibility that $d\vee p$ is an old element). Then \azeset{1b.2a.1}: $\set{e,f}\subseteq \ideal d$ gives that $e\wedge f=b$ and $\many=64.625$. Otherwise, if $\set{e,f}\not\subseteq \ideal d$, then we can assume that $e\not\leq d$, so \azeset{1b.2a.2}: $d\wedge e=b$, leading to $\many=78.25$. Hence,  \azeset{1b.2a} is excluded.

\eset{1b.2b}: $d\vee p=:x$ is a new element. 
Now \azeset{1b.2b.1}: $x<i$ gives that $\many=76.5$, \azeset{1b.2b.2}: $i<x$ leads to $\many=76$, and  \azeset{1b.2b.3}: $x\parallel i$ yields  that $x\vee i=:y$ is a new element and $\many=58.25$. Hence, \azeset{1b.2b} is excluded, and so are 
 \azeset{1b}, and \azeset{1}. 

\eset{2}: $u\parallel i$. (Apart from $u$--$v$ symmetry, this is the opposite of \azeset{1} even if $v=i$.) Then $i\vee u=:g$ is a new element and $u\wedge i=c$.

\eset{2a}: $\set{p,q}\subseteq\ideal d$. Then $p\wedge q=a$ and $p\vee q\leq d$. As in case \azeset{1a.1}, none of $p$ and $q$ covers $b$. There are two subcase. First, let \azeset{2a.1}: $p\vee q=d$.  If \azeset{2a.1a}: $\set{e,f}\subseteq \ideal d$, then $e\wedge f=b$ and $\many=61.625$. Otherwise (even if $f=d$) we can assume that \azeset{2a.1b}: $d\wedge e=b$, which gives that $\many=72.5$. Hence,  \azeset{2a.1} is excluded. Second, if
\azeset{2a.2}: $p\vee q=:x <d$ is a new element, then $\many=81.25$. Thus,   \azeset{2a} is excluded. 

\eset{2b}:  $\set{p,q}\not\subseteq\ideal d$ and, say, $p\notin\ideal d$. Then $p\wedge d=a$. Depending on $p\vee d$, there are three subcases, because only two of the old elements belong to $\filter p\cap\filter d$. First, let \azeset{2b.1}: $p\vee d=i$.
Then either \azeset{2b.1a}: $\set{e,f}\subseteq\ideal d$, whence $e\wedge f=b$  and $\many=65.5$, or \azeset{2b.1b}: $\set{e,f}\not\subseteq\ideal d$ and, say, $e\notin\ideal d$, whence  $e\wedge d=b$ and $\many=81.5$. Hence, \azeset{2b.1} is excluded.
Second, let \azeset{2b.2}: $p\vee d=g$. Then either \azeset{2b.2a}: $\set{e,f}\subseteq\ideal d$, so  $e\wedge f=b$  and $\many=65$, or \azeset{2b.2b}: $\set{e,f}\not\subseteq\ideal d$ and, say, $e\notin\ideal d$, whereby $e\wedge d=b$ and $\many=79.5$. Hence,  \azeset{2b.2} is excluded. 
Third, if \azeset{2b.3}: $p\vee d=:x$ is a new element, then either \azeset{2b.3a}: $\set{e,f}\subseteq \ideal d$, $e\wedge f=b$, and $\many=61.75$, 
or \azeset{2b.3b}: $d\wedge e=b$ allows only three old elements larger than $d$ and, thus, splits into 
\azeset{2b.3b.1}: $d\vee e=i$ with $\many=72.25$,  \azeset{2b.3b.2}: $d\vee e=x$ with $\many=72.75$, \azeset{2b.3b.3}: $d\vee e=g$ with $\many=70.75$, and \azeset{2b.3b.4}: $d\vee e=:y$, a new element, with $\many=75$. Hence, \azeset{2b.3} is excluded, and so are \azeset{2b}, and \azeset{2}.
Finally, \eqref{convinDrkt} completes the proof of Lemma~\ref{LmT4}.
\end{proof}

\section{Sub-\qn-lattices with entries and anchors}\label{sectanchors}
By Cz\'edli~\cite{czg83}, if $\nlat L$ is a finite \emph{lattice} (not only a nearlattice) that satisfies $\many\nlat L>83$, than $\nlat L$ is planar. From our perspective, this result is equivalent to the following lemma; see Convention~\ref{convinDrkt} for the notation. 

\begin{lemma}[Cz\'edli~\cite{czg83}]\label{lemmaCzG83rm} The nearlattice $\nlat\ce$ has at least two minimal elements. 
\end{lemma}

If $\nlat\ce$ has only two minimal elements, than condition \eqref{eqpbxvnMdvm} below clearly holds.

\begin{definition}\label{defwingeTc} With the assumption that
\begin{equation}\left.
\parbox{8cm}{any two distinct minimal elements of $\nlat\ce$ have the same join, which we denote by $m$,}\right\}
\label{eqpbxvnMdvm}
\end{equation}
we introduce the following notations and concepts. 
The filter $M:=\filter m$ is called the \emph{kernel} of $\nlat\ce$. For a minimal element $\xmin$ of $\nlat\ce$, the set 
$W(\xmin):=\filter\xmin\setminus M$ is the \emph{wing} of $\xmin$. 
An element $v\in M\setminus\set m$ is called an \emph{$\xmin$-entry} if it has a lower cover in the $\xmin$-wing $W(\xmin)$. Let us emphasize that $m$ is not an $\xmin$-entry. If $v$ is an $\xmin$-entry  and $u\in W(\xmin)$ is a lower cover of $v$, then $u$ is an \emph{$\xmin$-anchor} of $v$ and the edge $u\prec v$ is an \emph{$\xmin$-bridge}. By an \emph{entry}, an \emph{anchor}, or a \emph{bridge}, we mean an $\xmin$-entry, an $\xmin$-anchor, or an $\xmin$-bridge, respectively, for some minimal element~$\xmin$.
\end{definition}

\begin{lemma}\label{lemmaLPVblvjDlsJzB} Assuming \eqref{eqpbxvnMdvm},  let $\xmin$ be  a  minimal element of  $\nlat\ce$. Then every wing is an order ideal (also known as a down-set). 
The wing $W(\xmin)$  is the set of all $y\in\ce$ such that $\xmin$ is the only minimal element of the principal ideal $\ideal y$. 
If $\ymin$ is also a minimal element and $\ymin\neq \xmin$, then $u\parallel v$ holds for all $u\in W(\xmin)$ and $v\in W(\ymin)$; in particular, then $W(\xmin)\cap W(\ymin)=\emptyset$ and no $\xmin$-anchor is a $\ymin$-anchor.
Finally, an element of $\ce$ is either in the kernel $M$, or it is in a uniquely determined wing $W(\xmin)$.
\end{lemma}

\begin{proof} Since is $M$ is a filter also in $\filter\xmin$, it is evident that $W(\xmin):=\filter\xmin\setminus M$ is an order ideal. The second part follows from the fact every element above two distinct minimal elements are in the kernel $\filter m=M$. If, in spite of the assumptions, $u$ and $v$ are comparable, say, $u\leq v$, then  $\xmin\leq u \leq v$ would lead to
$\xmin\in\ideal v\subseteq W(\ymin)$ and so $\ymin\leq \xmin$, a contradiction. Finally, if $w\in \ce$, then take a minimal element  $\xmin$ of $\ce$ such that $\xmin\leq w$. If $\xmin$ is the only such element, then $w\in W(\xmin)$. Otherwise, there exists a minimal element $\ymin\in\ideal w\setminus\set\xmin$, and $w\geq \xmin\vee \ymin=m$ yields that $w\in M$. 
\end{proof}

\begin{lemma}\label{lemmaLthVmNLtmlmH}
Assuming \eqref{eqpbxvnMdvm}, let $\xmin$ be a minimal element of  $\nlat\ce$, and let $i$ be a minimal $\xmin$-entry with an $\xmin$-anchor $c$. Then for every $y\in \ce$, if $m\leq y<i$, then  $c\wedge y=c\wedge m$.
\end{lemma}

\begin{proof} Let $y=u_0\succ u_1\succ u_2\succ \dots \succ u_n=c\wedge y$ be a maximal chain in the interval $[c\wedge y, y]$. Since wings are order ideals and $c\in W(\xmin)$, we have that $c\wedge y\in W(\xmin)$. So, $c\wedge y=u_n\notin M$ but $u_0=y\in M$. Hence, there is a least subscript $t\in\set{0,\dots, n-1}$ such that $u_t\in M$ but $u_{t+1}\notin M$. Since $u_{t+1}\geq u_n=c\wedge y\in W(\xmin)\subseteq \filter\xmin$, we have that 
$u_{t+1}\in \filter\xmin$. Thus, $u_{t+1}\in \filter\xmin\setminus M=W(\xmin)$. Now if $u_t\neq m$, then $u_t$ is an $\xmin$-entry with $\xmin$-anchor $u_{t+1}$, but this is impossibly since $u_t\leq y<i$ and $i$ is a minimal $\xmin$-entry. Hence, $u_t=m$ and so $c\wedge y=u_n\leq u_t=m$, that is, $c\wedge y \leq m$. Using this inequality and $m\leq y$, we obtain that 
$c\wedge y=(c\wedge y)\wedge m=c\wedge (y\wedge m)=c\wedge m$, as required.
\end{proof}

The general assumption for the rest of this  section, which is stronger than \eqref{eqpbxvnMdvm}, is the following.
\begin{equation}
\parbox{9.8cm}{$\nlat\ce$ has exactly two minimal elements, $\amin$ and $\bmin$; the notations and concepts given in Definition~\ref{defwingeTc} will be in effect.}
\label{eqpbxmDnvrTlm}
\end{equation}
Note that in the input and output files, we write $A$ and $B$ instead of $\amin$ and $\bmin$, respectively. 
Consider the \qn-lattice  $\nlat{T_5}$ determined by Figure~\ref{figd5} and 
\begin{equation}
W^\ast:=\{a\vee b=m,\,\,  c\vee  m=i,\,\,  d\vee m =i,\,\, 
e\vee  m=j\}.
\label{eqWasterix}
\end{equation} 

\begin{lemma}\label{LmT5} Assuming \eqref{eqpbxmDnvrTlm}, the \qn-lattice  $\nlat{T_5}$ cannot be a sub-\qn-lattice of $\nlat\ce$ so that $a=\amin$, $b=\bmin$, and the thick edges are bridges.
\end{lemma}

\begin{remark}\label{remfTrzGbmcvbdTm}
Observe that an equivalent assumption for Lemma~\ref{LmT5} is that  $a=\amin$, $b=\bmin$, and the thick edges correspond to \emph{covering pairs} (that is, \emph{edges}) in the nearlattice $\nlat\ce$. Indeed, then the jm-constraint $a\vee b=m$ guarantees that $m\in T_5$ is $m\in \ce$, and since the order embedding $\poset{T_5}\to\poset{\ce}$ preserves incomparability, it follow easily that the covering pairs corresponding to the thick edges are bridges in $\nlat\ce$. Conversely, we know from Definition~\ref{defwingeTc} that bridges are covering pairs. Without separate mentioning,
analogous observations hold for the rest of the lemmas where bridges are mentioned. 
\end{remark}

\begin{proof}[Proof of Lemma~\ref{LmT5}] Suppose the contrary. If $c\wedge m>\amin$, then we modify $a$ to $c\wedge m=:a$. 
Similarly,  $d\wedge m=:b$. 
Since thick edges correspond to bridges and, thus, to coverings, see Remark~\ref{remfTrzGbmcvbdTm}, we obtain easily that $e\wedge i=c$ since $c\leq e\wedge i<i$. So, from now on, the equalities from the set 
$W_5:=\{c\wedge m=a,\,\,  d\wedge m=b,\,\, e\wedge i=c\}$ hold. Therefore, we can assume that $T_5$ is determined by the set $W^\ast\cup W_5$ of constraints. However, we will occasionally rely on the equality $\amin\vee\bmin=m$ and the inequalities provided $\amin\leq a$ and $\bmin\leq b$. 
First, we show that
\begin{equation}\left.
\parbox{9.0cm}{for every $x\in \ce$ such that $x\notin T_5=\set{a,b,c,d,e,i,j,m}$, we have that $x<a$, or $x<b$, or $x>m$.}
\right\}
\label{eqtxtxszbNNM}
\end{equation}
Suppose the contrary and take an $x$ that violates \eqref{eqtxtxszbNNM}; then $x\notin T_5$, $x\not< a$, $x\not< b$, and $x\not> m$.  First, assume that \azeset{1}: $x\parallel a$. 
If \azeset{1a}: $x\geq \amin$, then $x\wedge a=:y\geq\amin$ is a new element,  $m=\amin\vee\bmin\leq y\vee b\leq a\vee b\leq m$ gives that $y\vee b=m$, and  $\many=66.75$  excludes \azeset{1a} by \eqref{convinDrkt}. 
Second, assume that \azeset{1b}: $x\not\geq \amin$. Then $x\geq \bmin$ since there are only two minimal elements, whence $a\vee x\geq a\vee \bmin =m$. Since $\filter m$ contains only three old elements, we have that \azeset{1b.1}: $a\vee x=m$ and $\many=74.5$, or \azeset{1b.2}: $a\vee x=i$ and $\many=74.5$, or \azeset{1b.3}: $a\vee x=j$ and $\many=70$, or \azeset{1b.4}: $a\vee x=:y$ is a new element and $\many=76.25$. Hence, by \eqref{convinDrkt}, \azeset{1b} and \azeset{1} are excluded.
Therefore, since $x\not<a$ and $x\notin T_5$ have been assumed,  $x>a$.  Second, assume that \azeset{2}: $x\parallel b$ (in addition to $x>a$), and take $b\vee x$. Since $b\vee x\geq b\vee a=m$ and $|T_5\cap\filter m|=3$, there are only four possibilities for $b\vee x$; namely: \azeset{2a}: $b\vee x=m$, \azeset{2b}: $b\vee x=i$, \azeset{2c}: $b\vee x=j$, and 
\azeset{2d}: $b\vee x=:y$ is a new element.
Since their \tmany-values are $70.5$, 
$76.5$, $72$, and $77.25$, respectively, \eqref{convinDrkt} excludes \azeset{2}. Combining this fact with  $x\notin T_5$ and $x\not<b$, we obtain that $x>b$. But then $x\geq a\vee b=m $, $x\not>m$ and $x\notin T_5$ form a contradiction, which  proves \eqref{eqtxtxszbNNM}.

Clearly, for every $y\in E:=\ideal a\cup\ideal b\cup[i,j]\cup \filter j$, the poset $\poset{T_5\cup\set x}$ is planar.
Therefore, we can pick an element $y\in \ce\setminus(T_5\cup E)$; it follows from  \eqref{eqtxtxszbNNM} that $m<y$. If \azeset{3} $y\parallel j$, then we let $y\vee j=:z$ to  obtain that $\many=73.5$, which is impossible. This fact combined with $y\notin \filter j$ gives that $y<j$. Since $y\notin[i,j]$ but $y\in [m,j]$, we obtain easily that $y\parallel c$, $y\parallel d$, and $y\parallel e$; for example, $c\leq y$ would lead to $i=c\vee m\leq y\leq j$, which would contradict $y\notin [i,j]$. 
By \eqref{eqtxtxszbNNM}, if $e\wedge y\notin T_5$, then either $e\wedge y<a$, contradicting $e,y\in\filter a$, or $e\wedge y<b$, leading to
the contradiction $a\leq e\wedge y<b$, or $e\wedge y>m$, contradicting $e\not\geq m$. Hence $e\wedge y$ is in $T_5\cap \ideal e$, but it is neither $e$, nor $c$ since $e\parallel y$ 
and  $c\parallel y$. Similarly, $c\wedge y$ is in $T_5\cap \ideal c$ but $c\wedge y\neq c$ and  $d\wedge y$ is in $T_5\cap \ideal d$ but $d\wedge y\neq d$. Hence, \azeset{4}: $c\wedge y=a$, $d\wedge y=b$, $e\wedge y=a$, and $\many=70$ completes the proof by \eqref{convinDrkt}.
\end{proof}

Consider the \qn-lattice  $\nlat{T_6}$ determined by 
$W^\ast$ from \eqref{eqWasterix} and Figure~\ref{figd5}.

\begin{lemma}\label{LmT6} 
 Assuming \eqref{eqpbxmDnvrTlm}, the \qn-lattice  $\nlat{T_6}$ cannot be a sub-\qn-lattice of $\nlat\ce$ so that $a=\amin$, $b=\bmin$,  the thick edges are bridges, and $i$ is a minimal $\bmin$-entry.
\end{lemma}

\begin{proof} Suppose the contrary. The thick edges are coverings in $\nlat\ce$, whence it follows that $c\wedge j=e$. 
We can replace $a=\amin$ and $b=\bmin$ by $e\wedge m$ and $d\wedge m$, respectively.  Then $d\wedge j=b$ by Lemma~\ref{lemmaLthVmNLtmlmH}.  Hence, after letting $W_6':=\set{e\wedge m=a,\,\,d\wedge m=b,\,\, d\wedge j=b,\,\, c\wedge j=e}$,  we can assume that $\nlat{T_6}$ is determined by $W^\ast \cup W_6'$; see \eqref{eqWasterix}. Then $\many\nlat{T_6}=80$, and \eqref{convinDrkt} applies.
\end{proof}

The following remark will be useful in the proofs of several statements.

\begin{remark}\label{remdelEdge}
Assume that $(L;\leq_1)$ and $(L;\leq_2)$ are posets such that for every $x,y\in L$, if $x\leq_2 y$, then $x\leq_1 y$. Let $W$ be a set of jm-constraints over $L$. For $i\in\set{1,2}$, let $(L;\leq_i,\vee_i,\wedge_i)$ be the \qn-lattice determined by $W$ and (the diagram of) $(L;\leq_i)$. Then $\many(L;\leq_1,\vee_1,\wedge_1) \leq \many(L;\leq_2,\vee_2,\wedge_2)$. In particular, if we remove an edge from a diagram, then the \tmany-value of the situation (determined by a given set of jm-constraints) increases, provided we do not make use of some new incomparability $x\parallel_2 y$ (that is, $x\parallel_2 y$ such that $x\leq_1 y$ or $y\leq_1 x$).  
\end{remark}

The statement of this remark is trivial: if we have less comparable pairs, then the axioms (A1)--(A5) from Definition~\ref{defjpm} bring less jm-constraints in, so more subsets will obey the jm-constraints, whence \tmany{} increases. As we know from \eqref{convinDrkt}, our permanent intention is show that the \tmany-values are small enough. Hence, the practical value of Remark~\ref{remdelEdge} is the following. 
\begin{equation}\left.
\parbox{10.5cm}{In our proofs, we can delete any edge of a diagram provided we will not use the new incomparability and we still can show that $\many\leq 83$.}
\right\}
\label{eqpbxdelEDg}
\end{equation}
Of course, if the \tmany-value becomes too large after deleting an edge $x\prec y$ (or $x<y$), then our attention turns to the new join $x\vee y$, which is either an old element, or a new one, and even to $x\wedge y$ if it exists.

\begin{figure}[htb] 
\centerline
{\includegraphics[scale=1.0]{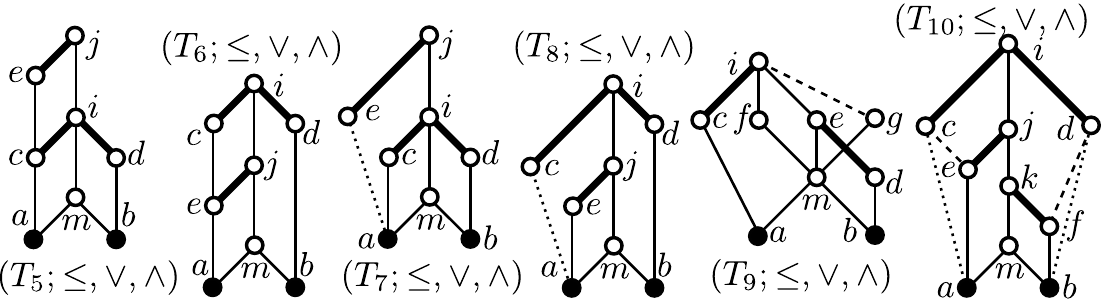}}
\caption{The \qn-lattices $\nlat{T_\ell}$,  $\,\,\ell\in\set{5,6,\dots,10}$
\label{figd5}}
\end{figure}%

Consider the \qn-lattice  $\nlat{T_7}$ determined by 
$W^\ast$ from \eqref{eqWasterix} and Figure~\ref{figd5}; the dotted edge is taken into account.

\begin{lemma}\label{LmT7} 
 Assuming \eqref{eqpbxmDnvrTlm}, the \qn-lattice  $\nlat{T_7}$ cannot be a sub-\qn-lattice of $\nlat\ce$ so that $a=\amin$, $b=\bmin$, and  the thick edges are bridges.
\end{lemma}

\begin{proof} Suppose the contrary.  
We can replace $a=\amin$ and $b=\bmin$ by $c\wedge m$ and $d\wedge m$, respectively. Note that now $a\leq e$ may fail; the purpose of the dotted edge is to remind us of this fact.  Since $c\parallel e$, from $e<c\vee e \leq j$ and $e\prec j$ we obtain that $c\vee e=j$. 
Similarly, we obtain that $b\vee e=j$.  
Hence, after deleting the dotted edge from the diagram and letting
$W_7':=\set{c\wedge m=a,\,\,d\wedge m=b,\,\, c\vee e=j,\,\, b\vee e=j}$,  we can assume that $\nlat{T_7}$ is determined by $W^\ast \cup W_7'$; see \eqref{eqWasterix} and \eqref{eqpbxdelEDg}. Clearly, $e\not<a$ since otherwise we would get that $e\leq i$ by transitivity. Hence, based on the relation between $a$ and $e$, 
there are two main cases. 
First, if \azeset{1}: $a\parallel e$, then $a$ and the $\amin$-anchor $e$ are in $\filter\amin$, so we can add their meet $a\wedge e=:x$ to $T_7$, we have that $b\vee x=m$ (since $m=\bmin \vee \amin\leq b\vee x\leq b\vee a\leq m$) and $a\vee e=j$ (since $e<a\vee e\leq j$ and $e\prec j$), 
and we obtain that $\many=79.5$. Hence, \eqref{convinDrkt} excludes \azeset{1}.  

Second, we assume that \azeset{2}: $a<e$. Depending on $c\wedge e$, there are two subcases. 
We begin with  \azeset{2a}: $c\wedge e=a$, which splits again depending on $e\wedge i$. 
The analysis for \azeset{2a.1}: $e\wedge i=a$ runs as follows. 
Since the situation at present describes a planar \qn-lattice and $\poset{T_5\cup\set{y}}$ remains planar for every $y\in\filter j$, there exists an additional element $y\in \ce\setminus (T_\cup\filter j)$. So $y\parallel j$ or $y<j$. 
If \azeset{2a.1a}: $y\parallel j$, then letting $j\vee y=:z$, $\many=71.5$. If \azeset{2a.1b}: $y<j$, then, using that $e\prec j$,  either \azeset{2a.1b.1}: $y\parallel e$ and so $e\vee y=j$ with $\many=83$, or \azeset{2a.1b.2}: $y<e$, when $y\geq \amin$ since $\bmin\not\leq e\in W(\amin)$. So \azeset{2a.1b.2a}: $a<y$ with $\many=70$, or \azeset{2a.1b.2b}: $\amin\leq y<a$ and so $y\vee b=m$ and $\many=75.5$, or \azeset{2a.1b.2c}: $y\parallel a$ and $\amin\leq a\wedge y=:u$ with $u\vee b=m$ and $\many= 57.75$. (Note that with $u$ playing the role of $y$, even \azeset{2a.1b.2b} excludes \azeset{2a.1b.2c}.) At this stage, \eqref{convinDrkt} excludes \azeset{2a.1}.
Since \azeset{2a.2}: $e\wedge i=:v>a$ is also excluded by its \tmany-value $76.5$, \azeset{2a} is impossible.  So is \azeset{2b}: $c\wedge e=:x>a$ by its $\many=82$. By \eqref{convinDrkt}, 
the proof of Lemma~\ref{LmT7} is complete.
\end{proof}

Consider the \qn-lattice  $\nlat{T_8}$ determined by 
$W^\ast$ from \eqref{eqWasterix} and Figure~\ref{figd5}.

\begin{lemma}\label{LmT8} 
 Assuming \eqref{eqpbxmDnvrTlm}, the \qn-lattice  $\nlat{T_8}$ cannot be a sub-\qn-lattice of $\nlat\ce$ so that $a=\amin$, $b=\bmin$,  the thick edges are bridges, and $i$ is a minimal $\bmin$-entry.
\end{lemma}

\begin{proof} Suppose the contrary.  As usual,
we can replace $a=\amin$ and $b=\bmin$ by $e\wedge m$ and $d\wedge m$, respectively. But note then $a\leq c$ may fail; this is what the dotted edge indicates. However, $c\not\leq a$ since otherwise we would obtain that $c\leq e$. By Lemma~\ref{lemmaLthVmNLtmlmH}, $d\wedge j=b$.
Using that $c\prec i$, $c\parallel e$, $c\parallel b$, and $\set{c,e,b}\subseteq \ideal i$, we obtain easily that $c\vee e=i$ and $c\vee b=i$. Therefore, after letting
$W_8':=\set{e\wedge m=a,\,\,d\wedge m=b,\,\, d\wedge j=b,\,\,  c\vee e=i,\,\, c\vee b=i}$ and removing the dotted edge from the diagram, see \eqref{eqpbxdelEDg}, we can assume that $\nlat{T_8}$ is determined by $W^\ast \cup W_8'$; see \eqref{eqWasterix}. If \azeset{1}: $c\parallel a$, then letting $c\wedge a=:x\geq \amin$ and adding $x$ to $T_8$, we obtain $\many=83$, which is excluded by \eqref{convinDrkt}. So let \azeset{2}: $a<c$. 
Then either \azeset{2a}: $c\wedge j=a$, the only element of $T_8\cap\ideal c\cap \ideal j$, and $\many=77$, 
or  \azeset{2b}: $c\wedge j=:x>a$ and $\many=81.5$. Thus,  \eqref{convinDrkt} applies and completes the proof.
\end{proof}

Next, let $\nlat{T_9}$ be the \qn-lattice determined by the set 
\begin{align*}
\{a\vee b=m,\,\ c\vee m=i,\,\, d\vee m=e\}
\end{align*}
of jm-constraints and Figure~\ref{figd5}; see Definition~\ref{defcjmcRsrtS}\eqref{defcjmcRsrtSe} about the dashed edge.

\begin{lemma}\label{LmT9}  Assuming \eqref{eqpbxmDnvrTlm},  the
\qn-lattice  $\nlat{T_9}$ cannot be a sub-\qn-lattice of $\nlat\ce$ so that  $a=\amin$,  $b=\bmin$, and the two thick edges are bridges. 
\end{lemma}

\begin{proof}  Suppose the contrary.  As usual,
 see the proof of Lemma~\ref{LmT5} or that of Lemma~\ref{LmT8},
we can assume that  $c\wedge m=a\geq\amin$ and $d\wedge m=b\geq \bmin$. Since $c\prec i$ and $d\prec e$, see Remark~\ref{remfTrzGbmcvbdTm}, we also have that $b\vee c=i$ and $a\vee d=e$. The parsing runs as follows.

\eset{1}: $f\vee e=i$. There are two subcases. First, if \azeset{1a}: $f\wedge e=m$,  then either \azeset{1a.1}: $e\wedge g=m$ and  $\many=79.5$, 
or \azeset{1a.2}: $e\wedge g=:x>m$ and $\many=70$. 
Second, if \azeset{1b}: $f\wedge e=:y>m$, then either \azeset{1b.1}: $c\wedge f=a$ and $\many=79.5$, or \azeset{1b.2}: $c\wedge f=:z>a$ (where $z$ is distinct from $y$ because otherwise we would obtain that $c\geq z=y > m$) and $\many=74.5$. Hence, \azeset{1} is excluded by \eqref{convinDrkt}.

\eset{2}: $f\vee e=:p<i$. Again, there are two subcases.
First, if \azeset{2a}: $e\wedge f=m$, then either \azeset{2a.1}: $f\wedge g=m$ and $\many=81.75$, or \azeset{2a.2}: $f\wedge g=:q>m$ (but $q\neq p$ since $q<f<p$) and $\many=69$. Hence, \azeset{2a} is excluded. Second, if \azeset{2b}: $e\wedge f=:x>m$, then either \azeset{2b.1}: $c\wedge f=a$ and $\many=77.75$,  or \azeset{2b.2} $c\wedge f=y>a$, which is distinct from $x$ since otherwise $c>y=x>m$ would contradict $c\parallel m$, and $\many=72.625$. Hence, \azeset{2b} and  \azeset{2} are excluded, and  \eqref{convinDrkt} completes the proof of Lemma~\ref{LmT9}.
\end{proof}

Next, we consider the \qn-lattice $\nlat{T_{10}}$ determined by the set $W^{\ast+}:=W^\ast\cup \set{m\vee f=k}$ of jm-constraints, see \eqref{eqWasterix}, and the diagram given in Figure~\ref{figd5}; the diagram is understood as follows. The dotted edges stand for $a<c$ and $b<d$, but they are not necessarily coverings and, later in the proof, we can drop $a<c$ and $b<d$. According to Definition~\ref{defcjmcRsrtS}\eqref{defcjmcRsrtSe}, the dashed edges mean that either $e < c$, or $e\parallel c$, and similarly, either $f<d$, or $f\parallel d$. Since there are four possibilities to choose the actual meanings of the dashed edges, we have defined \emph{four distinct versions} of $\nlat{T_{10}}$; the following lemma is stated for all of them.

\begin{lemma}\label{LmT10} 
 Assuming \eqref{eqpbxmDnvrTlm}, no version of the \qn-lattice  $\nlat{T_{10}}$ can be a sub-\qn-lattice of $\nlat\ce$ so that $a=\amin$, $b=\bmin$,  the thick edges are bridges, $i$ is a minimal common $\amin$-and-$\bmin$-entry, $j$ is a minimal $\amin$-entry, and $k$ is a minimal $\bmin$-entry.
\end{lemma}

\begin{proof} Suppose the contrary.  As usual,
 see the proof of Lemma~\ref{LmT5} or that of Lemma~\ref{LmT8},
we can assume that  $e\wedge m=a\geq\amin$ and $f\wedge m=b\geq \bmin$. Then, of course, the two dotted edges in the figure need not mean comparability. By Lemma~\ref{lemmaLthVmNLtmlmH}, $e\wedge k=e\wedge m=a$. Since the upper two thick edges stand for coverings, $b\vee c=i$ and $a\vee d=i$. (Alternatively,
$b\vee c =b\vee (\amin\vee c)= (b\vee \amin)\vee c= m\vee c=i$ and similarly for $a\vee d$.) 
Therefore, after letting
$W_{10}':=\set{e\wedge m=a,\,\,f\wedge m=b,\,\, e\wedge k=a,\,\,b\vee c=i,\,\,a\vee d=i}$, we can assume that $\nlat{T_{10}}$ is determined by $W^{\ast+} \cup W_{10}'$. 
Even if none of the dotted and dashed edges is considered, $\many=72.25$. (Otherwise, if some of these edges are also considered, $\many$ can be slightly smaller; see Remark~\ref{remdelEdge} and \eqref{eqpbxdelEDg}; for example, $\many=72$ if all these four edges are taken into account.) Finally, \eqref{convinDrkt} completes the proof of Lemma~\ref{LmT10}.
\end{proof}

\begin{figure}[htb] 
\centerline
{\includegraphics[scale=1.0]{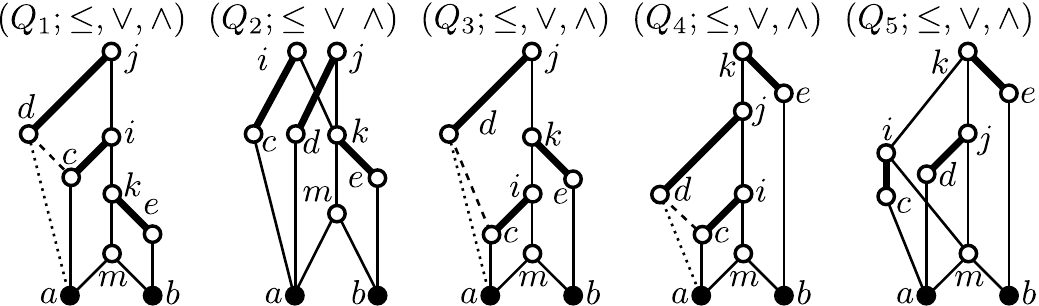}}
\caption{The \qn-lattices $\nlat{Q_\ell}$,  $\,\,\ell\in\set{1,\dots, 5}$
\label{figd6}}
\end{figure}%

Next, let $Q_1$ be the \qn-lattice determined by Figure~\ref{figd6} and the set 
\begin{equation}
W_Q:=
\{ 
a\vee b=m,\,\, c\vee m=i,\,\, d\vee m=j,\,\, e\vee m=k  \}
\label{eqWQ}
\end{equation}
of jm-constraints; note that even if the dashed edge is removed from the figure, we still assume that $a<d$. 
The general assumption for the following five lemmas is that 
\begin{equation}\left.
\parbox{7.2cm}
{the thick edges are bridges,  $a$ and $b$ are the minimal elements $\amin$ and $\bmin$ of $\nlat\ce$, $i$ is a minimal $\amin$-entry, and $k$ is a minimal $\bmin$-entry.}
\right\}
\label{eqpbxkvkshCsZtlMm}
\end{equation}

\begin{lemma}\label{LmQ1} If the nearlattice $\nlat\ce$ has exactly two minimal elements, then the
\qn-lattice  $\nlat{Q_1}$ cannot be its sub-\qn-lattice so that 
\eqref{eqpbxkvkshCsZtlMm} holds.
\end{lemma}

\begin{proof}
Suppose the contrary. As usual, 
 see the proof of Lemma~\ref{LmT5} or that of Lemma~\ref{LmT8},
 after replacing $\amin$ and $\bmin$ by appropriate meets if necessary, we assume that  $c\wedge m=a$ and $e\wedge m=b$; however, then  $a< d$ need not hold. 
Lemma~\ref{lemmaLthVmNLtmlmH} allows us to assume that $c\wedge k=a$. Finally, using that $d\prec j$ (or that $d\geq \amin$), we obtain that $b\vee d=j$. 
Hence, the initial set of jm-constraints is 
\[W:=W_Q\cup\set{c\wedge m=a,\,\, e\wedge m=b,\,\, c\wedge k=a,\,\, b\vee d=j}.
\]
Observe that $d<a$ is impossible since it would lead to $d<m$ by transitivity. Hence, based on the relation between $a$ and $b$, there are only two cases. 
First, if \azeset{1}: $a<d$, then either \azeset{1a}: $d\wedge k=a$, the only element in $Q_1\cap\ideal d\cap\ideal k$,  and $\many=77$, or \azeset{1b}: $d\wedge k=:x>a$ is a new element and $\many=71.75$.  Second, if \azeset{2}: $a\parallel d$, then $a\wedge d=:y \geq \amin$, $y\vee b=m$, and $\many=69.5$. 
Thus, \eqref{convinDrkt} completes the proof. 
\end{proof}

Next, let us consider the \qn-lattice $\nlat{Q_2}$ determined by Figure~\ref{figd6} and $W_Q$ given in \eqref{eqWQ}.

\begin{lemma}\label{LmQ2} If the nearlattice $\nlat\ce$ has exactly two minimal elements, then the
\qn-lattice  $\nlat{Q_2}$ cannot be its sub-\qn-lattice so that
\eqref{eqpbxkvkshCsZtlMm} holds.
\end{lemma}

\begin{proof} Suppose the contrary. Let $v=:i\vee j$.
By the same reasons as in the proof of Lemma~\ref{LmQ1}, we can work with the initial set
\[W:=W_Q\cup\set{c\wedge m=a,\,\, e\wedge m=b,\,\, c\wedge k=a,\,\, i\vee j=v,\,\, d\vee b=j}
\]
again; we have to drop the assumption that $a<d$ but we add that $i<v$ and $j<v$. Then $\many=69.5$ completes the proof.
\end{proof}

The next \qn-lattice is $\nlat{Q_3}$, determined by Figure~\ref{figd6} and $W_Q$; see \eqref{eqWQ}.

\begin{lemma}\label{LmQ3} If the nearlattice $\nlat\ce$ has exactly two minimal elements, then the
\qn-lattice  $\nlat{Q_3}$ cannot be its sub-\qn-lattice so that 
\eqref{eqpbxkvkshCsZtlMm} holds.
\end{lemma}

\begin{proof} Suppose the contrary. Compared to the previous two proofs, the only difference is that now $W:=W_Q\cup\set{c\wedge m=a,\,\, e\wedge m=b,\,\, e\wedge i=b,\,\, b\vee d=j}$. Since $d,i\in\filter\amin$, the meet $d\wedge i$ exists. Using that $Q_3\cap \ideal d\cap \ideal i \subseteq\set{c,i}$, \azeset{1}: $d\wedge i=c$ with $\many=75$, \azeset{2}: $d\wedge i=a$ with $\many=74$, and \azeset{3}: $d\wedge i=x$ with $\many=78.75$ are the only cases. Hence  \eqref{convinDrkt} applies.
\end{proof}

For the \qn-lattice  $\nlat{Q_4}$ determined by $W_Q$ from \eqref{eqWQ} and 
Figure~\ref{figd6}, we have the following lemma.

\begin{lemma}\label{LmQ4} If the nearlattice $\nlat\ce$ has exactly two minimal elements, then the
\qn-lattice  $\nlat{Q_4}$ cannot be its sub-\qn-lattice so that 
\eqref{eqpbxkvkshCsZtlMm} holds.
\end{lemma}

\begin{proof}
Suppose the contrary. 
As in the proofs of  Lemmas~\ref{LmQ1}--\ref{LmQ3}, we can work with  $W:=W_Q\cup\set{c\wedge m=a,\,\, e\wedge m=b,\,\, b\vee d=j,\,\, e\wedge i=b,\,\, e\wedge j=b}$,
and we drop the assumption that $a<d$. If \azeset{1}: $d\wedge m=a$, then $\many=79$. If \azeset{2}: $d\wedge m=:x\geq \amin$ is a new element, then $x\vee b=m$ and $\many=70.5$. Hence, \eqref{convinDrkt} completes the proof.
\end{proof}

Finally, for the \qn-lattice  $\nlat{Q_5}$ determined by $W_Q$ from \eqref{eqWQ} and 
Figure~\ref{figd6}, we have the following lemma.

\begin{lemma}\label{LmQ5} If the nearlattice $\nlat\ce$ has exactly two minimal elements, then the
\qn-lattice  $\nlat{Q_5}$ cannot be its sub-\qn-lattice so that 
\eqref{eqpbxkvkshCsZtlMm} holds.
\end{lemma}

\begin{proof}
Suppose the contrary. 
As in the proofs of  Lemmas~\ref{LmQ1}--\ref{LmQ4}, we can work with  $W:=W_Q\cup\set{c\wedge m=a,\,\, e\wedge m=b,\,\, b\vee d=j,\,\, e\wedge i=b,\,\, e\wedge j=b}$,
and we drop the assumption that $a<d$. Since \azeset{1}: $c\wedge d=a$ and \azeset{2}: $c\wedge d=:x$, a new element, give that $\many=82$ and $\many=81.5$, respectively, \eqref{convinDrkt} applies.
\end{proof}

\section{Our lemmas at work}\label{sectlemmasatwork}
Armed with the lemmas proved so far, we are ready to set off to prove our result.

\begin{proof}[Proof of Theorem~\ref{thmreform}]
It has been proved in Cz\'edli~\cite{czg83} that for every $n\geq 9$, there exists an $n$-element nonplanar lattice $\nlat L$, which is also a nearlattice, such that $\many\nlat L=83$. Hence, the second part of the theorem needs no proof here. We also know from \cite{czg83} that whenever $\nlat L$ is a finite \emph{lattice} with $\many\nlat L>83$, then this lattice is planar. Therefore, it suffices to prove the first part of the theorem only for nearlattices that are not lattices. For the sake of contradiction, suppose that the theorem fails. Thus, Convention~\ref{convinDrkt} will be in effect. So the final target is to show that $\nlat\ce$ does not exist.

First, assume that  $\nlat\ce$ has at least three minimal elements, and pick three distinct minimal elements, $\amin$, $\bmin$, and $\cmin$. Let
\[k:=|\set{\set{x,y}: x\vee y=\amin\vee\bmin\vee\cmin\text{ and }\set{x,y}\subseteq\set{\amin,\bmin,\cmin}}|.
\]
It follows from Lemma~\ref{LmT2} that $k\neq 0$. Similarly, Lemmas~\ref{LmT3} and \ref{LmT4} yield that $k\notin\set{1,2}$. Hence, $k=3$; this means that for any three distinct minimal elements,  
\begin{equation}
\amin\vee \bmin=\amin\vee\cmin=\bmin\vee \cmin=\amin\vee\bmin\vee\cmin.
\label{eqntwrgnZ}
\end{equation} 
Now let $\dmin$ be an additional minimal element. Applying \eqref{eqntwrgnZ} first to the triplet $(\amin,\bmin,\dmin)$, and then to $(\bmin,\cmin,\dmin)$, we have that $\amin\vee\bmin=\bmin\vee\dmin=\cmin\vee\dmin$, and it follows that any two distinct minimal elements of $\nlat\ce$ have the same join. Therefore, from now on, we can and we will rely on Definition~\ref{defwingeTc}.
Next, we claim that 
\begin{equation}
\text{for every minimal element $\xmin\in\ce$, there exists an $\xmin$-entry.}
\label{eqtxtVnxMntrlpsmb}
\end{equation} 
Suppose the contrary. Then $W(\xmin)\cup\set m$ is the interval $[\xmin,m]$, and it is planar by Lemma~\ref{lemmaproperplanar}. Observe that now the  subset $\ce\setminus W(\xmin)$ is a subnearlattice. In order to see this, let $u,v\in \ce\setminus W(\xmin)$; then $u\geq \ymin$ and $v\geq \zmin$ for some minimal elements $\ymin$ and $\zmin$ distinct from $\xmin$. The intersection $u\wedge v$ (if exists) cannot belong to $W(\xmin)$, because otherwise $\xmin\leq u\wedge v\leq u$ and $\ymin\leq u$ would give that $m=\xmin\vee \ymin\leq u$, we would obtain $m\leq v$ similarly, whence $m\leq u\wedge v$ would imply that $u\wedge v$ is outside $W(\xmin)$. Since $W(\xmin)$ is an order ideal by Lemma~\ref{lemmaLPVblvjDlsJzB}, $u\vee w$ is clearly outside $W(\xmin)$, and we conclude that $\ce\setminus W(\xmin)$ is, indeed, a subnearlattice. As a proper subnearlattice, it is planar, again by Lemma~\ref{lemmaproperplanar}. With $\nlat{\ce\setminus W(\xmin)}$, $\nlat{W(\xmin)\cup\set m}$, and $m$ playing the role of $\nlat L$, $\nlat K$ and $u$, Corollary~\ref{corolnarrow} is applicable and implies that $\nlat\ce$ is planar, which is a contradiction proving \eqref{eqtxtVnxMntrlpsmb}. 

Clearly, for every minimal element $\xmin$ of $\nlat\ce$,
\begin{equation}
\text{if $v$ is an $\xmin$-entry and $u$ is an $\xmin$-anchor of $v$, then $u\vee m=v$,}
\label{eqtxtgGvngzlfnnH}
\end{equation}
since this follows from $u\prec v$, $m\not\leq u$, and $u<u\vee m\leq v$. 
Next, we claim that
\begin{equation}
\left.
\parbox{10cm}{if $\amin$ and $\bmin$ are distinct minimal elements of $\nlat\ce$, $e$ is an $\amin$-entry, and $f$ is a $\bmin$-entry, then $e$ and $f$ are comparable elements.}\right\}
\label{eqpbxpRnTrSnL}
\end{equation}
For the sake of contradiction, suppose that \eqref{eqpbxpRnTrSnL} fails. Let $c$ and $d$ be an $\amin$-anchor of $e$ and a $\bmin$-anchor of $f$, respectively. Then 
$c\vee m=e$ and $d\vee m=f$ by \eqref{eqtxtgGvngzlfnnH}. We cannot have that $c\leq f$ since otherwise we would obtain that $e=c\vee m\leq f$. Also, $c\not\leq d$, because otherwise $c\leq d\leq f$. Symmetrically, $f\not\leq e$ and $d\not\leq e$. These considerations show that the subposet $\set{\amin,\bmin, j:=m, c,d,e,f,i:=e\vee f}$  of $\poset\ce$ is isomorphic to $\poset{T_1}$; see Figure~\ref{figd4}. This contradicts Lemma~\ref{Lt1} and proves \eqref{eqpbxpRnTrSnL}. 

Next, we claim that
\begin{equation}
\parbox{7.6cm}{$\nlat\ce$ has only two minimal elements; let us agree that they will be denoted by $\amin$ and $\bmin$.}
\label{eqpbxTsthNk}
\end{equation}
Suppose the contrary, and let $\amin$, $\bmin$, and $\cmin$ be three distinct minimal elements. Pick an entry  for each of them. We know from \eqref{eqpbxpRnTrSnL} that these entries form a chain. Since they are not necessarily distinct, the entries in question form a one-element, a two-element, or a three-element chain $C$.
Let  $d$, $e$, and $f$ be
an  $\amin$-anchor, a $\bmin$-anchor, and a $\cmin$-anchor of the corresponding entry, respectively, and define $a:=\amin\wedge m$, $b:=\bmin\wedge m$, and $c:=\cmin\wedge m$.  Since $W(\amin)$ is an order ideal, $a\in W(\amin)$. Similarly, $b\in W(\bmin)$ and $c\in W(\cmin)$. By Lemma~\ref{lemmaLPVblvjDlsJzB},  $\set{a,b,c}$ is a three-element antichain.
By \eqref{eqtxtgGvngzlfnnH}, it is clear that, up to permutation of the elements $a,b,c$, the subset $C\cup\set{a,b,c,d,e,f}$ of $\ce$  forms a 
sub-\qn-lattice isomorphic to one of the \qn-lattices  $\nlat{S_\ell}$, $\ell\in\set{1,\dots, 4}$, determined by 
 Figure~\ref{figd3} and $W\cup W_\ell$ in Example~\ref{examplSlCs}. But this is a contradiction by Lemma~\ref{LmSi}, which proves \eqref{eqpbxTsthNk}. 

Next, we are going to show that
\begin{equation}
\text{$\amin$ and $\bmin$ do not have a common entry,}
\label{eqtxtvllmpnDkwz}
\end{equation}
that is, no $\amin$-entry is a $\bmin$-entry.
Suppose the contrary, and let $i$ be a minimal common entry of $\amin$ and $\bmin$. In fact, since any two common entries are comparable by \eqref{eqpbxpRnTrSnL}, $i$ is the \emph{least common entry} of $\amin$ and $\bmin$.
Pick an $\amin$-anchor $c$ and a $\bmin$-anchor $d$ of $i$. 
By Lemma~\ref{lemmaLPVblvjDlsJzB}, $c\parallel d$. 
Since $\filter \amin=M\cup W(\amin)=\ce\setminus W(\bmin)$ is planar by Lemma~\ref{lemmaproperplanar} and it is a lattice, we can fix a planar diagram $D$ of $\nlat{\filter \amin}$. In this diagram, the kernel $M=\filter m=[m,1_\ce]$ of $\nlat\ce$ is a region by Kelly and Rival~\cite[Lemma 1.5]{kellyrival}. If $i$ was in the interior of this region, then no lower cover of $i$ could be outside $M$ by Kelly and Rival~\cite[Lemma 1.2]{kellyrival}. But we know that $c$ and $d$ are lower covers of $i$ outside $M$, and we conclude that 
\begin{equation}
\text{$i$ is on the boundary of $M$,}
\label{eqtxtNrnnMklZtfl}
\end{equation}
with respect to the fixed diagram $D$. By left-right symmetry, we can assume that 
\begin{equation}
\text{$i$ is on the left boundary chain of $M$.}
\label{eqtxtfnvbRghzlfm}
\end{equation}
We claim that 
\begin{equation}\left.
\parbox{6.0 cm}{
at least one of $\amin$ and $\bmin$ has an entry distinct from the least common entry $i$.}\right\}
\label{eqpbxmmfRsdnPdhs}
\end{equation}
For the sake of an additional contradiction, suppose that $i$ is the only $\amin$-entry and the only $\bmin$-entry.
In order to prepare a forthcoming statement, \eqref{eqpbxmjDlnjrnF}, note that the argument beginning here and lasting up to  \eqref{eqpbxmjDlnjrnF} will use less assumption than what we have at present; it will use only that $i$ is a minimal  $\amin$-entry. 

Let $x_0:=i$, let $x_1$ be the rightmost lower cover of $x_0$ that belongs to $W(\amin)$, and for $k>1$, let $x_k$ be the rightmost lower cover of $x_{k-1}$ in the diagram $D$, provided $x_{k-1}$ is not the smallest element $\amin$. Since $W(\amin)$ is an order ideal, $x_k$ is automatically in $W(\amin)$. Denote the finite chain $\set{x_0, x_1,x_2,\dots}$ by $X$. We also define another chain, $Y$, as follows. Let $y_0=i$, and let $y_1$ be the unique lower cover of $y_0$ on the left boundary of $M$. That is, by \eqref{eqtxtfnvbRghzlfm}, $y_1$ is the leftmost lower cover of $i$ in $M$. Yet another way to define $y_1$ is to say that $y_1$ is the leftmost lower cover of $i$ that is to the right of $x_1$. So,  $x_1$ and $y_1$ are neighbouring lower covers of $i$.  As long as $y_k\neq \amin$, let $y_{k+1}$ be the leftmost cover of $y_k$ in the diagram $D$. Let $x_u=y_v$ be the first place where, going downwards, the chains $X$ and $Y$  intersect \emph{first}. (This place exists, because the two chains intersect at $\amin$.) Let $X':=X\cap\filter x_u$ and $Y':=Y\cap\filter y_v$. By Kelly and Rival~\cite[Lemma 1.5]{kellyrival}, the interval $[x_u,i]$ is a region in $D$. The chains  $X'$ and $Y'$ divide this region into three parts;
with our temporary terminology, into the \emph{left part} of  $[x_u,i]$ consisting of the elements on the left of $X'$ (including the elements of $X'$), the \emph{right part} of  $[x_u,i]$ consisting of the elements on the right of $Y'$ (including the elements of $Y'$), and the \emph{middle part} of $[x_u,i]$ consisting of those elements that are simultaneously strictly on the right of $X'$ and strictly on the left of $Y'$. Of course, everything here is understood modulo the fixed diagram $D$ of $\filter \amin$. 
We claim that middle part of $[x_u,i]$ is empty. For the sake of contradiction, suppose that $h$ is an element of the middle part. Then there is a maximal chain $Z=\set{i=z_0\succ z_1\succ\dots\succ z_t=h}$ in $[z,i]$. Since $x_1$ and $y_1$ are neighbouring lower covers of $i$, $z_1$ is either in the left part of $[x_u,i]$, or in the right part. However, if 
$z_1$ is on the left part of $[x_u,i]$, then the whole $Z$ remains in the left part, because each of the $x_{k+1}$ is the rightmost lower cover of $x_k$ for $k\geq 1$ and because $Z$ cannot ``jump over'' $X'$ by Kelly and Rival~~\cite[Lemma 1.2]{kellyrival}.  Similarly, if $z_1$ is in the right part of $[x_u,i]$, then so is the whole $Z$. So  $Z$ is either entirely in the left part, or entirely in the right part, whereby $Z$ cannot contain the element $h$, which is in the middle part. This contradiction shows that the middle part is empty and we have seen that 
\begin{equation}
\text{$X'\cup Y'$ is a cell.}
\label{eqbZzgClL}
\end{equation}
Next, we show that $m\in Y'$ but $m\notin X'$; then, in particular, it appears that $m$ is not the least element of the cell given in \eqref{eqbZzgClL}.  Since $x_1\in W(\amin)$ and $W(x)$ is an order ideal by Lemma~\ref{lemmaLPVblvjDlsJzB}, $X'\setminus\set i\subseteq W(\amin)$ and   $m\notin W(\amin)$ yield that $m\notin X'$. Clearly, $y_0=i\in M$ but $y_v=x_u\in W(\amin)$ is not in $M$. So there is a least subscript $k$ such that $y_k\in M$ but $y_{k+1}\notin M$. Since we are in $D$, the diagram of $\filter \amin$, we know that $y_{k+1}\geq \amin$.  Since  $y_{k+1}\geq\bmin$ would lead to $y_{k+1}\geq \amin\vee \bmin=m$ and $y_{k+1}\in \filter m=M$, we obtain by Lemma~\ref{lemmaLPVblvjDlsJzB} that $y_{k+1}\in W(\amin)$. By the definition of $Y'$,  $k$ is at least 1 and so $y_k<i$. If we had that $y_k>m$, then $y_k$ would be an $\amin$-entry with $\amin$-anchor $y_{k+1}$, but this is not possible since $i$ is a minimal $\amin$-entry. Hence $y_k\not>m$. But $y_k\in M=\filter m$, and we conclude that $m=y_k\in Y'$, as required. 
By the definition of $Y$ and $Y'$, the $y_\ell$ for $\ell<k$ are on the right boundary chain of $M$. 
Below, for later reference, we summarize what we have just shown.
\begin{equation}
\left.
\parbox{10.5cm}{If $i$ is a minimal $\amin$-entry and it is on the left boundary chain of $M$, then there exists a cell in the fixed planar diagram of $\filter\amin$ such that $m$ and $i$ are on the right boundary chain $Y'$ of this cell, $m$ is not the smallest element of the cell, and  $Y'\cap[m,i]$ is the same as the intersection of $[m,i]$ and the left boundary chain of $M$.}\right\}
\label{eqpbxmjDlnjrnF}
\end{equation}
Next, we resume the latest assumption that $i$ is the only $\amin$-entry and the only $\bmin$-entry. However, we will not always fully exploit this assumption. For the sake of a later reference, note in advance that 
\begin{equation}
\parbox{8.6cm}{for the validity of the forthcoming \eqref{eqtxthphszRknR}, \eqref{eqtxtchbnTrpLspJsk}, and \eqref{eqtxtkmnkVsRtkn}, it suffices to assume that $i$ is the unique $\bmin$-entry,}
\label{eqpbxPmzllKbnszR}
\end{equation}
that is, we will not use for a while that $i$ is also the unique $\amin$-entry. 
We claim that 
\begin{equation}
\text{whenever $u\in W(\bmin)$, then $u\vee m\in\set{m,i}$.}
\label{eqtxthphszRknR}
\end{equation}
Indeed, take a maximal chain $u=y_0\prec y_1\prec\dots\prec y_s=u\vee m$ in the interval $[u,u\vee m]$, and assume that $u\vee m\neq m$, that is, $u\not\leq m$. 
Since $y_0=u\in W(\bmin)=\filter\bmin\setminus M$ but $y_s=u\vee m\in M$, so $y_s\notin W(\bmin)$, there is a smallest subscript $t$ such that $y_{t-1}\in W(\bmin)$ but  $y_t\notin W(\bmin)$.
Since $y_t>y_{t-1}\geq \bmin$ but $y_t \notin W(\bmin)$, we know
from Lemma~\ref{lemmaLPVblvjDlsJzB} that $\bmin$ is not the only minimal element in $\ideal{y_t}$. Hence, $m=\amin\vee \bmin\leq y_t$ and $y_t\in M$. This fact and $y_{t-1}\prec y_t$ give that $y_t$ is a $\bmin$-entry. Since $u\leq y_t$ but $u\not\leq m$, we obtain that $y_t\neq m$, whereby $y_t$ is the same as $i$, the only $\bmin$-entry. Hence, $u\leq i$ and $u\vee m\leq i$. On the other hand, by the choice of our maximal chain, $y_t\leq u\vee m$,  whence $i=y_t\leq u\vee m$. Thus, $u\vee m=i$, and we conclude \eqref{eqtxthphszRknR}.  Next, armed with \eqref{eqtxthphszRknR}, we claim that 
\begin{equation}
\text{$\nlat{W(\bmin)\cup\set{m,i}}$ is a subnearlattice of $\nlat\ce$.}
\label{eqtxtchbnTrpLspJsk}
\end{equation} 
In order to prove this, assume that $x,y \in W(\bmin)\cup\set{m,i}$ and $x\parallel y$;  we need to show that $x\wedge y$ (which exists since $\set{x,y}\subseteq\filter\bmin$) and $x\vee y$ are in also in $W(\bmin)\cup\set{m,i}$. Since $\set{m,i}$ is a chain but $\set{x,y}$ is not, at least one of $x$ and $y$ is not in $\set{m,i}$. So, we can assume that, in addition to $x\parallel y$, $x\in W(\bmin)$ and $y\in W(\bmin)\cup\set{m,i}$. Since $W(\xmin)$ is an order ideal, $x\wedge y$ is in $W(x)$; suppose that $x\vee y$ is not. Then $x\vee y\in M=\filter m$ by Lemma~\ref{lemmaLPVblvjDlsJzB}, whereby $x\vee y=x\vee y\vee m=(x\vee m)\vee (y\vee m)\in\set{m,i}$ follows from \eqref{eqtxthphszRknR}. This proves \eqref{eqtxtchbnTrpLspJsk}. 
By Lemma~\ref{lemmaproperplanar},
\begin{equation}
\text{$\nlat{W(\bmin)\cup\set{m,i}}$ is a planar nearlattice.}
\label{eqtxtkmnkVsRtkn}
\end{equation}
If we had an element $u\in W(\bmin)\cup\set{m,i}$ such that  $m<u< i$, then \eqref{eqtxthphszRknR} would give that $u=u\vee m=i$, which would contradict the inequality $u<i$.  Hence,  $m$ is a coatom in the nearlattice $\nlat{W(\bmin)\cup\set{\bmin,i}}$. Letting $(m,i,m,\filter \amin, W(\bmin)\cup\set{\bmin,i})$ play the role of $(u,v,w,L,K)$, it follows from Corollary~\ref{corolmoon}, \eqref{eqtxtfnvbRghzlfm}, and \eqref{eqpbxmjDlnjrnF}  that $\nlat\ce$ is planar, which is a contradiction. Completing the ``encapsulated indirect argument'', this proves \eqref{eqpbxmmfRsdnPdhs}.

Next, we continue our argument towards \eqref{eqtxtvllmpnDkwz}; $i$ is still the least common entry.
We claim that 
\begin{equation}\left.
\parbox{6cm}{$\ideal i\setminus\set{i}$ cannot contain both an $\amin$-entry  and a $\bmin$-entry simultaneously.}
\label{eqtxtNhfBmrSpQedv}\right\}
\end{equation}
Suppose the contrary. Then there exist a minimal $\amin$-entry $j$ and a minimal $\bmin$-entry $k$ such that $j<i$ and $k<i$.
Since $i$ is the least common entry, $j$ is not a $\bmin$-entry and $k$ is not an $\amin$-entry.  Hence, these two entries, $j$ and $k$, are distinct. Also, they are comparable by \eqref{eqpbxpRnTrSnL}. Since $\amin$ and $\bmin$ play a symmetric role, we can assume that $k<j<i$. In addition to the 
already picked $\amin$-anchor $c$ and $\bmin$-anchor $d$ of the common entry $i$, choose an $\amin$-anchor $e$ of $j$ and 
a $\bmin$-anchor $f$ of $k$. We know from \eqref{eqtxtgGvngzlfnnH} that $c\vee m=i$, $d\vee m=i$, $e\vee m=j$, and $f\vee m=k$. Since $c\vee m=i> j=e\vee m$, it follows that $c\not\leq e$, that is, $e\parallel c$ or $e<c$. Similarly, $f\parallel d$ or $f<d$. Of course, each of $c$ and $e$ is incomparable with each of $d$ and $f$ by Lemma~\ref{lemmaLPVblvjDlsJzB}. These facts show that, in $\nlat\ce$, the subset  $\set{\amin,\bmin,c,d,e,f,i,j,k,m}$ forms a sub-\qn-lattice isomorphic to (one of the four versions of) $\nlat{T_{10}}$. Since this is impossible by Lemma~\ref{LmT10}, we have proved  \eqref{eqtxtNhfBmrSpQedv}.

We know from \eqref{eqpbxmmfRsdnPdhs} that $i$ is not the only entry. Below, in order to prove  \eqref{eqtxtvllmpnDkwz}, we are going to deal with two cases; namely, either $i$ is a minimal entry, or it is not minimal.

First, assume that $i$ is a minimal entry, that is, neither an $\amin$-entry, nor a $\bmin$ entry can be smaller than $i$. 
Since any other entry is comparable with $i$ by \eqref{eqpbxpRnTrSnL} and there exists another entry by \eqref{eqpbxmmfRsdnPdhs}, and since $\amin$ and $\bmin$ play a symmetric role, we can assume that there exists an $\amin$-entry $j$ such that $j >i$. (It may but need not happen that $j$ is also a $\bmin$-entry.)  Let $e$ be an $\amin$-anchor of $j$. Few lines after \eqref{eqtxtvllmpnDkwz}, we mentioned that $c\parallel d$. Lemma~\ref{lemmaLPVblvjDlsJzB} gives also that  $e\parallel d$. 
Since $i=c\vee m=d\vee m$ and $j=e\vee m$ by \eqref{eqtxtgGvngzlfnnH}, $e\not\leq c$. So $c<e$ or $c\parallel e$. 
Using the just-mentioned consequences of \eqref{eqtxtgGvngzlfnnH} and depending on whether $c<e$ or $c\parallel e$,   
the \qn-lattice $\nlat{T_5}$ or the \qn-lattice $\nlat{T_7}$ is a sub-\qn-lattice of $\nlat\ce$ so that its minimal elements are $\amin$ and $\bmin$ and its thick edges are bridges, but this is impossible by Lemmas~\ref{LmT5} and \ref{LmT7}.

Therefore, since the opposite case has just been excluded, 
$i$ is not a minimal entry. By the $\amin$--$\bmin$ symmetry, we can pick an $\amin$-entry $j$ such that $j<i$. It follows from \eqref{eqtxtNhfBmrSpQedv} that $i$ is a minimal $\bmin$-entry.
Pick an $\amin$-anchor $e$ of $j$. 
In addition to  $c\parallel d$, Lemma~\ref{lemmaLPVblvjDlsJzB} gives also that  $e\parallel d$.
By \eqref{eqtxtgGvngzlfnnH}, the equalities $i=c\vee m$ and $j=e\vee m$, and $d\vee m=i$ hold; the first two of them together with $j<i$ yield that $e<c$ or $e\parallel c$.  Hence, either the \qn-lattice $\nlat{T_6}$, or the \qn-lattice $\nlat{T_8}$ is a sub-\qn-lattice of $\nlat\ce$, but this is impossible by Lemmas~\ref{LmT6} and Lemmas~\ref{LmT8}. This proves the validity of \eqref{eqtxtvllmpnDkwz}.

By \eqref{eqtxtVnxMntrlpsmb}, there are at least one $\amin$-entry and at least one $\bmin$-entry. We claim that
\begin{equation}
\text{there are at least two $\amin$-entries or  at least two $\bmin$-entries.}
\label{eqtxncSgyhNtB}
\end{equation}
(This statement should not be confused  
with \eqref{eqpbxmmfRsdnPdhs} where the existence of a common entry was assumed.)
For the sake of contradiction, suppose that \eqref{eqtxncSgyhNtB} fails. Let $i$ and $e$ be the unique $\amin$-entry and $\bmin$-entry with $\amin$-anchor $c$ and $\bmin$-anchor $d$, respectively; see, for example, $T_9$ in    Figure~\ref{figd5}. We know from \eqref{eqpbxpRnTrSnL} that $i$ and $e$ are comparable. Furthermore, they are distinct by \eqref{eqtxtvllmpnDkwz}. Thus, using that $\amin$ and $\bmin$ play a symmetric role, we can assume that $e<i$. As in  \eqref{eqtxtNrnnMklZtfl} and \eqref{eqtxtfnvbRghzlfm}, we can assume  that $i$ is on the left boundary chain of $M$, with respect to the fixed diagram $D$ of $\poset{\filter\amin}$. Letting $e$ play the role of $i$, it follows from \eqref{eqpbxPmzllKbnszR} and  \eqref{eqtxtkmnkVsRtkn} that $\nlat{W(\bmin)\cup\set{m,e}}$ is a planar nearlattice. (In fact, it is a lattice.) If $e$ is also on the left boundary of $M$, then the argument just after \eqref{eqtxtkmnkVsRtkn}, tailored to the present notation by letting $e$ play the role of $i$, shows that $\nlat\ce$ is planar, which is a contradiction. If $e$ is on the right boundary chain of $M$, then the planarity of $\nlat{\filter\amin}$ and that of $\nlat{W(\bmin)\cup\set{m,e}}$ trivially imply that  $\nlat\ce$ is planar, which is a contradiction again. Therefore, $e$ is not on the boundary of $M$, with respect to the fixed planar diagram $D$. In other terms, being not on its boundary, $e$ is in the ``interior'' of $M$. Hence, if we add $e$ to the right  boundary chain of $M$, which is a maximal chain, what we obtain is no longer a chain. This means that there is element $g$ on the right boundary chain such that $e\parallel g$. It is visually clear and follows rigorously from Kelly and Rival~\cite[Proposition 1.6]{kellyrival} that $e$ is to the left of $g$, that is, $g$ is to the right of $e$. By left--right symmetry, there is an element $f$ on the left boundary chain of $M$ such that $f$ is to the left of $e$. By an unpublished result of J.\ Zilber, see  Kelly and Rival~\cite[Proposition 1.7]{kellyrival}, ``left'' is transitive and implies incomparability. Thus, $f$ is to the left of $g$ and   $\set{f,e,g}$ is an antichain. 
Observe that $i$ and $f$ are comparable since both belong to the same (left boundary) chain, but $f\not\geq i$ since otherwise we would get that $f\geq e$. Thus, $f<i$.
By \eqref{eqtxtgGvngzlfnnH}, we have that $\set{\amin,\bmin,m,c,d,e,f,g,i}$ forms a sub-\qn-lattice of $\nlat\ce$ isomorphic to $\nlat{T_9}$; the meaning of the dashed edge in Figure~\ref{figd5} now depends on whether $g<i$ in $\nlat\ce$ or not. But this contradicts Lemma~\ref{LmT9}. Therefore, we conclude the  validity of \eqref{eqtxncSgyhNtB}.

Based on \eqref{eqtxtVnxMntrlpsmb}, \eqref{eqpbxpRnTrSnL},   \eqref{eqtxtvllmpnDkwz}, and \eqref{eqtxncSgyhNtB}, we know that, apart from $\amin$--$\bmin$-symmetry, 
\begin{equation}\left.
\parbox{9.0cm}{there exist a minimal
$\amin$-entry $i$ with $\amin$-anchor $c$, another $\amin$-entry $j$ with $\amin$-anchor $d$, and a minimal $\bmin$-entry $k$ with $\bmin$-anchor $e$ such that $|\set{i,j,k}|=3$, $i\nonparallel k$, $j\nonparallel k$, $i$ and $j$ are not $\bmin$-entries, and $k$ is not an $\amin$-entry.}
\right\}
\label{eqpbxzhgnpqWvmT}
\end{equation} 
Note that there can be more $\amin$-entries and $\bmin$-entries, and an entry can have more anchors than those mentioned in \eqref{eqpbxzhgnpqWvmT}. However, we obtain by comparing \eqref{eqtxtgGvngzlfnnH} and \eqref{eqpbxzhgnpqWvmT} with \eqref{eqWQ}, \eqref{eqpbxkvkshCsZtlMm}, and Figure~\ref{figd6} that one of the \qn-lattices $\nlat{Q_s}$, $s\in\set{1,2,\dots,5}$, is a sub-\qn-lattice of $\nlat\ce$ so that \eqref{eqpbxkvkshCsZtlMm} holds. But this is a contradiction by Lemmas~\ref{LmQ1}, \ref{LmQ2}, \ref{LmQ3}, \ref{LmQ4}, and \ref{LmQ5}. Now that all possible cases have led to contradiction, it follows that $\nlat\ce$ does not exist. This completes the proof of Theorem~\ref{thmreform}.
\end{proof}

When reading the following proof, Kelly and Rival~\cite{kellyrival} or Cz\'edli~\cite[Theorem 2.5]{czg83}, where the Kelly-Rival Theorem is cited,  should be near; at the time of this writing, both of these two papers are freely downloadable.

\begin{proof}[Proof of Remark~\ref{remmsVnpL}] By Cz\'edli~\cite[Remark 1.3]{czg83}, it suffices to show that whenever $\nlat L$ is a nearlattice with at least two minimal elements and $|L|\leq 7$, or $|L|=8$ and $\many\nlat L>74$, then $\nlat L$  is planar. 
Suppose the contrary, that is, either $|L|\leq 7$, or $|L|=8$ and $\many\nlat L>74$, and   $\nlat L$ is nonplanar. Then  the lattice $\nlat{L^{(+0)}}$ we obtain from $\nlat L$ by adding a smallest element is nonplanar either.
If $|L|\leq 7$, then this nonplanar lattice is necessarily the eight-element boolean lattice by the Kelly--Rival Theorem, see \cite{kellyrival}, and then $\nlat L$ is
$\nlat{T_2}$ in Figure~\ref{figd4}, which is planar. If $|L|=8$, then it is straightforward to see by the Kelly--Rival Theorem that there are two possibilities.  The first possibility is that  $\nlat{L^{(+0)}}$ is one of the nine-element lattices in Kelly and Rival's list and then $\nlat L$, which is obtained from this lattice, is planar simply because only \emph{minimal} nonplanar lattices are included in the list. The second possibility is that the eight-element boolean lattice $\nlat{B_8}$ is a subposet of $\nlat{L^{(+0)}}$; let $u$ be the unique element of $L^{(+0)}\setminus B_8$. 
If $u$ is the smallest element $0$ of $\nlat{L^{(+0)}}$, then $\nlat L=\nlat{B_8}$ with  $\many=74$ by Cz\'edli~\cite[Lemma 2.7]{czg83}, whereby this case is excluded. 

If $u\neq 0$ is a doubly irreducible element of $\nlat{L^{(+0)}}$, then $u\in L$ and $\nlat{L\setminus\set u}$ is a planar nearlattice, since it is seven-element, and it is trivial and it follows also from Corollary~\ref{corolnarrow} that $\nlat L$ is planar, too. 

We are left with the situation when $u\neq 0$ is either join-reducible, or meet-reducible.  It is easy to see that $u$ cannot be doubly reducible since there are no antichains $\set{x_1,x_2}$ and $\set{y_1,y_2}$ in $B_8$ such that  $x_i\leq y_j$ for all $i,j\in\set{1,2}$. First, assume that $u$ is join-reducible. Then  $u=x\vee y$ in $\nlat{L^{(+0)}}$ such that $x\parallel y$ and $u\prec v$ where $v$ is the join of $x$ and $y$ in $\nlat{B_8}$. Assume that $v\neq 1$.  Since  there 
is only a single  two-element antichain $\set{x',y'}$ such that $v=x'\vee y'$ in $\nlat{B_8}$, it follows easily that $v$ is a nonzero doubly irreducible element in $\nlat{L^{(+0)}}$. Hence, letting $v$ play the role of $u$ in the previous case, we obtain that $\nlat L$ is planar. Motivated by $\nlat{T_2}$ in Figure~\ref{figd4}, it is not hard to draw a planar diagram for $\nlat L$ if  $v=1$. 
The analogous details for the case when $u$ is meet-reducible are omitted. 
\end{proof}

\section{Appendix 1: notes on our data files}\label{sectappndf}

\subsection*{How to test or create these data files}
In general, an input file of our program parses a whole hierarchy of subcases. Since, for each subcase, the  file  has to repeat all jm-constraints of the parent case(s), the input file and also the output file are often considerably longer than the proof of the corresponding statement given in Sections~\ref{sectexsomeqn}--\ref{sectanchors}. Since long files threaten with typing errors and the repeated use of the copy-paste function of our word processor without some easy-to-follow strategy threatens with even more involved errors,  it is worth outlining a method that reduces the chance of these errors. The point is that instead of proofreading a long output file, we can economically create an input file by following the corresponding proof from Sections~\ref{sectexsomeqn}--\ref{sectanchors} and implementing a standard method of computer science to our work with the help of a  word processor. This method is for listing the leaves of a tree by means of using a \emph{first in, first out} data buffer (also called \emph{FIFO}). 
Roughly saying, we begin with a short text file  describing the initial case \azeset{}; at this moment, this description is the first half of the file and the second half consists of a single comment line like \verb!\P0!. Then we keep modifying  the first half  towards the required input file while using the second part as a FIFO. When the second half becomes empty, the first half is the input file.  After describing the algorithm below, its functioning will be explained for the tree of subcases of the (main) case \azeset{} given in Figure~\ref{figd7}. 

\begin{figure}[htb] 
\centerline
{\includegraphics[scale=1.0]{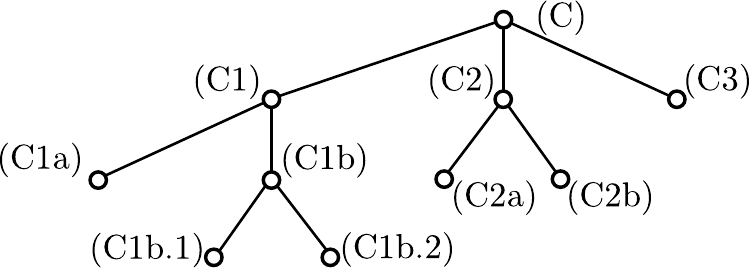}}
\caption{An example for the tree of subcases
\label{figd7}}
\end{figure}

\goodbreak

\begin{algorithm}\label{algorghT}\ 
\begin{enumeratei} 
\item\label{algorghTa} 
Start with `` $\azeset{}{\displaystyle\SUT0}\ed$ ''. Here $\ed$ is the end-of-file symbol, $\azeset{}$ refers to the main case. The encircled 0 after the main case means that none of its subcases has been processed yet. The underlining means the position of the ``cursor'' in the text file we are creating. Also, the text before the cursor is the ``first part'' of the file, while the text after the cursor is the ``second part''.  Note that as long as the file is not finished, it contains exactly one underlined encircled number, which is always the first encircled number; we will refer to this number as the \emph{cursor} or, shortly, $c$.
The case preceding the cursor will be denoted by $\azeset{\vec x}$. (At present, $\vec x$ is the empty string.)
\item\label{algorghTb} If there is no cursor, then the input file is ready and the algorithm terminates. Otherwise,  the first encircled number will be the cursor; underline it and  go to the following step below.
\item\label{algorghTc} Check if the case \azeset{$\vec x$} right before the cursor has to be split into further subcase(s) or not. (The answer is given in the corresponding proof, see Sections~\ref{sectexsomeqn}--\ref{sectanchors}. Note that when the cursor is 0, then  
the answer is affirmative if and only if the \tmany-value of the case is larger than 83.)
\item\label{algorghTd} If the answer for question \eqref{algorghTc} is negative and $c=0$, then delete the cursor and go to \eqref{algorghTb}. Otherwise, go to the following step below. 
\item\label{algorghTe} If the answer for question \eqref{algorghTc} is negative and $c>0$, then delete both the cursor and the case $\azeset{\vec x}$ preceding it, and go to \eqref{algorghTb}. Otherwise, go to the following step below.
\item\label{algorghTf} If the answer for question \eqref{algorghTc} is positive, then 
right before  $\azeset{\vec x}$, insert the $(c+1)$-th subcase $\azeset{\vec xy_{c+1}}\SUT0$, where $\azeset{\vec xy_{c+1}}$ denotes the $(c+1)$-th subcase of  $\azeset{\vec x}$. Change $\SUT c$ to $\snt j$ where $j:=c+1$. Now the just inserted $\SUT0$, the first encircled number, is the new cursor; this is why it is underlined. Go to \eqref{algorghTb}.
\end{enumeratei}
\end{algorithm}

\def\appendixhead{{Cz\'edli: Planar semilattices and nearlattices / Appendix 1}}
\markboth\appendixhead\appendixhead

When working with the text file, then we write ``\verb!\enddata!'' and a comment line ``\verb!\Pc!''  instead of $\ed$ and  $\SUT c$, respectively, while \azeset{$\vec x$} corresponds to the text starting with the last ``\verb!\beginjob!'' command and ending with the ``\verb!\endofjob!'' command before the cursor.
For the situation given by Figure~\ref{figd7}; Algorithm~\ref{algorghT} runs as follows. 
\allowdisplaybreaks{
\begin{align}                                                                   \azeset{}\SUT0\ed\cr
                                                                 \azeset{1}\SUT0\azeset{}\snt1\ed\cr 
                                                 \azeset{1a}\SUT0\azeset{1}\snt1\azeset{}\snt1\ed\cr
                                                      \azeset{1a}\azeset{1}\SUT1\azeset{}\snt1\ed\cr
                                      \azeset{1a}\azeset{1b}\SUT0\azeset{1}\snt2\azeset{}\snt1\ed\cr
                    \azeset{1a}\azeset{1b.1}\SUT0\azeset{1b}\snt1\azeset{1}\snt2\azeset{}\snt1\ed\cr
                         \azeset{1a}\azeset{1b.1}\azeset{1b}\SUT1\azeset{1}\snt2\azeset{}\snt1\ed\cr
       \azeset{1a}\azeset{1b.1}\azeset{1b.2}\SUT0\azeset{1b}\snt2\azeset{1}\snt2\azeset{}\snt1\ed\cr
            \azeset{1a}\azeset{1b.1}\azeset{1b.2}\azeset{1b}\SUT2\azeset{1}\snt2\azeset{}\snt1\ed\cr
                            \azeset{1a}\azeset{1b.1}\azeset{1b.2}\azeset{1}\SUT2\azeset{}\snt1\ed\cr
                                           \azeset{1a}\azeset{1b.1}\azeset{1b.2}\azeset{}\SUT1\ed\cr
                            \azeset{1a}\azeset{1b.1}\azeset{1b.2}\azeset{2}\SUT0\azeset{}\snt2\ed\cr
            \azeset{1a}\azeset{1b.1}\azeset{1b.2}\azeset{2a}\SUT0\azeset{2}\snt1\azeset{}\snt2\ed\cr
                 \azeset{1a}\azeset{1b.1}\azeset{1b.2}\azeset{2a}\azeset{2}\SUT1\azeset{}\snt2\ed \label{Awgfths}\\
 \azeset{1a}\azeset{1b.1}\azeset{1b.2}\azeset{2a}\azeset{2b}\SUT0\azeset{2}\snt2\azeset{}\snt2\ed \label{AwochRftF}\\
      \azeset{1a}\azeset{1b.1}\azeset{1b.2}\azeset{2a}\azeset{2b}\azeset{2}\SUT2\azeset{}\snt2\ed\cr
                     \azeset{1a}\azeset{1b.1}\azeset{1b.2}\azeset{2a}\azeset{2b}\azeset{}\SUT2\ed\cr
      \azeset{1a}\azeset{1b.1}\azeset{1b.2}\azeset{2a}\azeset{2b}\azeset{3}\SUT0\azeset{}\snt3\ed\cr
           \azeset{1a}\azeset{1b.1}\azeset{1b.2}\azeset{2a}\azeset{2b}\azeset{3}\azeset{}\SUT3\ed\cr
                         \azeset{1a}\azeset{1b.1}\azeset{1b.2}\azeset{2a}\azeset{2b}\azeset{3}\ed \label{AwcfhsgsB}
\end{align}
}%

For example, we have obtained \eqref{AwochRftF} from \eqref{Awgfths} by replacing the string
$\azeset{2}\SUT1$ by the string $\azeset{2b}\SUT0\azeset{2}\snt2$. At the end, \eqref{AwcfhsgsB} lists the leaves of the tree in Figure~\ref{figd7} from left to right.

Note that although Algorithm~\ref{algorghT} is good to organize our work and to see when a case has to be split into subcases, it needs the user's decision what these subcases should be. Also, it is the user who has to prove that a subcase is meaningful; for example the user but not the computer verifies whether a new element occurring in the subcase is really distinct from previous elements. 

\subsection*{A sample input file} Below, we give the input file \verb!LmQ4.txt! in connection with Lemma~\ref{LmQ4}. Since it contains only two cases, Algorithm~\ref{algorghT} is not used. Command names begin with `` \verb!\! ''; in particular,  ``\verb!\P!'' means (and begins)  comment lines to be \underline printed in the output file. In a line, everything after a ``\verb! % !'' 
character is also a comment but it will not be printed. When listing edges and constraints, comments begin with ``\verb!\w!''; these comments are \underline written in the output file.  The lattice operations are denoted by $+$ (join) and $*$ (meet). The second line of the file, beginning with ``\verb! % !'',
is superfluous; it was used as a ``mile stone'' to avoid oversized lines. 
\normalfont
\begin{verbatim}
\P Version of August 17, 2019
%                                                             E
\verbose=false
\subtrahend-in-exponent=8 
\operationsymbols=+* % This command is obligatory.

\beginjob
\name
LmQ4/C1 d*m=a
\size
9
\elements
abmcdeijk
\edges
ac am be bm ci dj ek ij jk mi
 ad \w C1
\constraints 
a+b=m c+m=i d+m=j e+m=k, c*m=a e*m=b b+d=j e*j=b e*i=b
 d*m=a \w C1
\endofjob

\beginjob
\name
LmQ4/C2 d*m=:x>=A
\size
10
\elements
abmcdeijkx
\edges
ac am be bm ci dj ek ij jk mi
 xd xm
\constraints 
a+b=m c+m=i d+m=j e+m=k, c*m=a e*m=b b+d=j e*j=b e*i=b
 d*m=x x+b=m
\endofjob

\P Also done: LmQ4/C (all cases)
\enddata
\end{verbatim}

\section{Appendix 2: output files}\label{sectappend2} This appendix presents the output files of our computer program. These files are integral parts of the proof of the main result of the paper; they are given in the following order; for each of them, we give the name of the \qn-lattice, the name of the downloadable output file, and  the relevant statement(s) in the paper.

\def\appendixhead{{Cz\'edli: Planar semilattices and nearlattices / Appendix 2}}
\markboth\appendixhead\appendixhead

\begin{itemize}
\item $Q_1$,\quad \texttt{LmQ1-out.txt};\quad see Lemma~\ref{LmQ1}
\item $Q_2$,\quad \texttt{LmQ2-out.txt};\quad see Lemma~\ref{LmQ2}
\item $Q_3$,\quad \texttt{LmQ3-out.txt};\quad see Lemma~\ref{LmQ3}
\item $Q_4$,\quad \texttt{LmQ4-out.txt};\quad see Lemma~\ref{LmQ4}
\item $Q_5$,\quad \texttt{LmQ5-out.txt};\quad see Lemma~\ref{LmQ5}
\item $S_1$--$S_4$,\quad \texttt{LmSi-out.txt};\quad see Lemma~\ref{LmSi}
\item $T_1$,\quad \texttt{LmT1a-out.txt};\quad see Lemma~\ref{LmT1a}
\item $T'_1$,\quad \texttt{LmT1b-out.txt};\quad see Lemma~\ref{LmT1b}
\item $T_2$,\quad \texttt{LmT2-out.txt};\quad see Lemma~\ref{LmT2}
\item $T_3$,\quad \texttt{LmT3-out.txt};\quad see Lemma~\ref{LmT3}
\item $T_4$,\quad \texttt{LmT4-out.txt};\quad see Lemma~\ref{LmT4}
\item $T_5$,\quad \texttt{LmT5-out.txt};\quad see Lemma~\ref{LmT5}
\item $T_6$,\quad \texttt{LmT6-out.txt};\quad see Lemma~\ref{LmT6}
\item $T_7$,\quad \texttt{LmT7-out.txt};\quad see Lemma~\ref{LmT7}
\item $T_8$,\quad \texttt{LmT8-out.txt};\quad see Lemma~\ref{LmT8}
\item $T_9$,\quad \texttt{LmT9-out.txt};\quad see Lemma~\ref{LmT9}
\item $T_{10}$,\quad \texttt{LmT10-out.txt};\quad see Lemma~\ref{LmT10}
\end{itemize}

Now, the output files are as follows. Except for the first one, some parts of the headings (like the version of the program and the research grant) will be omitted from them. 

\medskip
\centerline{{\LARGE$Q_1$,\quad \texttt{LmQ1-out.txt};\quad see Lemma~\ref{LmQ1}}}
\medskip

\begin{verbatim}
 Version of August 15, 2019
SUBLATTS ver. July 23, 2019 (start:=1:57:5.865
   (verbose=FALSE) reports:
 [ Supported by the Hungarian Research Grant KH 126581,
                                (C) Gabor Czedli, 2019 ]


L: LmQ1/C1a a<d  d*k=a
|L|=9, L={abcdemkij}. Edges:
ac am be bm ci dj ek ij ki mk ; C
 ad ; C1
-- Constraints:
a+b=m c+m=i d+m=j e+m=k  c*m=a e*m=b c*k=a b+d=j ; C
 d*k=a ; C1a
-- Result:   |Sub(L)|=154 for the partial lattice
-- LmQ1/C1a a<d  d*k=a. Thus,
sigma(L) = |Sub(L)|*2^(8-|L|) =   77.0000000000000000 .

L: LmQ1/C1b a<d  d*k=:x>a
|L|=10, L={abcdemkijx}. Edges:
ac am be bm ci dj ek ij ki mk ; C
 ad ; C1
  xd xk ax ; C1b
-- Constraints:
a+b=m c+m=i d+m=j e+m=k  c*m=a e*m=b c*k=a b+d=j ; C
 d*k=x ; C1b
-- Result:   |Sub(L)|=287 for the partial lattice
-- LmQ1/C1b a<d  d*k=:x>a. Thus,
sigma(L) = |Sub(L)|*2^(8-|L|) =   71.7500000000000000 .
 Also done: LmQ1/C1 a<d 

L: LmQ1/C2 a||d  a*d=:y>=A
|L|=10, L={abcdemkijy}. Edges:
ac am be bm ci dj ek ij ki mk ; C
 ya yd ; C2
-- Constraints:
a+b=m c+m=i d+m=j e+m=k  c*m=a e*m=b c*k=a b+d=j ; C
 a*d=y y+b=m ; C2
-- Result:   |Sub(L)|=278 for the partial lattice
-- LmQ1/C2 a||d  a*d=:y>=A. Thus,
sigma(L) = |Sub(L)|*2^(8-|L|) =   69.5000000000000000 .
 Also done: LmQ1/C (all cases) 

The computation took 0/1000 seconds.
\end{verbatim}

\medskip
\centerline{{\LARGE$Q_2$,\quad \texttt{LmQ2-out.txt};\quad see Lemma~\ref{LmQ2}}}
\medskip

\begin{verbatim}
 Version of August 16, 2019
L: LmQ2/C
|L|=10, L={abmcdeijkv}. Edges:
ac am bm be ci dj ek ki kj mk iv jv
-- Constraints:
a+b=m c+m=i d+m=j e+m=k  c*m=a e*m=b c*k=a i+j=v d+b=j
-- Result:   |Sub(L)|=278 for the partial lattice
-- LmQ2/C. Thus,
sigma(L) = |Sub(L)|*2^(8-|L|) =   69.5000000000000000 .

The computation took 0/1000 seconds.
\end{verbatim}

\medskip
\centerline{{\LARGE$Q_3$,\quad \texttt{LmQ3-out.txt};\quad see Lemma~\ref{LmQ3}}}
\medskip

\begin{verbatim}
 Version of August 16, 2019
L: LmQ3/C1 d*i=c
|L|=9, L={abmcdeijk}. Edges:
ac am be bm ci dj ek ik kj mi ; C
 cd ; C1
-- Constraints:
a+b=m c+m=i d+m=j e+m=k  c*m=a e*m=b e*i=b b+d=j ; C
 d*i=c ; C1
-- Result:   |Sub(L)|=150 for the partial lattice
-- LmQ3/C1 d*i=c. Thus,
sigma(L) = |Sub(L)|*2^(8-|L|) =   75.0000000000000000 .

L: LmQ3/C2 d*i=a
|L|=9, L={abmcdeijk}. Edges:
ac am be bm ci dj ek ik kj mi ; C
 ad ; C2
-- Constraints:
a+b=m c+m=i d+m=j e+m=k  c*m=a e*m=b e*i=b b+d=j ; C
 d*i=a ; C2
-- Result:   |Sub(L)|=148 for the partial lattice
-- LmQ3/C2 d*i=a. Thus,
sigma(L) = |Sub(L)|*2^(8-|L|) =   74.0000000000000000 .

L: LmQ3/C3 d*i=:x (new element)
|L|=10, L={abmcdeijkx}. Edges:
ac am be bm ci dj ek ik kj mi ; C
 xd xi ; C3
-- Constraints:
a+b=m c+m=i d+m=j e+m=k  c*m=a e*m=b e*i=b b+d=j ; C
 d*i=x ; C3
-- Result:   |Sub(L)|=315 for the partial lattice
-- LmQ3/C3 d*i=:x (new element). Thus,
sigma(L) = |Sub(L)|*2^(8-|L|) =   78.7500000000000000 .
 Also done: LmQ3/C (all cases)

The computation took 16/1000 seconds.
\end{verbatim}

\medskip
\centerline{{\LARGE$Q_4$,\quad \texttt{LmQ4-out.txt};\quad see Lemma~\ref{LmQ4}}}
\medskip

\begin{verbatim}
 Version of August 17, 2019
L: LmQ4/C1 d*m=a
|L|=9, L={abmcdeijk}. Edges:
ac am be bm ci dj ek ij jk mi
 ad ; C1
-- Constraints:
a+b=m c+m=i d+m=j e+m=k  c*m=a e*m=b b+d=j e*j=b e*i=b
 d*m=a ; C1
-- Result:   |Sub(L)|=158 for the partial lattice
-- LmQ4/C1 d*m=a. Thus,
sigma(L) = |Sub(L)|*2^(8-|L|) =   79.0000000000000000 .

L: LmQ4/C2 d*m=:x>=A
|L|=10, L={abmcdeijkx}. Edges:
ac am be bm ci dj ek ij jk mi
 xd xm
-- Constraints:
a+b=m c+m=i d+m=j e+m=k  c*m=a e*m=b b+d=j e*j=b e*i=b
 d*m=x x+b=m
-- Result:   |Sub(L)|=282 for the partial lattice
-- LmQ4/C2 d*m=:x>=A. Thus,
sigma(L) = |Sub(L)|*2^(8-|L|) =   70.5000000000000000 .
 Also done: LmQ4/C (all cases)

The computation took 0/1000 seconds.
\end{verbatim}

\medskip
\centerline{{\LARGE$Q_5$,\quad \texttt{LmQ5-out.txt};\quad see Lemma~\ref{LmQ5}}}
\medskip

\begin{verbatim}
 Version of August 17, 2019
L: LmQ5/C1 c*d=a
|L|=9, L={abmcdeijk}. Edges:
ac am bm be ci dj ek ik jk mi mj ; C
 ad ; C1
-- Constraints:
a+b=m c+m=i d+m=j e+m=k  c*m=a e*m=b b+d=j e*i=b e*j=b ; C
 c*d=a ; C1
-- Result:   |Sub(L)|=164 for the partial lattice
-- LmQ5/C1 c*d=a. Thus,
sigma(L) = |Sub(L)|*2^(8-|L|) =   82.0000000000000000 .

L: LmQ5/C2 c*d=x
|L|=10, L={abmcdeijkx}. Edges:
ac am bm be ci dj ek ik jk mi mj ; C
 xc xd ; C2
-- Constraints:
a+b=m c+m=i d+m=j e+m=k  c*m=a e*m=b b+d=j e*i=b e*j=b ; C
 c*d=x ; C2
-- Result:   |Sub(L)|=326 for the partial lattice
-- LmQ5/C2 c*d=x. Thus,
sigma(L) = |Sub(L)|*2^(8-|L|) =   81.5000000000000000 .
 Also done: LmQ5/C (all cases)

The computation took 0/1000 seconds.
\end{verbatim}

\medskip
\centerline{{\LARGE$S_1$--$S_4$,\quad \texttt{LmSi-out.txt};\quad see Lemma~\ref{LmSi}}}
\medskip

\begin{verbatim}
 Version of August 11, 2019
L: S_1
|L|=8, L={abcdefmi}. Edges:
ad am be bm cf cm di ei fi mi
-- Constraints:
a+b=m a+c=m b+c=m d*m=a e*m=b f*m=c ; W
d+m=i e+m=i f+m=i ; W_1
-- Result:   |Sub(L)|=77 for the partial lattice
-- S_1. Thus,
sigma(L) = |Sub(L)|*2^(8-|L|) =   77.0000000000000000 .

L: S_2
|L|=9, L={abcdefmij}. Edges:
ad am be bm cf cm di ei fj ji mj
-- Constraints:
a+b=m a+c=m b+c=m d*m=a e*m=b f*m=c ; W
d+m=i e+m=i f+m=j ; W_2
-- Result:   |Sub(L)|=139 for the partial lattice
-- S_2. Thus,
sigma(L) = |Sub(L)|*2^(8-|L|) =   69.5000000000000000 .

L: S_3
|L|=9, L={abcdefmij}. Edges:
ad am be bm cf cm dj ei fj ji mj
-- Constraints:
a+b=m a+c=m b+c=m d*m=a e*m=b f*m=c ; W
d+m=j e+m=i f+m=j ; W_3
-- Result:   |Sub(L)|=139 for the partial lattice
-- S_3. Thus,
sigma(L) = |Sub(L)|*2^(8-|L|) =   69.5000000000000000 .

L: S_4
|L|=10, L={abcdefmijk}. Edges:
ad am be bm cf cm di ej fk ji kj mk
-- Constraints:
a+b=m a+c=m b+c=m d*m=a e*m=b f*m=c ; W
d+m=i e+m=j f+m=k ; W_3
-- Result:   |Sub(L)|=259 for the partial lattice
-- S_4. Thus,
sigma(L) = |Sub(L)|*2^(8-|L|) =   64.7500000000000000 .

The computation took 0/1000 seconds.
\end{verbatim}

\medskip
\centerline{{\LARGE$T_1$,\quad \texttt{LmT1a-out.txt};\quad see Lemma~\ref{LmT1a}}}
\medskip

\begin{verbatim}
 Version of August 11, 2019
L: Lmt1a/C1 exists h||i  h+i=:k new
|L|=10, L={abcdefijhk}. Edges:
ac aj bd bj ce df ei fi je jf ; C
hk ik ; C1
-- Constraints:
a+b=j c+j=e d+j=f e+f=i  c*j=a d*j=b e*f=j ; C
h+i=k ; C1
-- Result:   |Sub(L)|=287 for the partial lattice
-- Lmt1a/C1 exists h||i  h+i=:k new. Thus,
sigma(L) = |Sub(L)|*2^(8-|L|) =   71.7500000000000000 .

L: Lmt1a/C2a exists h<i  h<a and h<b and so a*b=h is assumed
|L|=9, L={abcdefijh}. Edges:
ac aj bd bj ce df ei fi je jf ; C
 hi ; C2
  ha hb ; C2a
-- Constraints:
a+b=j c+j=e d+j=f e+f=i  c*j=a d*j=b e*f=j ; C
 a*b=h ; C2a
-- Result:   |Sub(L)|=147 for the partial lattice
-- Lmt1a/C2a exists h<i  h<a and h<b and so a*b=h is assumed. Thus,
sigma(L) = |Sub(L)|*2^(8-|L|) =   73.5000000000000000 .

L: Lmt1a/C2b.1 exists h<i  h>a h>b  so h>a+b=j  e+h=i
|L|=9, L={abcdefijh}. Edges:
ac aj bd bj ce df ei fi je jf ; C
 hi ; C2
  ah bh jh ; C2b
-- Constraints:
a+b=j c+j=e d+j=f e+f=i  c*j=a d*j=b e*f=j ; C
 e+h=i ; C2b.1
-- Result:   |Sub(L)|=154 for the partial lattice
-- Lmt1a/C2b.1 exists h<i  h>a h>b  so h>a+b=j  e+h=i. Thus,
sigma(L) = |Sub(L)|*2^(8-|L|) =   77.0000000000000000 .

L: Lmt1a/C2b.2 exists h<i  h>a h>b  so h>a+b=j  e+h=:k<i
|L|=10, L={abcdefijhk}. Edges:
ac aj bd bj ce df ei fi je jf ; C
 hi ; C2
  ah bh jh ; C2b
   ek hk ki ; C2b.2
-- Constraints:
a+b=j c+j=e d+j=f e+f=i  c*j=a d*j=b e*f=j ; C
 e+h=k ; C2b.2
-- Result:   |Sub(L)|=266 for the partial lattice
-- Lmt1a/C2b.2 exists h<i  h>a h>b  so h>a+b=j  e+h=:k<i. Thus,
sigma(L) = |Sub(L)|*2^(8-|L|) =   66.5000000000000000 .
 Also excluded: Lmt1a/C2b exists h<i, h>a h>b, so h>a+b=j

L: Lmt1a/C2c.1 exists h<i  h||a  h+a=:k<i is new
|L|=10, L={abcdefijhk}. Edges:
ac aj bd bj ce df ei fi je jf ; C
 hi ; C2
  hk ak ki ; C2c.1
-- Constraints:
a+b=j c+j=e d+j=f e+f=i  c*j=a d*j=b e*f=j ; C
 h+a=k ; C2c.1
-- Result:   |Sub(L)|=298 for the partial lattice
-- Lmt1a/C2c.1 exists h<i  h||a  h+a=:k<i is new. Thus,
sigma(L) = |Sub(L)|*2^(8-|L|) =   74.5000000000000000 .

L: Lmt1a/C2c.2 exists h<i  h||a  h+a=c
|L|=9, L={abcdefijh}. Edges:
ac aj bd bj ce df ei fi je jf ; C
 hi ; C2
-- Constraints:
a+b=j c+j=e d+j=f e+f=i  c*j=a d*j=b e*f=j ; C
 h+a=c ; C2c.2
-- Result:   |Sub(L)|=129 for the partial lattice
-- Lmt1a/C2c.2 exists h<i  h||a  h+a=c. Thus,
sigma(L) = |Sub(L)|*2^(8-|L|) =   64.5000000000000000 .

L: Lmt1a/C2c.3 exists h<i  h||a  h+a=j
|L|=9, L={abcdefijh}. Edges:
ac aj bd bj ce df ei fi je jf ; C
 hi ; C2
-- Constraints:
a+b=j c+j=e d+j=f e+f=i  c*j=a d*j=b e*f=j ; C
 h+a=j ; C2c.3
-- Result:   |Sub(L)|=151 for the partial lattice
-- Lmt1a/C2c.3 exists h<i  h||a  h+a=j. Thus,
sigma(L) = |Sub(L)|*2^(8-|L|) =   75.5000000000000000 .

L: Lmt1a/C2c.4 exists h<i  h||a  h+a=e
|L|=9, L={abcdefijh}. Edges:
ac aj bd bj ce df ei fi je jf ; C
 hi ; C2
-- Constraints:
a+b=j c+j=e d+j=f e+f=i  c*j=a d*j=b e*f=j ; C
 h+a=e ; C2c.4
-- Result:   |Sub(L)|=131 for the partial lattice
-- Lmt1a/C2c.4 exists h<i  h||a  h+a=e. Thus,
sigma(L) = |Sub(L)|*2^(8-|L|) =   65.5000000000000000 .

L: Lmt1a/C2c.5 exists h<i  h||a  h+a=f
|L|=9, L={abcdefijh}. Edges:
ac aj bd bj ce df ei fi je jf ; C
 hi ; C2
-- Constraints:
a+b=j c+j=e d+j=f e+f=i  c*j=a d*j=b e*f=j ; C
 h+a=f ; C2c.5
-- Result:   |Sub(L)|=135 for the partial lattice
-- Lmt1a/C2c.5 exists h<i  h||a  h+a=f. Thus,
sigma(L) = |Sub(L)|*2^(8-|L|) =   67.5000000000000000 .

L: Lmt1a/C2c.6 exists h<i  h||a  h+a=i
|L|=9, L={abcdefijh}. Edges:
ac aj bd bj ce df ei fi je jf ; C
 hi ; C2
-- Constraints:
a+b=j c+j=e d+j=f e+f=i  c*j=a d*j=b e*f=j ; C
 h+a=i ; C2c.6
-- Result:   |Sub(L)|=137 for the partial lattice
-- Lmt1a/C2c.6 exists h<i  h||a  h+a=i. Thus,
sigma(L) = |Sub(L)|*2^(8-|L|) =   68.5000000000000000 .
 Also excluded: Lmt1a/C2c exists h<i, h||a, 
 Also excluded: Lmt1a/C2 exists h<i
 Also excluded: Lmt1a/C

The computation took 46/1000 seconds.
\end{verbatim}

\medskip
\centerline{{\LARGE$T'_1$,\quad \texttt{LmT1b-out.txt};\quad see Lemma~\ref{LmT1b}}}
\medskip

\begin{verbatim}
 Version of August 12, 2019
L: LmT1b/C
|L|=9, L={abcdefjig}. Edges:
ac aj bj bd ce df ei fi ge gf jg
-- Constraints:
c*j=a d*j=b c+g=e d+g=f e*f=g e+f=i a+b=j  c+j=e d+j=f
-- Result:   |Sub(L)|=162 for the partial lattice
-- LmT1b/C. Thus,
sigma(L) = |Sub(L)|*2^(8-|L|) =   81.0000000000000000 .

The computation took 0/1000 seconds.
\end{verbatim}

\medskip
\centerline{{\LARGE$T_2$,\quad \texttt{LmT2-out.txt};\quad see Lemma~\ref{LmT2}}}
\medskip

\begin{verbatim}
 Version of August 12, 2019
L: LmT2/C1 exists d||i  let d+i=:j
|L|=9, L={aAbBcCidj}. Edges:
aB aC bA bC cA cB Ai Bi Ci ; LmT2/C
 dj ij ; C1
-- Constraints:
a+b=C a+c=B b+c=A  A+B=i B+C=i C+A=i  ; LmT2/C
A*C=b B*C=a A*B=c ; always  see the proof
 d+i=j ; C1
-- Result:   |Sub(L)|=151 for the partial lattice
-- LmT2/C1 exists d||i  let d+i=:j. Thus,
sigma(L) = |Sub(L)|*2^(8-|L|) =   75.5000000000000000 .

L: LmT2/C2a exists d<i  d<a d<b d<c
|L|=8, L={aAbBcCid}. Edges:
aB aC bA bC cA cB Ai Bi Ci ; LmT2/C
 di ; C2
  da db dc ; C2a
-- Constraints:
a+b=C a+c=B b+c=A  A+B=i B+C=i C+A=i  ; LmT2/C
A*C=b B*C=a A*B=c ; always  see the proof
 a*b=d a*c=d b*c=d ; C2a  see the proof
-- Result:   |Sub(L)|=74 for the partial lattice
-- LmT2/C2a exists d<i  d<a d<b d<c. Thus,
sigma(L) = |Sub(L)|*2^(8-|L|) =   74.0000000000000000 .

L: LmT2/C2b.1 exists d<i  d||a  a+d=:e<i is a new element
|L|=9, L={aAbBcCide}. Edges:
aB aC bA bC cA cB Ai Bi Ci ; LmT2/C
 di ; C2
  ae de ei ; C2b.1
-- Constraints:
a+b=C a+c=B b+c=A  A+B=i B+C=i C+A=i  ; LmT2/C
A*C=b B*C=a A*B=c ; always  see the proof
 a+d=e ; C2b.1
-- Result:   |Sub(L)|=155 for the partial lattice
-- LmT2/C2b.1 exists d<i  d||a  a+d=:e<i is a new element. Thus,
sigma(L) = |Sub(L)|*2^(8-|L|) =   77.5000000000000000 .

L: LmT2/C2b.2 exists d<i  d||a  a+d=B
|L|=8, L={aAbBcCid}. Edges:
aB aC bA bC cA cB Ai Bi Ci ; LmT2/C
 di ; C2
-- Constraints:
a+b=C a+c=B b+c=A  A+B=i B+C=i C+A=i  ; LmT2/C
A*C=b B*C=a A*B=c ; always  see the proof
 a+d=B ; C2b.2
-- Result:   |Sub(L)|=80 for the partial lattice
-- LmT2/C2b.2 exists d<i  d||a  a+d=B. Thus,
sigma(L) = |Sub(L)|*2^(8-|L|) =   80.0000000000000000 .

L: LmT2/C2b.3 exists d<i  d||a  a+d=C
|L|=8, L={aAbBcCid}. Edges:
aB aC bA bC cA cB Ai Bi Ci ; LmT2/C
 di ; C2
-- Constraints:
a+b=C a+c=B b+c=A  A+B=i B+C=i C+A=i  ; LmT2/C
A*C=b B*C=a A*B=c ; always  see the proof
 a+d=C ; C2b.3
-- Result:   |Sub(L)|=80 for the partial lattice
-- LmT2/C2b.3 exists d<i  d||a  a+d=C. Thus,
sigma(L) = |Sub(L)|*2^(8-|L|) =   80.0000000000000000 .

L: LmT2/C2b.4 exists d<i  d||a  a+d=i
|L|=8, L={aAbBcCid}. Edges:
aB aC bA bC cA cB Ai Bi Ci ; LmT2/C
 di ; C2
-- Constraints:
a+b=C a+c=B b+c=A  A+B=i B+C=i C+A=i  ; LmT2/C
A*C=b B*C=a A*B=c ; always  see the proof
 a+d=i ; C2b.4
-- Result:   |Sub(L)|=81 for the partial lattice
-- LmT2/C2b.4 exists d<i  d||a  a+d=i. Thus,
sigma(L) = |Sub(L)|*2^(8-|L|) =   81.0000000000000000 .
 Also done: LmT2/C2b exists d<i, d||a
 Also done: LmT2/C2 exists d<i
 Also done: LmT2/C

The computation took 31/1000 seconds.
\end{verbatim}

\medskip
\centerline{{\LARGE$T_3$,\quad \texttt{LmT3-out.txt};\quad see Lemma~\ref{LmT3}}}
\medskip

\begin{verbatim}
 Version of August 12, 2019
L: LmT3/C1 p q<=d and u v<=e
|L|=10, L={abcdeipquv}. Edges:
ad bd be ce di ei ; C
 ap pd aq qd cu ue cv ve ; C1
-- Constraints:
a+b=d a+c=i b+c=e d+e=i  d*e=b ; C
 p*q=a u*v=c ; C1
-- Result:   |Sub(L)|=253 for the partial lattice
-- LmT3/C1 p q<=d and u v<=e. Thus,
sigma(L) = |Sub(L)|*2^(8-|L|) =   63.2500000000000000 .

L: LmT3/C2a  p q<=d  not u<=e  e+u=i
|L|=9, L={abcdeipqu}. Edges:
ad bd be ce di ei ; C
 ap pd aq qd cu; /C2
  ui ; C2a
-- Constraints:
a+b=d a+c=i b+c=e d+e=i  d*e=b ; C
 p*q=a e*u=c ; /C2
  e+u=i ; C2a
-- Result:   |Sub(L)|=146 for the partial lattice
-- LmT3/C2a  p q<=d  not u<=e  e+u=i. Thus,
sigma(L) = |Sub(L)|*2^(8-|L|) =   73.0000000000000000 .

L: LmT3/C2b  p q<=d  not u<=e  e+u=:w is new
|L|=10, L={abcdeipquw}. Edges:
ad bd be ce di ei ; C
 ap pd aq qd cu; /C2
  ew uw ; C2b
-- Constraints:
a+b=d a+c=i b+c=e d+e=i  d*e=b ; C
 p*q=a e*u=c ; /C2
  e+u=w ; C2b
-- Result:   |Sub(L)|=305 for the partial lattice
-- LmT3/C2b  p q<=d  not u<=e  e+u=:w is new. Thus,
sigma(L) = |Sub(L)|*2^(8-|L|) =   76.2500000000000000 .
 Also done: LmT3/C2  p,q<=d, not u<=e

L: LmT3/C3a neither p<=d nor u<=e  d*p=a e*u=c  d+p=e+u=i
|L|=8, L={abcdeipu}. Edges:
ad bd be ce di ei ; C
 ap cu ; C3
  pi ui ; C3a
-- Constraints:
a+b=d a+c=i b+c=e d+e=i  d*e=b ; C
 d*p=a e*u=c; C3
   d+p=i e+u=i ; C3a
-- Result:   |Sub(L)|=81 for the partial lattice
-- LmT3/C3a neither p<=d nor u<=e  d*p=a e*u=c  d+p=e+u=i. Thus,
sigma(L) = |Sub(L)|*2^(8-|L|) =   81.0000000000000000 .

L: LmT3/C3b neither p<=d nor u<=e  d*p=a e*u=c  d+p=x e+u=i
|L|=9, L={abcdeipux}. Edges:
ad bd be ce di ei ; C
 ap cu ; C3
  dx px ui ; C3b
-- Constraints:
a+b=d a+c=i b+c=e d+e=i  d*e=b ; C
 d*p=a e*u=c; C3
  d+p=x e+u=i ; C3b
-- Result:   |Sub(L)|=163 for the partial lattice
-- LmT3/C3b neither p<=d nor u<=e  d*p=a e*u=c  d+p=x e+u=i. Thus,
sigma(L) = |Sub(L)|*2^(8-|L|) =   81.5000000000000000 .

L: LmT3/C3c.1 not p<=d not u<=e d*p=a e*u=c d+p=:x e+u=:y x+y=:z
|L|=11, L={abcdeipuxyz}. Edges:
ad bd be ce di ei ; C
 ap cu ; C3
  dx px ey uy ; C3c
   xz yz; C3c.1
-- Constraints:
a+b=d a+c=i b+c=e d+e=i  d*e=b ; C
 d*p=a e*u=c; C3
  d+p=x e+u=y ; C3c
   x+y=z ; C3c.1
-- Result:   |Sub(L)|=519 for the partial lattice
-- LmT3/C3c.1 not p<=d not u<=e d*p=a e*u=c d+p=:x e+u=:y x+y=:z. Thus,
sigma(L) = |Sub(L)|*2^(8-|L|) =   64.8750000000000000 .

L: LmT3/C3c.2 not p<=d not u<=e d*p=a e*u=c d+p=:x e+u=:y x+y=i
|L|=10, L={abcdeipuxy}. Edges:
ad bd be ce di ei ; C
 ap cu ; C3
  dx px ey uy ; C3c
   xi yi ; C3c.2
-- Constraints:
a+b=d a+c=i b+c=e d+e=i  d*e=b ; C
 d*p=a e*u=c; C3
  d+p=x e+u=y ; C3c
   x+y=i ; C3c.2
-- Result:   |Sub(L)|=248 for the partial lattice
-- LmT3/C3c.2 not p<=d not u<=e d*p=a e*u=c d+p=:x e+u=:y x+y=i. Thus,
sigma(L) = |Sub(L)|*2^(8-|L|) =   62.0000000000000000 .
 Also done: 
  LmT3/C3c neither p<=d nor u<=e, d*p=a e*u=c, d+p=:x, e+u=:y

L: LmT3/C3d neither p<=d nor u<=e  d*p=a e*u=c  d+p=e+u=:x
|L|=9, L={abcdeipux}. Edges:
ad bd be ce di ei ; C
 ap cu ; C3
  dx px ux ex ix ; C3d
-- Constraints:
a+b=d a+c=i b+c=e d+e=i  d*e=b ; C
 d*p=a e*u=c; C3
  d+p=x e+u=x; C3d
-- Result:   |Sub(L)|=142 for the partial lattice
-- LmT3/C3d neither p<=d nor u<=e  d*p=a e*u=c  d+p=e+u=:x. Thus,
sigma(L) = |Sub(L)|*2^(8-|L|) =   71.0000000000000000 .
 Also done: LmT3/C3 neither p<=d nor u<=e, d*p=a e*u=c
 Also done: LmT3/C

The computation took 47/1000 seconds.
\end{verbatim}

\medskip
\centerline{{\LARGE$T_4$,\quad \texttt{LmT4-out.txt};\quad see Lemma~\ref{LmT4}}}
\medskip

\begin{verbatim}
 Version of August 13, 2019
L: LmT4/C1a.1a u<i v<i  u+v=i  p<d q<d  e||d
|L|=10, L={abcdipeuvq}. Edges:
ad bd ci di ap be cu ; C
 cv ui vi ; C1
  pd qd aq  ; C1a.1
-- Constraints:
a+b=d a+c=i b+c=i c+d=i ; C
 u*v=c ; C1
  u+v=i ; C1a
   p*q=a  ; C1a.1
    d*e=b ; C1a.1a
-- Result:   |Sub(L)|=306 for the partial lattice
-- LmT4/C1a.1a u<i v<i  u+v=i  p<d q<d  e||d. Thus,
sigma(L) = |Sub(L)|*2^(8-|L|) =   76.5000000000000000 .

L: LmT4/C1a.1b u<i v<i  u+v=i  p<d q<d   f exists
|L|=11, L={abcdipeuvqf}. Edges:
ad bd ci di ap be cu ; C
 cv ui vi ; C1
  pd qd aq  ; C1a.1
   bf ; C1a.1b
-- Constraints:
a+b=d a+c=i b+c=i c+d=i ; C
 u*v=c ; C1
  u+v=i ; C1a
   p*q=a  ; C1a.1
    e*f=b ; C1a.1b
-- Result:   |Sub(L)|=607 for the partial lattice
-- LmT4/C1a.1b u<i v<i  u+v=i  p<d q<d   f exists. Thus,
sigma(L) = |Sub(L)|*2^(8-|L|) =   75.8750000000000000 .
 Also done: LmT4/C1a.1 u<i v<i, u+v=i, p<d q<d  DUPLUM for a 

L: LmT4/C1a.2a.1  u<i v<i  u+v=i p not<d p*d=a p+d=i  e f<d
|L|=10, L={abcdipeuvf}. Edges:
ad bd ci di ap be cu ; C
 cv ui vi ; C1
  pi ; C1a.2a
   ed bf fd ; C1a.2a.1
-- Constraints:
a+b=d a+c=i b+c=i c+d=i ; C
 u*v=c ; C1
  u+v=i ; C1a
   p*d=a ; C1a.2
    p+d=i ; C1a.2a
     e*f=b ; C1a.2a.1
-- Result:   |Sub(L)|=288 for the partial lattice
-- LmT4/C1a.2a.1  u<i v<i  u+v=i p not<d p*d=a p+d=i  e f<d. Thus,
sigma(L) = |Sub(L)|*2^(8-|L|) =   72.0000000000000000 .

L: LmT4/C1a.2a.2a  u<i v<i u+v=i p not<d p*d=a p+d=i d*e=b d+e=i
|L|=9, L={abcdipeuv}. Edges:
ad bd ci di ap be cu ; C
 cv ui vi ; C1
  pi ; C1a.2a
   ei ; C1a.2a.2a
-- Constraints:
a+b=d a+c=i b+c=i c+d=i ; C
 u*v=c ; C1
  u+v=i ; C1a
   p*d=a ; C1a.2
    p+d=i ; C1a.2a
     d*e=b ; C1a.2a.2
      d+e=i ; C1a.2a.2a
-- Result:   |Sub(L)|=163 for the partial lattice
-- LmT4/C1a.2a.2a  u<i v<i u+v=i p not<d p*d=a p+d=i d*e=b d+e=i. Thus,
sigma(L) = |Sub(L)|*2^(8-|L|) =   81.5000000000000000 .

L: LmT4/C1a.2a.2b  u<i v<i u+v=i p not<d p*d=a p+d=i d*e=b d+e=:x
|L|=10, L={abcdipeuvx}. Edges:
ad bd ci di ap be cu ; C
 cv ui vi ; C1
  pi ; C1a.2a
   dx ex ; C1a.2a.2b
-- Constraints:
a+b=d a+c=i b+c=i c+d=i ; C
 u*v=c ; C1
  u+v=i ; C1a
   p*d=a ; C1a.2
    p+d=i ; C1a.2a
     d*e=b ; C1a.2a.2
      d+e=x ; C1a.2a.2b
-- Result:   |Sub(L)|=307 for the partial lattice
-- LmT4/C1a.2a.2b  u<i v<i u+v=i p not<d p*d=a p+d=i d*e=b d+e=:x. Thus,
sigma(L) = |Sub(L)|*2^(8-|L|) =   76.7500000000000000 .
 Also done: 
    LmT4/C1a.2a.2  u<i v<i,u+v=i,p not<d p*d=a,p+d=i,d*e=b
 Also done: LmT4/C1a.2a  u<i v<i, u+v=i,p not<d p*d=a,p+d=i

L: LmT4/C1a.2b.1a u<i v<i u+v=i p*d=a p+d=:x x<i e<d f<d
|L|=11, L={abcdipeuvxf}. Edges:
ad bd ci di ap be cu ; C
 cv ui vi ; C1
  px dx ; C1a.2b
   xi ; C1a.2b.1
    ed fd bf ; C1a.2b.1a
-- Constraints:
a+b=d a+c=i b+c=i c+d=i ; C
 u*v=c ; C1
  u+v=i ; C1a
   p*d=a ; C1a.2
    p+d=x ; C1a.2b
     e*f=b ; C1a.2b.1a
-- Result:   |Sub(L)|=493 for the partial lattice
-- LmT4/C1a.2b.1a u<i v<i u+v=i p*d=a p+d=:x x<i e<d f<d. Thus,
sigma(L) = |Sub(L)|*2^(8-|L|) =   61.6250000000000000 .

L: LmT4/C1a.2b.1b u<i v<i u+v=i p*d=a p+d=:x x<i e*d=b
|L|=10, L={abcdipeuvx}. Edges:
ad bd ci di ap be cu ; C
 cv ui vi ; C1
  px dx ; C1a.2b
   xi ; C1a.2b.1
-- Constraints:
a+b=d a+c=i b+c=i c+d=i ; C
 u*v=c ; C1
  u+v=i ; C1a
   p*d=a ; C1a.2
    p+d=x ; C1a.2b
     e*d=b ; C1a.2b.1b
-- Result:   |Sub(L)|=306 for the partial lattice
-- LmT4/C1a.2b.1b u<i v<i u+v=i p*d=a p+d=:x x<i e*d=b. Thus,
sigma(L) = |Sub(L)|*2^(8-|L|) =   76.5000000000000000 .
 Also done: LmT4/C1a.2b.1 u<i v<i,u+v=i,p*d=a,p+d=:x,x<i

L: LmT4/C1a.2b.2a u<i v<i u+v=i p*d=a p+d=:x i<x  e<d f<d
|L|=11, L={abcdipeuvxf}. Edges:
ad bd ci di ap be cu ; C
 cv ui vi ; C1
  px dx ; C1a.2b
   ix ; C1a.2b.2
    ed fd bf ; C1a.2b.2a
-- Constraints:
a+b=d a+c=i b+c=i c+d=i ; C
 u*v=c ; C1
  u+v=i ; C1a
   p*d=a ; C1a.2
    p+d=x ; C1a.2b
     e*f=b ; C1a.2b.2a
-- Result:   |Sub(L)|=489 for the partial lattice
-- LmT4/C1a.2b.2a u<i v<i u+v=i p*d=a p+d=:x i<x  e<d f<d. Thus,
sigma(L) = |Sub(L)|*2^(8-|L|) =   61.1250000000000000 .

L: LmT4/C1a.2b.2b u<i v<i u+v=i p*d=a p+d=:x i<x  e*d=b
|L|=10, L={abcdipeuvx}. Edges:
ad bd ci di ap be cu ; C
 cv ui vi ; C1
  px dx ; C1a.2b
   ix ; C1a.2b.2
-- Constraints:
a+b=d a+c=i b+c=i c+d=i ; C
 u*v=c ; C1
  u+v=i ; C1a
   p*d=a ; C1a.2
    p+d=x ; C1a.2b
     e*d=b ; C1a.2b.2b
-- Result:   |Sub(L)|=298 for the partial lattice
-- LmT4/C1a.2b.2b u<i v<i u+v=i p*d=a p+d=:x i<x  e*d=b. Thus,
sigma(L) = |Sub(L)|*2^(8-|L|) =   74.5000000000000000 .
 Also done: LmT4/C1a.2b.2 u<i v<i,u+v=i,p*d=a,p+d=:x, i<x

L: LmT4/C1a.2b.3 u<i v<i u+v=i p*d=a p+d=:x  i||x and i+x=:y
|L|=11, L={abcdipeuvxy}. Edges:
ad bd ci di ap be cu ; C
 cv ui vi ; C1
  px dx ; C1a.2b
   iy xy ; C1a.2b.3
-- Constraints:
a+b=d a+c=i b+c=i c+d=i ; C
 u*v=c ; C1
  u+v=i ; C1a
   p*d=a ; C1a.2
    p+d=x ; C1a.2b
     i+x=y ; C1a.2b.3
-- Result:   |Sub(L)|=526 for the partial lattice
-- LmT4/C1a.2b.3 u<i v<i u+v=i p*d=a p+d=:x  i||x and i+x=:y. Thus,
sigma(L) = |Sub(L)|*2^(8-|L|) =   65.7500000000000000 .
 Also done: LmT4/C1a.2b u<i v<i,u+v=i,p*d=a,p+d=:x
 Also done: C1a.2 u<i v<i, u+v=i,p not<d p*d=a
 Also done: LmT4/C1a u<i v<i, u+v=i

L: LmT4/C1b.1 u<i v<i u+v=:g p<d q<d
|L|=11, L={abcdipeuvgq}. Edges:
ad bd ci di ap be cu ; C
 cv ui vi ; C1
  ug vg gi ; C1b
   pd qd aq ; C1b.1
-- Constraints:
a+b=d a+c=i b+c=i c+d=i ; C
 u*v=c ; C1
  u+v=g ; C1b
   p*q=a ; C1b.1
-- Result:   |Sub(L)|=640 for the partial lattice
-- LmT4/C1b.1 u<i v<i u+v=:g p<d q<d. Thus,
sigma(L) = |Sub(L)|*2^(8-|L|) =   80.0000000000000000 .

L: LmT4/C1b.2a.1 u<i v<i u+v=:g p not<d  d+p=i  e<d f<d
|L|=11, L={abcdipeuvgf}. Edges:
ad bd ci di ap be cu ; C
 cv ui vi ; C1
  ug vg gi ; C1b
   pi ; C1b.2a
    ed fd bf ; C1b.2a.1
-- Constraints:
a+b=d a+c=i b+c=i c+d=i ; C
 u*v=c ; C1
  u+v=g ; C1b
   p*d=a ; C1b.2
    d+p=i ; C1b.2a
     e*f=b ; C1b.2a.1
-- Result:   |Sub(L)|=517 for the partial lattice
-- LmT4/C1b.2a.1 u<i v<i u+v=:g p not<d  d+p=i  e<d f<d. Thus,
sigma(L) = |Sub(L)|*2^(8-|L|) =   64.6250000000000000 .

L: LmT4/C1b.2a.2 u<i v<i u+v=:g p not<d  d+p=i  d*e=b
|L|=10, L={abcdipeuvg}. Edges:
ad bd ci di ap be cu ; C
 cv ui vi ; C1
  ug vg gi ; C1b
   pi ; C1b.2a
-- Constraints:
a+b=d a+c=i b+c=i c+d=i ; C
 u*v=c ; C1
  u+v=g ; C1b
   p*d=a ; C1b.2
    d+p=i ; C1b.2a
     d*e=b ; C1b.2a.2
-- Result:   |Sub(L)|=313 for the partial lattice
-- LmT4/C1b.2a.2 u<i v<i u+v=:g p not<d  d+p=i  d*e=b. Thus,
sigma(L) = |Sub(L)|*2^(8-|L|) =   78.2500000000000000 .
 Also done: LmT4/C1b.2a u<i v<i,u+v=:g,p not<d, d+p=i

L: LmT4/C1b.2b.1 u<i v<i u+v=:g p not<d p+d=:x  x<i
|L|=11, L={abcdipeuvgx}. Edges:
ad bd ci di ap be cu ; C
 cv ui vi ; C1
  ug vg gi ; C1b
   px dx ; C1b.2b
    xi ; C1b.2b.1
-- Constraints:
a+b=d a+c=i b+c=i c+d=i ; C
 u*v=c ; C1
  u+v=g ; C1b
   p*d=a ; C1b.2
    p+d=x ; C1b.2b
-- Result:   |Sub(L)|=612 for the partial lattice
-- LmT4/C1b.2b.1 u<i v<i u+v=:g p not<d p+d=:x  x<i. Thus,
sigma(L) = |Sub(L)|*2^(8-|L|) =   76.5000000000000000 .

L: LmT4/C1b.2b.2 u<i v<i u+v=:g p not<d p+d=:x  i<x
|L|=11, L={abcdipeuvgx}. Edges:
ad bd ci di ap be cu ; C
 cv ui vi ; C1
  ug vg gi ; C1b
   px dx ; C1b.2b
    ix ; C1b.2b.2
-- Constraints:
a+b=d a+c=i b+c=i c+d=i ; C
 u*v=c ; C1
  u+v=g ; C1b
   p*d=a ; C1b.2
    p+d=x ; C1b.2b
-- Result:   |Sub(L)|=608 for the partial lattice
-- LmT4/C1b.2b.2 u<i v<i u+v=:g p not<d p+d=:x  i<x. Thus,
sigma(L) = |Sub(L)|*2^(8-|L|) =   76.0000000000000000 .

L: LmT4/C1b.2b.3 u<i v<i u+v=:g p not<d p+d=:x  x+i=:y
|L|=12, L={abcdipeuvgxy}. Edges:
ad bd ci di ap be cu ; C
 cv ui vi ; C1
  ug vg gi ; C1b
   px dx ; C1b.2b
    xy iy ; C1b.2b.3
-- Constraints:
a+b=d a+c=i b+c=i c+d=i ; C
 u*v=c ; C1
  u+v=g ; C1b
   p*d=a ; C1b.2
    p+d=x ; C1b.2b
     x+i=y ; C1b.2b.3
-- Result:   |Sub(L)|=932 for the partial lattice
-- LmT4/C1b.2b.3 u<i v<i u+v=:g p not<d p+d=:x  x+i=:y. Thus,
sigma(L) = |Sub(L)|*2^(8-|L|) =   58.2500000000000000 .
 Also done: LmT4/C1b.2b u<i v<i,u+v=:g,p not<d,p+d=:x
 Also done: LmT4/C1b.2 u<i v<i,u+v=:g,p not<d,
 Also done: LmT4/C1b u<i v<i,u+v=:g
 Also done: LmT4/C1 u<i v<i

L: LmT4/C2a.1a u||i i+u=:g  p<d q<d p+q=d  e<d f<d
|L|=11, L={abcdipeugqf}. Edges:
ad bd ci di ap be cu ; C
 ig ug ; C2
  pd qd aq ; C2a
   ed fd bf ; C2a.1a
-- Constraints:
a+b=d a+c=i b+c=i c+d=i ; C
 i+u=g i*u=c ; C2
  p*q=a ; C2a
   p+q=d ; C2a.1
    e*f=b ; C2a.1a
-- Result:   |Sub(L)|=493 for the partial lattice
-- LmT4/C2a.1a u||i i+u=:g  p<d q<d p+q=d  e<d f<d. Thus,
sigma(L) = |Sub(L)|*2^(8-|L|) =   61.6250000000000000 .

L: LmT4/C2a.1b u||i i+u=:g  p<d q<d p+q=d  e||d
|L|=10, L={abcdipeugq}. Edges:
ad bd ci di ap be cu ; C
 ig ug ; C2
  pd qd aq ; C2a
-- Constraints:
a+b=d a+c=i b+c=i c+d=i ; C
 i+u=g i*u=c ; C2
  p*q=a ; C2a
   p+q=d ; C2a.1
    e*d=b ; C2a.1b
-- Result:   |Sub(L)|=290 for the partial lattice
-- LmT4/C2a.1b u||i i+u=:g  p<d q<d p+q=d  e||d. Thus,
sigma(L) = |Sub(L)|*2^(8-|L|) =   72.5000000000000000 .
 Also done: LmT4/C2a.1 u||i i+u=:g, p<d q<d,p+q=d

L: LmT4/C2a.2 u||i i+u=:g  p<d q<d  p+q=:x<d
|L|=11, L={abcdipeugqx}. Edges:
ad bd ci di ap be cu ; C
 ig ug ; C2
  pd qd aq ; C2a
   px qx xd ; C2a.2
-- Constraints:
a+b=d a+c=i b+c=i c+d=i ; C
 i+u=g i*u=c ; C2
  p*q=a ; C2a
   p+q=x ; C2a.2
-- Result:   |Sub(L)|=650 for the partial lattice
-- LmT4/C2a.2 u||i i+u=:g  p<d q<d  p+q=:x<d. Thus,
sigma(L) = |Sub(L)|*2^(8-|L|) =   81.2500000000000000 .
 Also done: LmT4/C2a u||i i+u=:g, p<d q<d

L: LmT4/C2b.1a u||i i+u=:g  p*d=a  p+d=i  e<d f<d
|L|=10, L={abcdipeugf}. Edges:
ad bd ci di ap be cu ; C
 ig ug ; C2
  pi ; C2b.1
   ed fd bf ; C2b.1a
-- Constraints:
a+b=d a+c=i b+c=i c+d=i ; C
 i+u=g i*u=c ; C2
  p*d=a ; C2b
   p+d=i ; C2b.1
    e*f=b ; C2b.1a
-- Result:   |Sub(L)|=262 for the partial lattice
-- LmT4/C2b.1a u||i i+u=:g  p*d=a  p+d=i  e<d f<d. Thus,
sigma(L) = |Sub(L)|*2^(8-|L|) =   65.5000000000000000 .

L: LmT4/C2b.1b u||i i+u=:g  p*d=a  p+d=i  e not< d
|L|=9, L={abcdipeug}. Edges:
ad bd ci di ap be cu ; C
 ig ug ; C2
  pi ; C2b.1
-- Constraints:
a+b=d a+c=i b+c=i c+d=i ; C
 i+u=g i*u=c ; C2
  p*d=a ; C2b
   p+d=i ; C2b.1
    e*d=b ; C2b.1b
-- Result:   |Sub(L)|=163 for the partial lattice
-- LmT4/C2b.1b u||i i+u=:g  p*d=a  p+d=i  e not< d. Thus,
sigma(L) = |Sub(L)|*2^(8-|L|) =   81.5000000000000000 .
 Also done: LmT4/C2b.1 u||i i+u=:g, p*d=a, p+d=i

L: LmT4/C2b.2a u||i i+u=:g  p*d=a  p+d=g  e<d f<d
|L|=10, L={abcdipeugf}. Edges:
ad bd ci di ap be cu ; C
 ig ug ; C2
  pg ; C2b.2
   ed fd bf ; C2b.2a
-- Constraints:
a+b=d a+c=i b+c=i c+d=i ; C
 i+u=g i*u=c ; C2
  p*d=a ; C2b
   p+d=g ; C2b.2
    e*f=b ; C2b.2a
-- Result:   |Sub(L)|=260 for the partial lattice
-- LmT4/C2b.2a u||i i+u=:g  p*d=a  p+d=g  e<d f<d. Thus,
sigma(L) = |Sub(L)|*2^(8-|L|) =   65.0000000000000000 .

L: LmT4/C2b.2b u||i i+u=:g  p*d=a  p+d=g  e*d=b
|L|=9, L={abcdipeug}. Edges:
ad bd ci di ap be cu ; C
 ig ug ; C2
  pg ; C2b.2
-- Constraints:
a+b=d a+c=i b+c=i c+d=i ; C
 i+u=g i*u=c ; C2
  p*d=a ; C2b
   p+d=g ; C2b.2
    e*d=b ; C2b.2b
-- Result:   |Sub(L)|=159 for the partial lattice
-- LmT4/C2b.2b u||i i+u=:g  p*d=a  p+d=g  e*d=b. Thus,
sigma(L) = |Sub(L)|*2^(8-|L|) =   79.5000000000000000 .
 Also done: LmT4/C2b.2 u||i i+u=:g, p*d=a, p+d=g

L: LmT4/C2b.3a u||i i+u=:g  p*d=a  p+d=:x  e<d f<d
|L|=11, L={abcdipeugxf}. Edges:
ad bd ci di ap be cu ; C
 ig ug ; C2
  px dx ; C2b.3
   ix ; C2b.3a
    ed fd bf ; C2b.3a
-- Constraints:
a+b=d a+c=i b+c=i c+d=i ; C
 i+u=g i*u=c ; C2
  p*d=a ; C2b
   p+d=x ; C2b.3
    e*f=b ; C2b.3a
-- Result:   |Sub(L)|=494 for the partial lattice
-- LmT4/C2b.3a u||i i+u=:g  p*d=a  p+d=:x  e<d f<d. Thus,
sigma(L) = |Sub(L)|*2^(8-|L|) =   61.7500000000000000 .

L: LmT4/C2b.3b.1 u||i i+u=:g  p*d=a  p+d=:x  d*e=b  d+e=i
|L|=10, L={abcdipeugx}. Edges:
ad bd ci di ap be cu ; C
 ig ug ; C2
  px dx ; C2b.3
   ei ; C2b.3b.1
-- Constraints:
a+b=d a+c=i b+c=i c+d=i ; C
 i+u=g i*u=c ; C2
  p*d=a ; C2b
   p+d=x ; C2b.3
    d*e=b ; C2b.3b
     d+e=i ; C2b.3b.1
-- Result:   |Sub(L)|=289 for the partial lattice
-- LmT4/C2b.3b.1 u||i i+u=:g  p*d=a  p+d=:x  d*e=b  d+e=i. Thus,
sigma(L) = |Sub(L)|*2^(8-|L|) =   72.2500000000000000 .

L: LmT4/C2b.3b.2 u||i i+u=:g  p*d=a  p+d=:x  d*e=b  d+e=x
|L|=10, L={abcdipeugx}. Edges:
ad bd ci di ap be cu ; C
 ig ug ; C2
  px dx ; C2b.3
   ex ; C2b.3b.2
-- Constraints:
a+b=d a+c=i b+c=i c+d=i ; C
 i+u=g i*u=c ; C2
  p*d=a ; C2b
   p+d=x ; C2b.3
    d*e=b ; C2b.3b
     d+e=x ; C2b.3b.2
-- Result:   |Sub(L)|=291 for the partial lattice
-- LmT4/C2b.3b.2 u||i i+u=:g  p*d=a  p+d=:x  d*e=b  d+e=x. Thus,
sigma(L) = |Sub(L)|*2^(8-|L|) =   72.7500000000000000 .

L: LmT4/C2b.3b.3 u||i i+u=:g  p*d=a  p+d=:x  d*e=b  d+e=g
|L|=10, L={abcdipeugx}. Edges:
ad bd ci di ap be cu ; C
 ig ug ; C2
  px dx ; C2b.3
   eg ; C2b.3b.3
-- Constraints:
a+b=d a+c=i b+c=i c+d=i ; C
 i+u=g i*u=c ; C2
  p*d=a ; C2b
   p+d=x ; C2b.3
    d*e=b ; C2b.3b
     d+e=g ; C2b.3b.3
-- Result:   |Sub(L)|=283 for the partial lattice
-- LmT4/C2b.3b.3 u||i i+u=:g  p*d=a  p+d=:x  d*e=b  d+e=g. Thus,
sigma(L) = |Sub(L)|*2^(8-|L|) =   70.7500000000000000 .

L: LmT4/C2b.3b.4 u||i i+u=:g  p*d=a  p+d=:x  d*e=b  d+e=:y
|L|=11, L={abcdipeugxy}. Edges:
ad bd ci di ap be cu ; C
 ig ug ; C2
  px dx ; C2b.3
   dy ey ; C2b.3b.4
-- Constraints:
a+b=d a+c=i b+c=i c+d=i ; C
 i+u=g i*u=c ; C2
  p*d=a ; C2b
   p+d=x ; C2b.3
    d*e=b ; C2b.3b
     d+e=y ; C2b.3b.4
-- Result:   |Sub(L)|=600 for the partial lattice
-- LmT4/C2b.3b.4 u||i i+u=:g  p*d=a  p+d=:x  d*e=b  d+e=:y. Thus,
sigma(L) = |Sub(L)|*2^(8-|L|) =   75.0000000000000000 .
 Also done: LmT4/C2b.3b u||i i+u=:g, p*d=a, p+d=:x, d*e=b
 Also done: LmT4/C2b.3 u||i i+u=:g, p*d=a, p+d=:x
 Also done: LmT4/C2b u||i i+u=:g, p*d=a
 Also done: LmT4/C2 u||i i+u=:g
 All cases have been excluded.

The computation took 129/1000 seconds.
\end{verbatim}

\medskip
\centerline{{\LARGE$T_5$,\quad \texttt{LmT5-out.txt};\quad see Lemma~\ref{LmT5}}}
\medskip

\begin{verbatim}
 Version of Aug 14, 2019
L: LmT5/C1a x||a  x>A x*a=:y and y+b=m
|L|=10, L={abcdemijxy}. Edges:
am ac bm bd ce ci di ej ij mi ; C
 yx ya ym ; C1a
-- Constraints:
a+b=m c+m=i d+m=i e+m=j c*m=a d*m=b e*i=c ; C
 x*a=y y+b=m ; C1a
-- Result:   |Sub(L)|=267 for the partial lattice
-- LmT5/C1a x||a  x>A x*a=:y and y+b=m. Thus,
sigma(L) = |Sub(L)|*2^(8-|L|) =   66.7500000000000000 .

L: LmT5/C1b.1 x||a  x not>A  a+x=m
|L|=9, L={abcdemijx}. Edges:
am ac bm bd ce ci di ej ij mi ; C
 xm ; C1b.1
-- Constraints:
a+b=m c+m=i d+m=i e+m=j c*m=a d*m=b e*i=c ; C
 a+x=m ; C1b.1
-- Result:   |Sub(L)|=149 for the partial lattice
-- LmT5/C1b.1 x||a  x not>A  a+x=m. Thus,
sigma(L) = |Sub(L)|*2^(8-|L|) =   74.5000000000000000 .

L: LmT5/C1b.2 x||a  x not>A  a+x=i
|L|=9, L={abcdemijx}. Edges:
am ac bm bd ce ci di ej ij mi ; C
 xi ; C1b.2
-- Constraints:
a+b=m c+m=i d+m=i e+m=j c*m=a d*m=b e*i=c ; C
 a+x=i ; C1b.2
-- Result:   |Sub(L)|=149 for the partial lattice
-- LmT5/C1b.2 x||a  x not>A  a+x=i. Thus,
sigma(L) = |Sub(L)|*2^(8-|L|) =   74.5000000000000000 .

L: LmT5/C1b.3 x||a  x not>A  a+x=j
|L|=9, L={abcdemijx}. Edges:
am ac bm bd ce ci di ej ij mi ; C
 xj ; C1b.3
-- Constraints:
a+b=m c+m=i d+m=i e+m=j c*m=a d*m=b e*i=c ; C
 a+x=j ; C1b.3
-- Result:   |Sub(L)|=140 for the partial lattice
-- LmT5/C1b.3 x||a  x not>A  a+x=j. Thus,
sigma(L) = |Sub(L)|*2^(8-|L|) =   70.0000000000000000 .

L: LmT5/C1b.4 x||a  x not>A  a+x=:y is a new element
|L|=10, L={abcdemijxy}. Edges:
am ac bm bd ce ci di ej ij mi ; C
 ay xy ; C1b.4
-- Constraints:
a+b=m c+m=i d+m=i e+m=j c*m=a d*m=b e*i=c ; C
 a+x=y ; C1b.4
-- Result:   |Sub(L)|=305 for the partial lattice
-- LmT5/C1b.4 x||a  x not>A  a+x=:y is a new element. Thus,
sigma(L) = |Sub(L)|*2^(8-|L|) =   76.2500000000000000 .
 Also done: LmT5/C1b x||a, x not>A
 Also done: LmT5/C1 x||a

L: LmT5/C2a x>a x||b  b+x=m
|L|=9, L={abcdemijx}. Edges:
am ac bm bd ce ci di ej ij mi ; C
 ax ; C2
  xm ; C2a
-- Constraints:
a+b=m c+m=i d+m=i e+m=j c*m=a d*m=b e*i=c ; C
 b+x=m ; C2a
-- Result:   |Sub(L)|=141 for the partial lattice
-- LmT5/C2a x>a x||b  b+x=m. Thus,
sigma(L) = |Sub(L)|*2^(8-|L|) =   70.5000000000000000 .

L: LmT5/C2b x>a x||b  b+x=i
|L|=9, L={abcdemijx}. Edges:
am ac bm bd ce ci di ej ij mi ; C
 ax ; C2
  xi ; C2b
-- Constraints:
a+b=m c+m=i d+m=i e+m=j c*m=a d*m=b e*i=c ; C
 b+x=i ; C2b
-- Result:   |Sub(L)|=153 for the partial lattice
-- LmT5/C2b x>a x||b  b+x=i. Thus,
sigma(L) = |Sub(L)|*2^(8-|L|) =   76.5000000000000000 .

L: LmT5/C2c x>a x||b  b+x=j
|L|=9, L={abcdemijx}. Edges:
am ac bm bd ce ci di ej ij mi ; C
 ax ; C2
  xj ; C2c
-- Constraints:
a+b=m c+m=i d+m=i e+m=j c*m=a d*m=b e*i=c ; C
 b+x=j ; C2c
-- Result:   |Sub(L)|=144 for the partial lattice
-- LmT5/C2c x>a x||b  b+x=j. Thus,
sigma(L) = |Sub(L)|*2^(8-|L|) =   72.0000000000000000 .

L: LmT5/C2d x>a x||b  b+x=:y (new element)
|L|=10, L={abcdemijxy}. Edges:
am ac bm bd ce ci di ej ij mi ; C
 ax ; C2
  by xy ; C2d
-- Constraints:
a+b=m c+m=i d+m=i e+m=j c*m=a d*m=b e*i=c ; C
 b+x=y ; C2d
-- Result:   |Sub(L)|=309 for the partial lattice
-- LmT5/C2d x>a x||b  b+x=:y (new element). Thus,
sigma(L) = |Sub(L)|*2^(8-|L|) =   77.2500000000000000 .
 Also done: LmT5/C2 x>a x||b

L: LmT5/C3 y||j  y+j=:z
|L|=10, L={abcdemijyz}. Edges:
am ac bm bd ce ci di ej ij mi ; C
 yz jz ; C3
-- Constraints:
a+b=m c+m=i d+m=i e+m=j c*m=a d*m=b e*i=c ; C
 y+j=z ; C3
-- Result:   |Sub(L)|=294 for the partial lattice
-- LmT5/C3 y||j  y+j=:z. Thus,
sigma(L) = |Sub(L)|*2^(8-|L|) =   73.5000000000000000 .

L: LmT5/C4
|L|=9, L={abcdemijy}. Edges:
am ac bm bd ce ci di ej ij mi ; C
 my yj ; C4
-- Constraints:
a+b=m c+m=i d+m=i e+m=j c*m=a d*m=b e*i=c ; C
 c*y=a d*y=b e*y=a ; C4
-- Result:   |Sub(L)|=140 for the partial lattice
-- LmT5/C4. Thus,
sigma(L) = |Sub(L)|*2^(8-|L|) =   70.0000000000000000 .
 All cases have been excluded

The computation took 47/1000 seconds.
\end{verbatim}

\medskip
\centerline{{\LARGE$T_6$,\quad \texttt{LmT6-out.txt};\quad see Lemma~\ref{LmT6}}}
\medskip

\begin{verbatim}
 Version of August 14, 2019
L: LmT6/C
|L|=8, L={abcdeijm}. Edges:
ae am bd bm ci di ec ej ji mj
-- Constraints:
a+b=m  c+m=i d+m=i e+m=j  e*m=a d*m=b d*j=b c*j=e
-- Result:   |Sub(L)|=80 for the partial lattice
-- LmT6/C. Thus,
sigma(L) = |Sub(L)|*2^(8-|L|) =   80.0000000000000000 .

The computation took 0/1000 seconds.
\end{verbatim}

\medskip
\centerline{{\LARGE$T_7$,\quad \texttt{LmT7-out.txt};\quad see Lemma~\ref{LmT7}}}
\medskip

\begin{verbatim}
 Version of August 14, 2019
L: LmT7/C1 a||e a*e=:x b+x=m a+e=j
|L|=9, L={abcdeijmx}. Edges:
ac am bd bm ci di ej ij mi ; C
 xa xe ; C1
-- Constraints:
a+b=m c+m=i d+m=i e+m=j   c*m=a d*m=b c+e=j b+e=j ; C
 a*e=x b+x=m a+e=j ; C1
-- Result:   |Sub(L)|=159 for the partial lattice
-- LmT7/C1 a||e a*e=:x b+x=m a+e=j. Thus,
sigma(L) = |Sub(L)|*2^(8-|L|) =   79.5000000000000000 .

L: LmT7/C2a.1a a<e  c*e=a  e*i=a  y||j j+y=:z
|L|=10, L={abcdeijmyz}. Edges:
ac am bd bm ci di ej ij mi ; C
 ae ; C2
  jz yz ; C2a.1a
-- Constraints:
a+b=m c+m=i d+m=i e+m=j   c*m=a d*m=b c+e=j b+e=j ; C
 c*e=a ; C2a
  e*i=a ; C2a.1
   j+y=z; C2a.1a
-- Result:   |Sub(L)|=286 for the partial lattice
-- LmT7/C2a.1a a<e  c*e=a  e*i=a  y||j j+y=:z. Thus,
sigma(L) = |Sub(L)|*2^(8-|L|) =   71.5000000000000000 .

L: LmT7/C2a.1b.1 a<e  c*e=a  e*i=a  y<j  y||e e+y=j
|L|=9, L={abcdeijmy}. Edges:
ac am bd bm ci di ej ij mi ; C
 ae ; C2
  yj ; C2a.1b
-- Constraints:
a+b=m c+m=i d+m=i e+m=j   c*m=a d*m=b c+e=j b+e=j ; C
 c*e=a ; C2a
  e*i=a ; C2a.1
   e+y=j ; C2a.1b.1
-- Result:   |Sub(L)|=166 for the partial lattice
-- LmT7/C2a.1b.1 a<e  c*e=a  e*i=a  y<j  y||e e+y=j. Thus,
sigma(L) = |Sub(L)|*2^(8-|L|) =   83.0000000000000000 .

L: LmT7/C2a.1b.2a a<e  c*e=a  e*i=a  y<j  y<e  a<y
|L|=9, L={abcdeijmy}. Edges:
ac am bd bm ci di ej ij mi ; C
 ae ; C2
  yj ; C2a.1b
   ye ; C2a.1b.2
    ay ; C2a.1b.2a
-- Constraints:
a+b=m c+m=i d+m=i e+m=j   c*m=a d*m=b c+e=j b+e=j ; C
 c*e=a ; C2a
  e*i=a ; C2a.1
-- Result:   |Sub(L)|=140 for the partial lattice
-- LmT7/C2a.1b.2a a<e  c*e=a  e*i=a  y<j  y<e  a<y. Thus,
sigma(L) = |Sub(L)|*2^(8-|L|) =   70.0000000000000000 .

L: LmT7/C2a.1b.2b a<e  c*e=a  e*i=a  y<j  y<e  A<=y<a y+b=m
|L|=9, L={abcdeijmy}. Edges:
ac am bd bm ci di ej ij mi ; C
 ae ; C2
  yj ; C2a.1b
   ye ; C2a.1b.2
    ya ; C2a.1b.2b
-- Constraints:
a+b=m c+m=i d+m=i e+m=j   c*m=a d*m=b c+e=j b+e=j ; C
 c*e=a ; C2a
  e*i=a ; C2a.1
   y+b=m ; C2a.1b.2b
-- Result:   |Sub(L)|=151 for the partial lattice
-- LmT7/C2a.1b.2b a<e  c*e=a  e*i=a  y<j  y<e  A<=y<a y+b=m. Thus,
sigma(L) = |Sub(L)|*2^(8-|L|) =   75.5000000000000000 .

L: LmT7/C2a.1b.2c a<e  c*e=a  e*i=a  y<j  y<e  y||a a*y=:u u+b=m
|L|=10, L={abcdeijmyu}. Edges:
ac am bd bm ci di ej ij mi ; C
 ae ; C2
  yj ; C2a.1b
   ye ; C2a.1b.2
    ua uy um ; C2a.1b.2c
-- Constraints:
a+b=m c+m=i d+m=i e+m=j   c*m=a d*m=b c+e=j b+e=j ; C
 c*e=a ; C2a
  e*i=a ; C2a.1
   a*y=u u+b=m ; C2a.1b.2c
-- Result:   |Sub(L)|=231 for the partial lattice
-- LmT7/C2a.1b.2c a<e  c*e=a  e*i=a  y<j  y<e  y||a a*y=:u u+b=m. Thus,
sigma(L) = |Sub(L)|*2^(8-|L|) =   57.7500000000000000 .
 Also done: LmT7/C2a.1b.2 a<e, c*e=a, e*i=a, y<j, y<e
 Also done: LmT7/C2a.1b a<e, c*e=a, e*i=a, y<j
 Also done: LmT7/C2a.1 a<e, c*e=a, e*i=a

L: LmT7/C2a.2 a<e  c*e=a  e*i=:v>a
|L|=9, L={abcdeijmv}. Edges:
ac am bd bm ci di ej ij mi ; C
 ae ; C2
  ve vi av ; C2a.2
-- Constraints:
a+b=m c+m=i d+m=i e+m=j   c*m=a d*m=b c+e=j b+e=j ; C
 c*e=a ; C2a
  e*i=v ; C2a.2
-- Result:   |Sub(L)|=153 for the partial lattice
-- LmT7/C2a.2 a<e  c*e=a  e*i=:v>a. Thus,
sigma(L) = |Sub(L)|*2^(8-|L|) =   76.5000000000000000 .
 Also done: LmT7/C2a a<e, c*e=a

L: LmT7/C2b a<e  c*e=:x>a
|L|=9, L={abcdeijmx}. Edges:
ac am bd bm ci di ej ij mi ; C
 ae ; C2
  xc xe ax ; C2b
-- Constraints:
a+b=m c+m=i d+m=i e+m=j   c*m=a d*m=b c+e=j b+e=j ; C
 c*e=x ; C2b
-- Result:   |Sub(L)|=164 for the partial lattice
-- LmT7/C2b a<e  c*e=:x>a. Thus,
sigma(L) = |Sub(L)|*2^(8-|L|) =   82.0000000000000000 .
 Also done: LmT7/C2 a<e
 Also done: LmT7/C (all cases)

The computation took 55/1000 seconds.
\end{verbatim}

\medskip
\centerline{{\LARGE$T_8$,\quad \texttt{LmT8-out.txt};\quad see Lemma~\ref{LmT8}}}
\medskip

\begin{verbatim}
 Version of August 15, 2019
L: LmT8/C1 c||a c*a=:x>=A
|L|=9, L={abcdeijmx}. Edges:
ae am bd bm ci di ej ji mj ; C
 xc xa ; C1
-- Constraints:
a+b=m c+m=i d+m=i e+m=j e*m=a d*m=b d*j=b c+e=i c+b=i; C
 c*a=x ; C1
-- Result:   |Sub(L)|=166 for the partial lattice
-- LmT8/C1 c||a c*a=:x>=A. Thus,
sigma(L) = |Sub(L)|*2^(8-|L|) =   83.0000000000000000 .

L: LmT8/C2a a<c  c*j=a
|L|=8, L={abcdeijm}. Edges:
ae am bd bm ci di ej ji mj ; C
 ac ; C2
-- Constraints:
a+b=m c+m=i d+m=i e+m=j e*m=a d*m=b d*j=b c+e=i c+b=i; C
 c*j=a ; C2a
-- Result:   |Sub(L)|=77 for the partial lattice
-- LmT8/C2a a<c  c*j=a. Thus,
sigma(L) = |Sub(L)|*2^(8-|L|) =   77.0000000000000000 .

L: LmT8/C2b a<c  c*j=:x>a
|L|=9, L={abcdeijmx}. Edges:
ae am bd bm ci di ej ji mj ; C
 ac ; C2
  xc xj ax ; C2b
-- Constraints:
a+b=m c+m=i d+m=i e+m=j e*m=a d*m=b d*j=b c+e=i c+b=i; C
 c*j=x ; C2b
-- Result:   |Sub(L)|=163 for the partial lattice
-- LmT8/C2b a<c  c*j=:x>a. Thus,
sigma(L) = |Sub(L)|*2^(8-|L|) =   81.5000000000000000 .
 Also done: LmT8/C2 a<c
 Also done: LmT8/C (all cases)

The computation took 16/1000 seconds.
\end{verbatim}

\medskip
\centerline{{\LARGE$T_9$,\quad \texttt{LmT9-out.txt};\quad see Lemma~\ref{LmT9}}}
\medskip

\begin{verbatim}
 Version of August 15, 2019
L: LmT9/C1a.1 f+e=i  f*e=m  e*g=m
|L|=9, L={abcdefgim}. Edges:
ac am bd bm ci de ei fi mf me mg ; C
-- Constraints:
a+b=m c+m=i d+m=e c*m=a d*m=b  b+c=i a+d=e ; C
 f+e=i ; C1
  f*e=m ; C1a
   e*g=m ; C1a.1
-- Result:   |Sub(L)|=159 for the partial lattice
-- LmT9/C1a.1 f+e=i  f*e=m  e*g=m. Thus,
sigma(L) = |Sub(L)|*2^(8-|L|) =   79.5000000000000000 .

L: LmT9/C1a.2 f+e=i  f*e=m  e*g=:x>m
|L|=10, L={abcdefgimx}. Edges:
ac am bd bm ci de ei fi mf me mg ; C
 xe xg mx ; C1a.2
-- Constraints:
a+b=m c+m=i d+m=e c*m=a d*m=b  b+c=i a+d=e ; C
 f+e=i ; C1
  f*e=m ; C1a
   e*g=x ; C1a.2
-- Result:   |Sub(L)|=280 for the partial lattice
-- LmT9/C1a.2 f+e=i  f*e=m  e*g=:x>m. Thus,
sigma(L) = |Sub(L)|*2^(8-|L|) =   70.0000000000000000 .
 Also done: LmT9/C1a f+e=i, f*e=m

L: LmT9/C1b.1 f+e=i  f*e=:y>m  c*f=a
|L|=10, L={abcdefgimy}. Edges:
ac am bd bm ci de ei fi mf me mg ; C
 yf ye my ; C1b
-- Constraints:
a+b=m c+m=i d+m=e c*m=a d*m=b  b+c=i a+d=e ; C
 f+e=i ; C1
  f*e=y ; C1b
   c*f=a ; C1b.1
-- Result:   |Sub(L)|=318 for the partial lattice
-- LmT9/C1b.1 f+e=i  f*e=:y>m  c*f=a. Thus,
sigma(L) = |Sub(L)|*2^(8-|L|) =   79.5000000000000000 .

L: LmT9/C1b.2 f+e=i  f*e=:y>m  c*f=:z>a
|L|=11, L={abcdefgimyz}. Edges:
ac am bd bm ci de ei fi mf me mg ; C
 yf ye my ; C1b
  zc zf az ; C1b.2
-- Constraints:
a+b=m c+m=i d+m=e c*m=a d*m=b  b+c=i a+d=e ; C
 f+e=i ; C1
  f*e=y ; C1b
   c*f=z ; C1b.2
-- Result:   |Sub(L)|=596 for the partial lattice
-- LmT9/C1b.2 f+e=i  f*e=:y>m  c*f=:z>a. Thus,
sigma(L) = |Sub(L)|*2^(8-|L|) =   74.5000000000000000 .
 Also done: LmT9/C1b f+e=i, f*e=:y>m
 Also done: LmT9/C1 f+e=i

L: LmT9/C2a.1 f+e=:p<i  e*f=m  f*g=m
|L|=10, L={abcdefgimp}. Edges:
ac am bd bm ci de ei fi mf me mg ; C
 fp ep pi ; C2
-- Constraints:
a+b=m c+m=i d+m=e c*m=a d*m=b  b+c=i a+d=e ; C
 f+e=p ; C2
  e*f=m ; C2a
   f*g=m ; C2a.1
-- Result:   |Sub(L)|=327 for the partial lattice
-- LmT9/C2a.1 f+e=:p<i  e*f=m  f*g=m. Thus,
sigma(L) = |Sub(L)|*2^(8-|L|) =   81.7500000000000000 .

L: LmT9/C2a.2 f+e=:p<i  e*f=m  f*g=:q>m
|L|=11, L={abcdefgimpq}. Edges:
ac am bd bm ci de ei fi mf me mg ; C
 fp ep pi ; C2
  qf qg mq ; C2a.2
-- Constraints:
a+b=m c+m=i d+m=e c*m=a d*m=b  b+c=i a+d=e ; C
 f+e=p ; C2
  e*f=m ; C2a
   f*g=q ; C2a.2
-- Result:   |Sub(L)|=552 for the partial lattice
-- LmT9/C2a.2 f+e=:p<i  e*f=m  f*g=:q>m. Thus,
sigma(L) = |Sub(L)|*2^(8-|L|) =   69.0000000000000000 .
 Also done: LmT9/C2a f+e=:p<i, e*f=m

L: LmT9/C2b.1 f+e=:p<i  e*f=:x>m  c*f=a
|L|=11, L={abcdefgimpx}. Edges:
ac am bd bm ci de ei fi mf me mg ; C
 fp ep pi ; C2
  xe xf mx ; C2b
-- Constraints:
a+b=m c+m=i d+m=e c*m=a d*m=b  b+c=i a+d=e ; C
 f+e=p ; C2
  e*f=x ; C2b
   c*f=a ; C2b.1
-- Result:   |Sub(L)|=622 for the partial lattice
-- LmT9/C2b.1 f+e=:p<i  e*f=:x>m  c*f=a. Thus,
sigma(L) = |Sub(L)|*2^(8-|L|) =   77.7500000000000000 .

L: LmT9/C2b.2 f+e=:p<i  e*f=:x>m  c*f=:y>a
|L|=12, L={abcdefgimpxy}. Edges:
ac am bd bm ci de ei fi mf me mg ; C
 fp ep pi ; C2
  xe xf mx ; C2b
   yc yf ay ; C2b.2
-- Constraints:
a+b=m c+m=i d+m=e c*m=a d*m=b  b+c=i a+d=e ; C
 f+e=p ; C2
  e*f=x ; C2b
   c*f=y ; C2b.2
-- Result:   |Sub(L)|=1162 for the partial lattice
-- LmT9/C2b.2 f+e=:p<i  e*f=:x>m  c*f=:y>a. Thus,
sigma(L) = |Sub(L)|*2^(8-|L|) =   72.6250000000000000 .
 Also done: LmT9/C2b f+e=:p<i, e*f=:x>m
 Also done: LmT9/C2 f+e=:p<i
 Also done: LmT9/C (all cases)

The computation took 46/1000 seconds.
\end{verbatim}

\medskip
\centerline{{\LARGE$T_{10}$,\quad \texttt{LmT10-out.txt};\quad see Lemma~\ref{LmT10}}}
\medskip

\begin{verbatim}
 Version of August 15, 2019
L: LmT10/C
|L|=10, L={abcdefmkji}. Edges:
ae am bf bm ci di ej fk ji kj mk
-- Constraints:
a+b=m c+m=i d+m=i e+m=j m+f=k e*m=a f*m=b e*k=a b+c=i a+d=i
-- Result:   |Sub(L)|=289 for the partial lattice
-- LmT10/C. Thus,
sigma(L) = |Sub(L)|*2^(8-|L|) =   72.2500000000000000 .

L: LmT10/C with all the four edges
|L|=10, L={abcdefmkji}. Edges:
ae am bf bm ci di ej fk ji kj mk
 ec fd ac bd ; the four edges in addition
-- Constraints:
a+b=m c+m=i d+m=i e+m=j m+f=k e*m=a f*m=b e*k=a b+c=i a+d=i
-- Result:   |Sub(L)|=288 for the partial lattice
-- LmT10/C with all the four edges. Thus,
sigma(L) = |Sub(L)|*2^(8-|L|) =   72.0000000000000000 .

The computation took 15/1000 seconds.
\end{verbatim}

\end{document}